\documentclass[letterpaper,onecolumn, superscriptaddress,10pt,accepted=2025-05-16,issue=3, volume=7, shorttitle='Universal pseudomorphisms']{compositionalityarticle}
\pdfoutput=1

\newcommand{\authorsforheader}{Gurski and Johnson}
\newcommand{\paperdoi}{https://doi.org/10.46298/compositionality-7-3}
\newcommand{\receiveddate}{2023-12-19}

\usepackage[utf8]{inputenc}
\usepackage[english]{babel}

\usepackage{amsmath}
\usepackage{stackrel} 

\usepackage{graphicx}
\usepackage{xcolor}

\DeclareRobustCommand{\SkipTocEntry}[5]{}

\usepackage[calc]{datetime2}
\DTMnewdatestyle{ddmnameyyyy}{%
}
\DTMsetdatestyle{ddmnameyyyy}

\def\semicolon{;}
\def\applytolist#1{
  \expandafter\def\csname multi#1\endcsname##1{
    \def\multiack{##1}\ifx\multiack\semicolon
      \def\next{\relax}
    \else
      \csname #1\endcsname{##1}
      \def\next{\csname multi#1\endcsname}
    \fi
    \next}
  \csname multi#1\endcsname
}

\usepackage[mathscr]{eucal} 

\def\calc#1{\expandafter\def\csname c#1\endcsname{{\mathcal #1}}}
\applytolist{calc}ABCDEFGHIJKLMNOPQRSTUVWXYZ;

\usepackage{bbm}
\def\bbc#1{\expandafter\def\csname #1#1\endcsname{{\mathbbm #1}}}
\applytolist{bbc}ABCDEFGHIJKLMNOPQRSTUVWXYZ;

\DeclareMathAlphabet{\mathpzc}{OT1}{pzc}{m}{it}

\newcommand{\al}{\alpha}
\newcommand{\be}{\beta}
\newcommand{\ga}{\gamma}
\newcommand{\de}{\delta}
\newcommand{\epz}{\varepsilon}

\newcommand{\ka}{\kappa}
\newcommand{\la}{\lambda}

\newcommand{\si}{\sigma}
\newcommand{\ze}{\zeta}
\newcommand{\om}{\omega}
\newcommand{\ups}{\upsilon}
\newcommand{\Ga}{\Gamma}
\newcommand{\La}{\Lambda}
\newcommand{\De}{\Delta}
\newcommand{\Si}{\Sigma}
\newcommand{\Th}{\Theta}

\usepackage{enumitem}
\renewcommand{\theenumi}{\textit{\roman{enumi}}}
\renewcommand{\labelenumi}{\theenumi.}
\renewcommand{\theenumii}{\textit{\alph{enumii}}}
\renewcommand{\labelenumii}{\theenumii)}

\numberwithin{equation}{section} 
\numberwithin{figure}{section} 

\newcommand{\unloadamsthm}{ 
  \makeatletter
  \expandafter\let\csname ver@amsthm.sty\endcsname\relax
  \let\theoremstyle\relax
  \let\newtheoremstyle\relax
  \let\pushQED\relax
  \let\popQED\relax
  \let\qedhere\relax
  \let\mathqed\relax
  \let\openbox\relax
  \let\proof\relax\let\endproof\relax
  \makeatother
}
\unloadamsthm  
\newcommand{\unloadthmstyles}{
  \makeatletter
  \expandafter\let\csname th@plain\endcsname\relax
  \let\proofname\relax
  \makeatother
}
\unloadthmstyles

\usepackage[
 pdfusetitle,
 bookmarks=true,bookmarksnumbered=false,
 breaklinks=false,
 backref=false,
 colorlinks=true,
 linkcolor=blue!70!black,
 citecolor=black,
 urlcolor=blue!78!red,
 pagebackref,
 final
]{hyperref}

\renewcommand*{\backref}[1]{}
\renewcommand*{\backrefalt}[4]{%
  \ifcase #1 %
No citations.
  \or
(cit.\ on p.\ #2).%
  \else
(cit.\ on pp.\ #2).%
  \fi%
}
\usepackage[open]{bookmark} 

\usepackage{amsmath} 
\usepackage[hyperref,thmmarks,amsmath,amsthm]{ntheorem}

\theoremseparator{.}
\theorempreskip{\topsep} 
\theorempostskip{\topsep}
\makeatletter
\renewtheoremstyle{plain}%
  {\item[\hskip\labelsep \theorem@headerfont ##1\ ##2\theorem@separator]}%
  {\item[\hskip\labelsep \theorem@headerfont ##1\ ##2\ ({##3})\theorem@separator]}
\makeatother

\usepackage[capitalise]{cleveref}  
\makeatletter
\newcommand{\clevertheorem}[3]{%
\expandafter\@ifundefined{#1}%
    {\newtheorem{#1}[equation]{#2}}%
    {\renewtheorem{#1}[equation]{#2}}%
  \crefname{#1}{#2}{#3}
}
\makeatother

\newcommand{\crefrangeconjunction}{ through~}
\theoremstyle{plain} 
\newtheorem{thm}[equation]{Theorem}  
\crefname{thm}{Theorem}{Theorems}   
\clevertheorem{prop}{Proposition}{Propositions}
\clevertheorem{lem}{Lemma}{Lemmas}
\clevertheorem{cor}{Corollary}{Corollaries}
\clevertheorem{conj}{Conjecture}{Conjectures}

\theoremstyle{definition} 
\theoremsymbol{\ensuremath{\diamond}} 
\clevertheorem{defn}{Definition}{Definitions}
\clevertheorem{example}{Example}{Examples}
\clevertheorem{nonexample}{Non-Example}{Non-Examples}
\clevertheorem{exercise}{Exercise}{Exercises}
\clevertheorem{explanation}{Explanation}{Explanations}
\clevertheorem{notation}{Notation}{Notations}
\clevertheorem{notn}{Notation}{Notations}
\clevertheorem{convention}{Convention}{Conventions}
\clevertheorem{rmk}{Remark}{Remarks}

\crefname{equation}{}{} 

\crefformat{equation}{(#2#1#3)}
\crefmultiformat{equation}{(#2#1#3)}%
{ and~(#2#1#3)}{,~(#2#1#3)}{, and~(#2#1#3)}
\crefrangeformat{equation}{(#3#1#4)\crefrangeconjunction(#5#2#6)}

\crefformat{enumi}{(#2#1#3)}
\crefmultiformat{enumi}{(#2#1#3)}%
{ and~(#2#1#3)}{,~(#2#1#3)}{, and~(#2#1#3)}
\crefrangeformat{enumi}{(#3#1#4)\crefrangeconjunction(#5#2#6)}

\crefformat{enumii}{(#2#1#3)}
\crefmultiformat{enumii}{(#2#1#3)}%
{ and~(#2#1#3)}{,~(#2#1#3)}{, and~(#2#1#3)}
\crefrangeformat{enumii}{(#3#1#4)\crefrangeconjunction(#5#2#6)}

\newcommand\genatop[2]{\genfrac{}{}{0pt}{}{#1}{#2}}

\usepackage{amssymb}

\newcommand{\mh}{\mbox{-}} 
\newcommand{\cn}{\colon}
\newcommand{\bacn}{\mathrel{:}\mkern-.25\thinmuskip} 
\newcommand{\inv}{{-1}}
\newcommand{\wt}{\widetilde}
\newcommand{\wh}{\widehat}
\newcommand{\ol}{\overline}

\DeclareMathOperator{\Set}{\mathpzc{Set}}
\DeclareMathOperator{\Ab}{\mathpzc{Ab}}

\DeclareMathOperator{\Cat}{\mathpzc{Cat}}
\DeclareMathOperator{\cat}{\mathpzc{Cat}}

\newcommand{\hty}{\simeq}
\newcommand{\iso}{\cong}

\newcommand{\txprod}{{\textstyle\prod}}

\newcommand{\andspace}{\quad \text{and} \quad}
\newcommand{\orspace}{\quad \text{or} \quad}
\newcommand{\inspace}{\quad \text{in} \quad}
\newcommand{\forspace}{\quad \text{for} \quad}
\newcommand{\foreachspace}{\quad \text{for each} \quad}
\newcommand{\withspace}{\quad \text{with} \quad}
\newcommand{\sothatspace}{\quad \text{so that} \quad}

\usepackage{tikz}
\usepackage{tikz-cd}
\usepackage{pgfplots}
\pgfplotsset{compat=1.13}

\usetikzlibrary{calc}
\usetikzlibrary{arrows}
\usetikzlibrary{decorations.markings}

\tikzset{/tikz/commutative diagrams/arrow style=tikz,>=stealth} 
\tikzset{0cell/.style={commutative diagrams/every diagram,
    every cell,
    execute at begin node=$\displaystyle, execute at end node=$, 
    nodes={scale=#1},
  }
}
\pgfkeys{/tikz/0cell/.default=1}
\tikzset{1cell/.style={commutative diagrams/.cd,
    every arrow,
    every label,
    execute at begin node=$\displaystyle, execute at end node=$,
    font={\small},
    nodes={scale=#1},
  }
}
\pgfkeys{/tikz/1cell/.default=1}
\tikzset{xarrow/.style={
    scale=\labelscale, 
  }
}
\tikzset{2cell/.style={commutative diagrams/every diagram,
    every cell,
    execute at begin node=$\displaystyle, execute at end node=$, 
    font={\small},
    nodes={scale=#1},
  }
}
\pgfkeys{/tikz/2cell/.default=1}
\tikzset{equal/.style={commutative diagrams/equal}}
\tikzset{
  between/.style args={#1 and #2 at #3}{
    at = ($(#1)!#3!(#2)$)
  },
}
\tikzset{
  2label/.style args={#1,#2}{
    inner sep=0mm, label={[inner sep=0pt]#1:#2}
  },
}
\tikzset{vcenter/.style={baseline=(current bounding box.center)}}

\newlength{\arrowlen}
\newcommand{\arrbegin}{\hspace{-.1cm}\begin{tikzpicture}[baseline={(0,-0.1175)}]}
\newcommand{\arrend}{\end{tikzpicture}\hspace{-.1cm}}
\newcommand{\genarrow}[2]{
  \arrbegin
  \node[0cell] (A) at (0,0) {}; \node[0cell] (B) at (#1*\arrowlen,0) {}; \draw[1cell,#2] (A) to (B);
  \arrend
}
\renewcommand{\to}{\genarrow{1}{->}}

\renewcommand{\hookrightarrow}{\genarrow{1.3}{right hook->}}

\tikzset{mapto/.style={decoration={markings,mark=at position 0 with %
    {\arrow[yscale=.75]{|}}},postaction={decorate}}}
\renewcommand{\mapsto}{\genarrow{0.9}{mapto}}

\newlength{\arrowtoplablen}
\newlength{\arrowbotlablen}
\newcommand{\computearrowlength}[2][]{%
  \settowidth{\arrowtoplablen}{$#2$}%
  \settowidth{\arrowbotlablen}{$#1$}%
  \ifdim\arrowtoplablen<\arrowbotlablen
  \arrowtoplablen=\arrowbotlablen \fi%
}
\newcommand{\labelscale}{.85}
\newcommand{\xarrowpadding}{2.1em}
\renewcommand{\xrightarrow}[2][]{
  \computearrowlength[#1]{#2}
  \arrbegin
  \node[0cell] (A) at (0,0) {}; \node[0cell] (B) at (\labelscale*\arrowtoplablen+\xarrowpadding,0) {};
  \draw[1cell] (A) to node[scale=\labelscale]{#2}node[swap,scale=\labelscale]{#1} (B);
  \arrend
}
\newcommand{\fto}{\xrightarrow}

\newcommand{\lradj}{
  \arrbegin
  \node[0cell] (A) at (0,0) {}; \node[0cell] (B) at (1.3*\arrowlen,0) {};
  \draw[1cell,shift left] (A) to (B);
  \draw[1cell,shift left] (B) to (A);
  \arrend
}

\tikzset{smallhead/.style={decoration={markings,mark=at position 1 with %
    {\arrow[scale=.7,>=stealth]{>}}},shorten >=.5pt,postaction={decorate}}}
\tikzset{downar/.style={commutative diagrams/.cd,
    every arrow,
    every label,
    font={\small},
  }
}

\newcommand{\urladdrhref}[1]{\urladdr{\url{#1}}}

\newcommand{\authorGurski}{
\author{Nick Gurski}
\affiliation{Department of Mathematics, Applied Mathematics, and Statistics,
Case Western Reserve University}
\email{nick.gurski@case.edu}
\homepage{https://mathstats.case.edu/faculty/nick-gurski/}
\orcid{0000-0001-9691-1981}
}

\newcommand{\authorJohnson}{
\author{Niles Johnson}
\affiliation{Department of Mathematics,
The Ohio State University Newark}
\email{niles@math.osu.edu}
\homepage{https://nilesjohnson.net}
\orcid{0000-0002-4838-4651}
}

\title[Universal pseudomorphisms]{Universal pseudomorphisms, with applications to diagrammatic coherence for braided and symmetric monoidal functors}
\authorGurski
\authorJohnson

\hypersetup{bookmarksdepth=2}

\newcommand{\printkwds}{functor coherence, braided monoidal, symmetric monoidal, pseudomorphism coherence, pseudomorphism classifier}

\newcommand{\printmsc}{18C15 (Primary); 18D20, 18M05, 18M15, 18N15, 19D23 (Secondary)}

\makeatletter
\hypersetup{ 
  pdfinfo={
    MSC2020={\printmsc},
    Fulltitle={\@title}
  }
}
\makeatother

\usetikzlibrary{braids}
\tikzset{stdbraidstyle/.style = {
      braid/every strand/.style={thick},
      braid/height=-3mm,
      braid/width=5mm,
      braid/gap=.25, 
      braid/control factor=0.6, 
      braid/nudge factor=0.1, 
}}

\definecolor{brcola}{HTML}{EC4300} 
\definecolor{brcolb}{HTML}{4C1D9F} 

\pgfdeclaredecoration{midzzdecor}{initial}{
  \def\startlen{\pgfdecoratedpathlength/2-2*\pgfdecorationsegmentlength}
  \state{initial}[width=\startlen,next state=zz]{
    \pgfpathlineto{\pgfpoint{\startlen}{0pt}}
  }
  \state{zz}[width=\pgfdecorationsegmentlength,repeat state=2,next state=final]{
    \def\zzw{\pgfdecorationsegmentlength/4}
    \def\zza{\pgfdecorationsegmentamplitude}
    \pgfpathmoveto{\pgfpointorigin}
    \pgfpathlineto{\pgfpoint{1*\zzw}{\zza}}
    \pgfpathlineto{\pgfpoint{2*\zzw}{0pt}}
    \pgfpathlineto{\pgfpoint{3*\zzw}{-1*\zza}}
    \pgfpathlineto{\pgfpoint{4*\zzw}{0pt}}
  }
  \state{final}{
    \pgfpathlineto{\pgfpointdecoratedpathlast}
  }
}
\def\midzzSegLen{2pt}
\def\midzzTotLen{6pt} 
\def\midzzAmp{1.5pt}
\tikzset{zz1cell/.style={1cell, decorate, decoration={midzzdecor,segment length=\midzzSegLen,amplitude=\midzzAmp}}} 
\newcommand{\zzfto}[2][]{
  \computearrowlength[#1]{#2}
  \arrbegin
  \node[0cell] (A) at (0,0) {}; \node[0cell] (B) at (\labelscale*\arrowtoplablen+\xarrowpadding,0) {};
  \draw[zz1cell] (A) to node[scale=\labelscale]{#2}node[swap,scale=\labelscale]{#1} (B);
  \arrend
}
\newcommand{\zzto}{\zzfto{}}

\newcommand{\ang}[1]{\langle #1 \rangle}
\newcommand{\bang}[1]{\bigl\langle #1 \bigr\rangle}
\newcommand{\abs}[1]{|#1|}
\newcommand{\scs}{\;,\;} 
\newcommand{\sss}{\;;\;} 

\newcommand{\ob}{\mathsf{ob}}
\newcommand{\disc}{\mathsf{disc}}
\newcommand{\indisc}{\mathsf{indisc}}
\DeclareMathOperator{\Mon}{\mathpzc{Mon}}

\makeatletter
\newcommand*\bigcdot@[2]{\mathbin{\vcenter{\hbox{\scalebox{#2}{$\m@th#1\bullet$}}}}}
\newcommand*\bcdot{\mathpalette\bigcdot@{.5}}
\makeatother
\newcommand{\bcdots}{\bcdot\bcdot\bcdot}

\newcommand{\K}{\mathpzc{K}} 
\newcommand{\Kzero}{\K_{\;0}} 
\newcommand{\zA}{\mathpzc{A}} 
\newcommand{\zAzero}{\zA_{\,0}} 
\newcommand{\zB}{\mathpzc{B}} 
\newcommand{\zBzero}{\zB_{\,0}} 
\newcommand{\zL}{\mathpzc{L}} 
\newcommand{\zLzero}{\zL_{\,0}} 

\newcommand{\bii}{\mathbbm{2}}
\newcommand{\bfone}{\mathbf{1}}

\definecolor{usfcolor}{HTML}{981959} 
\newcommand{\usf}{{\color{usfcolor}\mathsf{u}}}
\definecolor{ittcolor}{HTML}{12622A} 
\newcommand{\itt}{{\color{ittcolor}\mathtt{i}}}

\newcommand{\ssf}{\mathsf{s}}
\newcommand{\qsf}{\mathsf{q}}

\newcommand{\ess}{\chi}

\newcommand{\T}{\mathsf{T}} 
\newcommand{\Talg}{\T\mh\operatorname{\mathsf{Alg}}}
\newcommand{\Talgz}{\Talg_0}
\newcommand{\Talgs}{\T\mh\operatorname{\mathsf{Alg_s}}}
\newcommand{\Talgsz}{\Talg_{s0}}
\newcommand{\Talgii}{\Talg^{\bii,s}}
\newcommand{\Xmul}{x}
\newcommand{\Ymul}{y}
\newcommand{\Wmul}{w}
\newcommand{\Zmul}{z}

\newcommand{\Q}{\mathsf{Q}}
\renewcommand{\P}{\mathsf{P}}

\newcommand{\M}{\mathsf{M}}
\renewcommand{\S}{\mathsf{S}} 
\newcommand{\B}{\mathsf{B}}
\newcommand{\Mg}{\M^{\mathsf{g}}}
\newcommand{\Sg}{\S^{\mathsf{g}}}
\newcommand{\Bg}{\B^{\mathsf{g}}}
\newcommand{\Tg}{\T^{\mathsf{g}}}
\newcommand{\Malg}{\M\mh\operatorname{\mathsf{Alg}}}

\newcommand{\Malgs}{\M\mh\operatorname{\mathsf{Alg_s}}}
\newcommand{\Salgs}{\S\mh\operatorname{\mathsf{Alg_s}}}

\newcommand{\pidt}{\pi^{\scriptscriptstyle \wt{D}}}

\theoremstyle{plain}
\clevertheorem{slogan}{Slogan}{Slogans}

\newcommand{\sgn}{\mathsf{sgn}}

\newcounter{artpart}
\renewcommand{\theartpart}{\Roman{artpart}}
\newcommand{\artpart}[1]{%
  \refstepcounter{artpart}%
  \addtocontents{toc}{\smallskip}%
   \addtocontents{toc}{
    \protect\contentsline{section}{
       \protect\numberline{}{\large Part \theartpart: #1}
     }{}{}
  }
  \section*{Part \theartpart: #1}}

\begin{document}

\begin{abstract}
  This work introduces a general theory of universal pseudomorphisms and develops their connection to diagrammatic coherence. The main results give hypotheses under which pseudomorphism coherence is equivalent to the coherence theory of strict algebras. Applications include diagrammatic coherence for plain, symmetric, and braided monoidal functors. The final sections include a variety of examples.
\end{abstract}

\maketitle
\tableofcontents 

\section{Introduction}

The main results of this paper are coherence theorems for pseudomorphisms between algebras over a 2-monad $\T$.
For example, $\T$ may be the 2-monad for plain, symmetric, or braided monoidal categories.
Coherence theorems for pseudomorphisms are, in these cases, coherence theorems for plain, symmetric, or braided strong monoidal functors.
Our main interest is what we call \emph{diagrammatic} coherence: general conditions that guarantee commutativity of (formal) diagrams.

\begin{example}\label{ex:main}
Consider the following diagram \cref{eq:mystery1-intro} for a braided monoidal functor
\[
  f\cn (A,+,\be) \to (A',\bcdot,\be),
\]
where $A$ and $A'$ are braided strict monoidal categories with braid isomorphisms $\be$ and monoidal products $+$ and $\bcdot$, respectively.
The two composites around the diagram apply different combinations of braidings $\beta$ and monoidal constraints $f_2$.
\begin{equation}\label{eq:mystery1-intro}
  \begin{tikzpicture}[x=25ex,y=8ex,vcenter,xscale=1.2]
    \draw[0cell] 
    (0,0) node (s1) {f(a) \bcdot f(a) \bcdot f(a)}
    (s1)+(1,0) node (s2) {f(a+a) \bcdot f(a)}
    (s2)++(1,0) node (s3) {f(a) \bcdot f(a+a)}
    (s3)++(0,-1) node (s4) {f(a+a+a)}
    (s1)++(0,-1) node (t2) {f(a+a+a)}
    (t2)++(1,0) node (t3) {f(a+a+a)}
    ;
    \draw[1cell] 
    (s1) edge node {f_2 \bcdot 1} (s2)
    (s2) edge node {\beta} (s3)
    (s3) edge node {f_2} (s4)
    (s1) edge['] node {f_2} (t2)
    (t2) edge node {f(1 + \beta)} (t3)
    (t3) edge node {f(\beta +1)} (s4)
    ;
  \end{tikzpicture}
\end{equation} 
This diagram satisfies a condition called \emph{formal} because it is determined entirely by the monoidal functor and braided monoidal category data of $f$, $A$, and $A'$.

Our diagrammatic coherence determines commutativity of formal diagrams like \cref{eq:mystery1-intro} by converting them to different---often simpler---formal diagrams that do not depend on the structure morphisms $f_2$.
The latter are called \emph{dissolution diagrams}, and the dissolution diagram for \cref{eq:mystery1-intro} is given as follows.
(See \cref{example:mystery1} for further explanation.)
\begin{center}
  \begin{tikzpicture}[x=28ex,y=10ex,vcenter,xscale=1.2]
    \draw[0cell] 
    (0,0) node (s1) {\Bigl(\; f(a) \scs f(a) \scs f(a) \;\Bigr)}
    (s1)+(.84,0) node (s2) {\Bigl(\; f(a) \scs f(a) \scs f(a) \;\Bigr)}
    (s1)++(2,0) node (s3) {\Bigl(\; f(a) \scs f(a) \scs f(a) \;\Bigr)}
    (s3)++(0,-1) node (s4) {\Bigl(\; f(a) \scs f(a) \scs f(a) \;\Bigr)}
    (s1)++(0,-1) node (t2) {\Bigl(\; f(a) \scs f(a) \scs f(a) \;\Bigr)}
    (t2)++(1,0) node (t3) {\Bigl(\; f(a) \scs f(a) \scs f(a) \;\Bigr)}
    ;
    \draw[1cell] 
    (s1) edge node {1} (s2)
    (s2) edge node {\beta_{(\,f(a),f(a)\,)\scs f(a)}} (s3)
    (s3) edge node {1} (s4)
    (s1) edge['] node {1} (t2)
    (t2) edge node {\bigl(\, 1 \scs \beta \,\bigr)} (t3)
    (t3) edge node {\bigl(\, \beta \scs 1 \,\bigr)} (s4)
    ;
    \node[between=s1 and s4 at .5, scale=.5] {
      \begin{tikzpicture}[baseline=(E.base)]
        \pic[
        stdbraidstyle,
        braid/every strand/.style={ultra thick},
        braid/width=7mm,
        braid/gap=.25, 
        braid/control factor=0.5, 
        braid/nudge factor=0.05, 
        braid/number of strands = 3,
        name prefix=braid0,
        braid/strand 1/.style={brcola},
        braid/strand 2/.style={brcola},
        braid/strand 3/.style={brcola},
        ] at (0,0)
        {braid={s_2 s_1}}
        ;
        \node (E) at (0,0) {}; 
      \end{tikzpicture}
    };
  \end{tikzpicture}
\end{center}
The composites around the above diagram have the same underlying braid, drawn in the center, and hence the diagram commutes in the free braided monoidal category on the object $f(a)$.
Our main result, \cref{thm:main1}, shows that the original diagram \cref{eq:mystery1-intro} therefore commutes in $A'$.
\end{example}

In the particular example above, one can also use naturality of $f_2$ along with axioms for $f$ and $\be$ to determine commutativity of \cref{eq:mystery1-intro} directly.
Indeed, every formal diagram for $f$ is amenable to such an approach.
The purpose of the diagrammatic coherence results in this work is to provide a general theory that eliminates the need to determine, for each diagram, \emph{which} combination of axioms is necessary.


\subsection*{General Overview}
Suppose $\K$ is a 2-category and $\T$ is a 2-monad on $\K$.
Under the respective hypotheses, our results reduce coherence for $\T$-algebra pseudomorphisms, such as $f$ in \cref{ex:main} above, to that of individual $\T$-algebras, such as the free algebra on $f(a)$ in \cref{ex:main}.
Indeed, the conclusions of \cref{thm:main1,thm:main2} are that coherence for $\T$-algebra pseudomorphisms is \emph{equivalent} to that of $\T$-algebras, in the following sense.

Assuming the hypotheses of \cref{thm:main1,thm:main2}, each 1-cell
\[
  \phi \cn C \to C' \inspace \K
\]
has an associated $\T$-algebra $\T(C',\phi)$ and a \emph{universal pseudomorphism}
\begin{equation}\label{eq:intro-wtphi}
  \wt{\phi}\cn \T C \to \T(C',\phi) 
\end{equation}
together with an equivalence of $\T$-algebras
\begin{equation}\label{eq:intro-De}
  \De \cn \T(C',\phi) \fto{\hty} \T C'.
\end{equation}
The universality of $\wt{\phi}$ and the construction of $\De$ are explained in \cref{sec:upc}.

In the case $\K = \Cat$, the 2-category of small categories, the universality of $\wt{\phi}$ gives a notion of \emph{formal diagrams} for a $\T$-algebra pseudomorphism $f$.
%
Then, the equivalence \cref{eq:intro-De} means that diagrams in $\T(C',\phi)$ commute if and only if the corresponding diagrams in $\T C'$ commute.
Thus, the universality of $\wt{\phi}$ and the equivalence $\De$ are the essential technical channels by which the coherence theory for pseudomorphisms reduces to that of $\T$-algebras.


\subsection*{Main applications}

We provide three statements of main results.
The first, \cref{thm:main1}, is the simplest.
It is formulated using overly-broad hypotheses that nevertheless hold in many applications of interest.
It follows as a special case of our third statement, \cref{thm:main2} below.
Recall that a 2-monad $\T$ is \emph{finitary} if it preserves all filtered colimits.
\begin{thm}[Finitary Pseudomorphism Coherence]\label{thm:main1}
  Suppose $\T$ is a finitary 2-monad on a 2-category $\K$ that is both complete and cocomplete.
  Then $\T$ admits universal pseudomorphisms
  \[
    \wt{\phi} \cn\T C \to \T(C',\phi) \forspace \phi\cn C \to C' \inspace \K
  \]
  such that, for each $\phi$, the induced strict morphism of $\T$-algebras \cref{eq:Delta}
  \begin{equation}\label{eq:De-intro}
    \De \cn \T(C',\phi) \to\T C'
  \end{equation}
  is a surjective equivalence in $\Talgs$ (\cref{defn:adj-surj-equiv}).
\end{thm}

Our second statement of main results, \cref{thm:main-msb}, is an application with $\K = \Cat$, the 2-category of small categories.
We explain some notation, terminology, and motivation before stating \cref{thm:main-msb}.
The hypotheses of \cref{thm:main1} hold when $\T$ is one of the three 2-monads $\{\Mg,\Sg,\Bg\}$ for plain, symmetric, or braided monoidal structure on categories (\cref{notn:monoidal-variants}).
In this notation, the superscript $\mathsf{g}$ indicates general monoidal structure, in contrast to the strictly associative and unital structure that we will later discuss.
In these cases we have the following:
\begin{itemize}
\item In the plain monoidal case $\T = \Mg$, an $\Mg$-algebra is a monoidal category and a pseudomorphism is a strong monoidal functor.
\item In the symmetric case $\T = \Sg$, an $\Sg$-algebra is a symmetric monoidal category and a pseudomorphism is a symmetric strong monoidal functor.
\item In the braided case $\T = \Bg$, a $\Bg$-algebra is a braided monoidal category and a pseudomorphism is a braided strong monoidal functor.
\end{itemize}

The statement of \cref{thm:main-msb} uses the following terms explained further in \cref{sec:formal-diagrams}.
\begin{itemize}
\item A \emph{diagram} in a $\T$-algebra $X'$ is a pair $(\DD,D)$ consisting of a small category $\DD$ and a functor
  \[
    \DD \fto{D} X'
  \]
  in $\Cat$.
\item A \emph{formal diagram for a pseudomorphism $f$} is a diagram that lifts through a canonical strict morphism of $\T$-algebras defined in \cref{eq:La}:
  \[
    \La \cn \T(\ob X',\phi) \to X',
  \]
  where $\phi = f_\ob$ denotes the restriction of $f$ to objects.
\item Each formal diagram $(\DD,D)$ for $f$ has a \emph{dissolution} diagram in the free algebra $\T(\ob X')$:
  \[
    \DD \fto{\abs{D}} \T(\ob X'),
  \]
  obtained by composing with $\De$ \cref{eq:De-intro}.
\end{itemize}
The dissolution diagram $\abs{D}$ is generally simpler that the original diagram $D$.
Indeed, for $\T \in \{\Mg,\Sg,\Bg\}$, \cref{defn:De-application}~\cref{defn:De-application-iv} shows that $\De$ sends monoidal and unit constraints of $f$ to \emph{identities} in $\T(\ob X')$.
\Cref{ex:main} from the beginning of this introduction shows a specific formal diagram \cref{eq:mystery1-intro} followed by its corresponding dissolution diagram.


In general, we have the following by \cref{thm:main1}.
\begin{thm}[Strong Monoidal Functor Coherence]\label{thm:main-msb}
  Suppose $\T$ is one of the three 2-monads $\{\Mg,\Sg,\Bg\}$ for plain, symmetric, or braided monoidal structure on $\K = \Cat$.
  Suppose given $\T$-algebras $X$ and $X'$, together with a $\T$-algebra pseudomorphism
  \[
    f \cn X \to X'
  \]
  and a diagram
  \[
    \DD \fto{D} X'.
  \]
  If $(\DD,D)$ is a formal diagram for $f$ such that the dissolution $\abs{D}$ commutes in $\T(\ob X')$, then the diagram $(\DD,D)$ commutes in $X'$.
\end{thm}

The assertions of \cref{thm:main-msb} may be summarized informally as follows.
\begin{slogan}\label{slogan:msb}
  In the cases $\T \in \{\Mg,\Sg,\Bg\}$, commutativity of formal diagrams for $f$ reduces to checking commutativity of the simpler dissolution diagrams, in which the monoidal and unit constraints of $f$ are replaced by identities.
\end{slogan}
The definitions of $\T(\ob X',\phi)$, $\La$, and $\De$ explain precisely how such a replacement of monoidal and unit constraints can be done.
We give a variety of detailed examples and further discussion in \cref{sec:other,sec:doub-quad}.
The interested reader is invited to skip forward for additional motivation, and then back to the relevant definitions and constructions as needed.

\subsection*{Main technical result}
Our third statement of main results, \cref{thm:main2}, is the most general and technical.
It identifies more precisely how the different features of our work rely on a collection of interrelated hypotheses.
In particular, \cref{thm:main2} states explicitly how the existence of universal pseudomorphisms $\wt{\phi}$ \cref{eq:intro-wtphi} relates to existence of a \emph{pseudomorphism classifier} $\Q$ for the 2-monad $\T$.
\Cref{sec:bg-psmor-class,sec:Qi-strictification} review those aspects of pseudomorphism classifiers that will be necessary in this work.

A pseudomorphism classifier can arise under various hypotheses, e.g., those discussed in \cite{BKP1989Two,power1989coherence,Lac02Codescent}.
One aim of our treatment is to explore the relationship between existence of a pseudomorphism classifier $\Q$, however it may arise, and existence of universal pseudomorphisms $\wt{\phi}$.

The proof of \cref{thm:main2} is included here.
It combines the essential results from the technical heart of this work, and serves as a high-level summary.
Here, we use the following notation.
\begin{itemize}
\item $\Talg$ and $\Talgs$ denote the 2-categories of $\T$-algebras with
  pseudomorphisms and strict morphisms, respectively, (\cref{defn:Talg-Talgs}).
\item $\bii$ and $\II$ denote the small categories consisting of two objects and a single nonidentity morphism, respectively isomorphism (\cref{notn:arrow}).
\item For each $C \in \K$, we write $\P C = \T(C,1_C)$ (\cref{defn:QTC-TC1}).
\end{itemize}
Further review of 2-monads, and of the limits and colimits necessary for this work, is given in \cref{sec:bg-2-monads,sec:bg-lim-colim}.
\begin{thm}[Pseudomorphism Coherence]\label{thm:main2}
  Suppose $\T$ is a 2-monad on a 2-category $\K$ and suppose that
  \begin{enumerate}
    \renewcommand{\theenumi}{\arabic{enumi}}
    \renewcommand{\labelenumi}{(\theenumi)}
  \item\label{it:hyp1} $\K$ admits pseudolimits of 1-cells;
  \item\label{it:hyp2} $\K$ admits cotensors
    \begin{enumerate}
      \renewcommand{\labelenumii}{(\theenumii)}
    \item\label{it:hyp2a} of the form $\{\bii, C\}$ for $C \in \K$ and 
    \item\label{it:hyp2b} of the form $\{\II, C\}$ for $C \in \K$;
    \end{enumerate}
  \item\label{it:hyp3} $\Talgs$ admits pushouts; and
  \item\label{it:hyp4} $\Talgs$ admits coequalizers of $\P$-free pairs (\cref{defn:P-free-fork}).
  \end{enumerate}
  Then the following two conditions are equivalent.
  \begin{enumerate}
    \renewcommand{\theenumi}{\Alph{enumi}}
    \renewcommand{\labelenumi}{(\theenumi)}
  \item\label{it:propA} $\T$ admits a pseudomorphism classifier $(\Q,\itt,\ze,\de)$.
  \item\label{it:propB} $\T$ admits universal pseudomorphisms $\wt{\phi}$.
  \end{enumerate}
  Moreover, in this case, the following hold for each $\T$-algebra $Y$ and each 1-cell $\phi \cn C \to C'$ in $\K$.
  \begin{enumerate}
    \renewcommand{\theenumi}{\Alph{enumi}}
    \renewcommand{\labelenumi}{(\theenumi)}
    \setcounter{enumi}{2}
  \item\label{it:propC} The components $\ze_Y$ and $\de_Y$ are part of an adjoint surjective equivalence.
  \item\label{it:propD} The induced strict morphism of $\T$-algebras \cref{eq:Delta} $\De\cn \T(C',\phi) \to \T C'$ is a surjective equivalence in $\Talgs$.
  \end{enumerate}
\end{thm}
\begin{proof} 
  \begin{description}
  \item[\cref{thm:BKP42} {\cite{BKP1989Two}}:]
    Suppose $\K$ satisfies \cref{it:hyp1,it:hyp2a}.
    Then \cref{it:propA} implies \cref{it:propC}.
  \item[\cref{thm:finitary-udc}:]
    Suppose $\Talgs$ satisfies \cref{it:hyp3}.
    Then \cref{it:propA,it:propC} together imply \cref{it:propB}, with $\T(C',\phi)$ constructed as a pushout \cref{eq:udc-pushout} in $\Talgs$.
  \item[\cref{thm:udc}:]
    Suppose $\K$ satisfies \cref{it:hyp2b}.
    Then \cref{it:propA,it:propB,,it:propC} together imply \cref{it:propD}.
  \item[\cref{rmk:alt-udc}:] Alternate proof that \cref{it:propA,it:propB,,it:propC} together imply \cref{it:propD}, under the assumption that $\T(C',\phi)$ is the pushout \cref{eq:udc-pushout} in $\Talgs$.
  \item[\cref{prop:UDC-Qi-adj}:]
    Suppose $\K$ satisfies \cref{it:hyp2a} and $\Talgs$ satisfies \cref{it:hyp4}.
    Then \cref{it:propB} implies \cref{it:propA}.
  \end{description}
\end{proof}

\subsection*{Relation to literature}

Our approach via universal pseudomorphisms in \cref{sec:upc} is based on the approach to coherence for monoidal functors in \cite[Theorem~1.7]{JS1993Braided} and for pseudofunctors between bicategories in \cite[Theorem~2.21]{Gurski13Coherence}.
Our use of pseudomorphism classifiers is motivated by their appearance in the 2-monadic approaches to coherence in \cite{BKP1989Two,power1989coherence,Lac02Codescent}.

It is important to note that this work focuses on pseudomorphism coherence rather than the more general \emph{lax} morphism coherence.
Certain special cases of the latter are treated in work of Epstein \cite{Eps1966Functors}, Lewis \cite{Lew1974Coherence}, and Malkiewich-Ponto \cite{MP2021Coherence}.
These coherence theorems focus on plain and symmetric monoidal structures, with Malkiewich-Ponto extending to bicategorical applications.
The following example due to Lewis illustrates the potential subtlety of lax morphisms.
\begin{nonexample}[{\cite[Pages~5--6]{Lew1974Coherence}}]\label{example:Lewis}
  Suppose given monoidal categories $A = (A,\bcdot,I)$ and $A' = (A',\bcdot, I')$ with monoidal products denoted $\bcdot$ and monoidal units denoted $I$ and $I'$, respectively.
  Suppose $f\cn A \to A'$ is a lax monoidal functor.
  The following diagram in $A'$ does not necessarily commute.
  \begin{equation}\label{eq:Lewis}
    \begin{tikzpicture}[x=22ex,y=10ex,vcenter]
      \draw[0cell] 
      (0,0) node (a) {f(I)}
      (a)++(1,0) node (b) {I' \bcdot f(I)}
      (a)++(0,-1) node (c) {f(I) \bcdot I'}
      (c)++(1,0) node (d) {f(I) \bcdot f(I)}
      ;
      \draw[1cell] 
      (a) edge node {\la^\inv} (b)
      (a) edge['] node {\rho^\inv} (c)
      (b) edge node {f_0 \bcdot 1} (d)
      (c) edge node {1 \bcdot f_0} (d)
      ;
    \end{tikzpicture}
  \end{equation}
  In the above diagram, $\la$ and $\rho$ are the left and right unit isomorphisms for $A'$, respectively, and $f_0$ is the monoidal unit constraint for $f$.

  For a specific case where \cref{eq:Lewis} does not commute, let $f$ be the forgetful functor from the category of abelian groups $A = (\Ab, \otimes, \ZZ)$, to the category of sets $A' = (\Set, \times, \bfone)$.
  This functor is lax monoidal, and the function $f_0 \cn \bfone \to \ZZ$ is given by sending the unique element of $\bfone$ to $1 \in \ZZ$.
  Then the two composites around the diagram are given by the functions $n \mapsto (1,n)$ for the top/right composite and $n \mapsto (n,1)$ for the left/bottom composite.
\end{nonexample}
Thus, the theory of coherence for lax monoidal functors is \emph{not} equivalent to that of monoidal categories, where every formal diagram commutes.

In contrast, our results show that often the coherence for $\T$-algebra pseudomorphisms is equivalent to that of strict $\T$-algebras.
Thus, the context for our work is restricted to pseudomorphisms, but broadened to a general 2-monad $\T$.
\Cref{rmk:no-lax} provides further details on a key step where our restriction to pseudomorphisms is required.

Our results are related to, but somewhat different from, coherence theorems for pseudo\emph{algebras} such as those of Power \cite{power1989coherence}, Hermida \cite{Her2001Coherent}, and Lack \cite{Lac02Codescent}.
The latter are formulated to show that there is a left adjoint to the inclusion
\[
  \Talgs \to \mathsf{Ps}\mh\Talg,
\]
such that the components of the unit are equivalences in $\mathsf{Ps}\mh\Talg$.
Here, $\mathsf{Ps}\mh\Talg$ is the 2-category of pseudoalgebras and pseudomorphisms for $\T$.
Such coherence results show that pseudoalgebras and pseudomorphisms for $\T$ can be replaced with equivalent strict algebras and strict morphisms.
They do not directly address the diagrammatic coherence questions that are resolved by \cref{thm:main-msb} for pseudomorphisms.

\subsection*{Outline}

This work is organized into three parts.
Part I consists of \cref{sec:bg-2-monads,sec:bg-lim-colim,sec:bg-psmor-class,sec:Qi-strictification} and reviews relevant parts of 2-monad theory.
\Cref{sec:bg-2-monads,sec:bg-lim-colim} recall basic definitions, limits, and colimits.
\Cref{sec:bg-psmor-class,sec:Qi-strictification} recall essential parts of the theory of pseudomorphism classifiers.

Part II consists of
\cref{sec:upc,sec:finitary-udc,sec:De1-equiv,sec:udc-implies-Qi} and contains the core technical work.
The definition of universal pseudomorphisms $\wt{\phi}$ and their basic properties are given in \cref{sec:upc}.
\Cref{sec:finitary-udc} gives a construction of $\T(C',\phi)$ as a pushout of a pseudomorphism classifier $\Q$, in the case that $\Talgs$ admits pushouts.
\Cref{sec:De1-equiv} proves that $\De$ is an equivalence in each of two separate results with slightly different hypotheses.
\Cref{sec:udc-implies-Qi} identifies hypotheses under which the existence of universal pseudomorphisms $\wt{\phi}$ implies the existence of a pseudomorphism classifier $\Q$.

Part III contains applications to diagrammatic coherence for 2-monads over $\Cat$.
\Cref{sec:formal-diagrams} gives a general definition of formal diagrams for such 2-monads $\T$, and the remaining sections focus on three special cases for plain, symmetric, and braided monoidal structures.
\Cref{sec:Mv-and-Tv} recalls the relevant definitions and the standard coherence theorems in those cases.
\Cref{sec:symm-coh} contains a novel simplification in the symmetric monoidal case.
\cref{sec:psmorclass,sec:applications} give detailed explanations of the abstract constructions from Part II for plain, symmetric, and braided monoidal structures.

\cref{sec:other} contains a number of examples that apply the results above to check commutativity of various diagrams for symmetric and braided strong monoidal functors.
\cref{sec:doub-quad} treats two specific monoidal functors and a diagram \cref{eq:cursed-cyclic} that is \emph{not} simplified by the dissolution approach developed in this work.
Both \cref{sec:other,sec:doub-quad} have been written to minimize explicit dependence on the preceding theory, and to be read as independently as possible.
Some readers may find it interesting to read those sections immediately after this introduction, and then follow the references from there back to the main body as necessary.

\artpart{Background}

\section{2-monads}
\label{sec:bg-2-monads}

For basic theory of categories and 2-categories, we refer the reader to
\cite{ML98Categories,Lac10Companion,Gurski13Coherence,JY212Dim}.
\begin{convention}\label{convention:2cats}
  Throughout this work, we let $\K$ denote a 2-category.
  We denote 1-cells as
  \[
    \phi\cn C \to C' \orspace \psi\cn D \to D'.
  \]
  We use a \emph{relative dimension convention} and denote 2-cells as
  \[
    \Ga\cn \phi \to \phi'
    \orspace
    \begin{tikzpicture}[x=12ex,y=10ex,vcenter]
      \draw[0cell] 
      (0,0) node (c) {C}
      (1,0) node (c') {C'.}
      ;
      \draw[1cell] 
      (c) edge[bend left] node (P) {\phi} (c')
      (c) edge[',bend right] node (P') {\phi'} (c')
      ;
      \draw[2cell] 
      node[between=P and P' at .5, rotate=270, 2label={below,\Ga}] {\Rightarrow}
      ;
    \end{tikzpicture}
  \]
\end{convention}

\begin{defn}\label{defn:2-monad}
  Suppose $\K$ is a 2-category.
  A \emph{2-monad} on $\K$ is a triple $(\T,\mu,\eta)$ consisting of
  \begin{itemize}
  \item a 2-functor $\T \cn \K \to \K$,
  \item a 2-natural transformation $\mu\cn \T^2 \to \T$, and
  \item a 2-natural transformation $\eta\cn 1_\K \to \T$.
  \end{itemize}
  These data are required to make the following unity and associativity diagrams commute.
  \[
    \begin{tikzpicture}[x=12ex,y=7ex,xscale=1.2]
      \draw[0cell] 
      (0,0) node (a) {\T^3}
      (a)++(1,0) node (b) {\T^2}
      (a)++(0,-1) node (c) {\T^2}
      (b)++(0,-1) node (d) {\T}
      ;
      \draw[1cell] 
      (a) edge node {1_\T * \mu} (b)
      (c) edge node {\mu} (d)
      (a) edge['] node {\mu * 1_\T} (c)
      (b) edge node {\mu} (d)
      ;
    \end{tikzpicture}
    \qquad
    \begin{tikzpicture}[x=12ex,y=7ex,xscale=1.2]
      \draw[0cell] 
      (0,0) node (a) {1_\K \T}
      (a)++(1,0) node (b) {\T^2}
      (b)++(1,0) node (z) {\T 1_\K}
      (a)++(0,-1) node (c) {\T}
      (b)++(0,-1) node (d) {\T}
      (z)++(0,-1) node (w) {\T}
      ;
      \draw[1cell] 
      (a) edge node {\eta * 1_\T} (b)
      (z) edge['] node {1_\T * \eta} (b)
      (c) edge[equal] node {} (d)
      (a) edge[',equal] node {} (c)
      (b) edge node {\mu} (d)
      (z) edge[equal] node {} (w)
      (w) edge[',equal] node {} (d)
      ;
    \end{tikzpicture}
  \]
 We often write a 2-monad as $\T$, leaving $\mu, \eta$ implicit.
\end{defn}

\begin{defn}\label{defn:t-alg}
  Suppose $\T$ is a 2-monad on $\K$.
  A \emph{$\T$-algebra} is a pair $(X,x)$ consisting of
  \begin{itemize}
  \item an object $X \in \K$ and
  \item a structure 1-cell $x\cn \T X \to X$ in $\K$
  \end{itemize}
  such that the following two diagrams commute.
  \begin{equation}\label{zeta-theta-def}
  \begin{tikzpicture}[x=12ex,y=8ex,vcenter,xscale=1.2]
    \draw[0cell] 
    (0,0) node (x) {X}
    (0,-1) node (tx) {\T X}
    (1,-1) node (x') {X}
    ;
    \draw[1cell] 
    (x) edge[swap] node {\eta_X} (tx)
    (x) edge node (id) {1_X} (x')
    (tx) edge[swap] node {\Xmul} (x')
    ;
  \end{tikzpicture}
  \qquad\qquad
  \begin{tikzpicture}[x=12ex,y=8ex,vcenter,xscale=1.2]
    \draw[0cell] 
    (0,0) node (t2x) {\T^2 X}
    (1,0) node (tx) {\T X}
    (0,-1) node (tx') {\T X}
    (1,-1) node (x) {X}
    ;
    \draw[1cell] 
    (t2x) edge[swap] node {\T \Xmul} (tx')
    (t2x) edge node {\mu_X} (tx)
    (tx) edge node (id) {\Xmul} (x)
    (tx') edge[swap] node {\Xmul} (x)
    ;
  \end{tikzpicture} 
  \end{equation}
  \ 
\end{defn}

\begin{defn}\label{defn:t-maps}
  Suppose $(X,x)$ and $(Y,y)$ are $\T$-algebras for a 2-monad $\T$ on $\K$.
  A \emph{$\T$-algebra pseudomorphism}, or \emph{$\T$-map}, is a pair
  \[
    (f,f_{\bullet}) \cn (X,x) \to (Y,y)
  \]
   consisting of
  \begin{itemize}
  \item a 1-cell $f \cn X \to Y$ in $\K$ called the \emph{underlying 1-cell} and
  \item an invertible 2-cell $f_{\bullet}$ in $\K$ as shown below, called the \emph{algebra constraint} of $f$.
    \begin{equation}\label{eq:Tmap-2cell}
      \begin{tikzpicture}[x=11ex,y=9ex,vcenter,xscale=1.2]
        \draw[0cell] 
        (0,0) node (tx) {\T X}
        (1,0) node (tx') {\T Y}
        (0,-1) node (x) {X}
        (1,-1) node (x') {Y}
        ;
        \draw[1cell] 
        (tx) edge node {\T f} (tx')
        (tx) edge['] node {\Xmul} (x)
        (tx') edge node {\Ymul} (x')
        (x) to['] node {f} (x')
        ;
        \draw[2cell] 
        node[between=x and tx' at {.5}, rotate=225, 2label={below,f_\bullet}, 2label={above,\iso}] {\Rightarrow}
        ;
      \end{tikzpicture}
    \end{equation} 
  \end{itemize}

  These data are required to satisfy unit and multiplication axioms indicated by the two equalities of pasting diagrams below.
  In these diagrams, the unlabeled regions commute because $X$ and $Y$ are assumed to be $\T$-algebras.
  \begin{equation}\label{lax-morphism-unit-axiom}
    \begin{tikzpicture}[x=12ex,y=10ex,baseline=(tx.base),xscale=1.2]
      \newcommand{\boundary}{
        \draw[0cell] 
        (0,0) node (tx) {\T X}
        (0,-1) node (x) {X}
        (1,-1) node (x') {Y}
        (tx) ++(45:1) node (a) {X}
        (a) ++(1,0) node (b) {Y}
        ;
        \draw[1cell] 
        (tx) edge['] node {\Xmul} (x)
        (a) edge['] node (etax) {\eta_X} (tx)
        (b) edge[bend left] node (1x') {1_{Y}} (x') 
        (x) to['] node {f} (x')
        (a) to node {f} (b)
        ;
      }
      \draw[font=\Large] (2.25,0) node {=};
      \begin{scope}[shift={(0,0)}]\
        \boundary
        \draw[0cell] 
        (1,0) node (tx') {\T Y}
        ;
        \draw[1cell] 
        (tx) edge node {\T f} (tx')
        (tx') edge node {\Ymul} (x')
        (b) edge['] node {\eta_{Y}} (tx')
        ;
        \draw[2cell] 
        node[between=x and tx' at {.5}, rotate=225, 2label={below,f_\bullet}] {\Rightarrow}
        ;          
      \end{scope}
      \begin{scope}[shift={(3,0)}]
        \boundary
        \draw[0cell] 
        ;
        \draw[1cell] 
        (a) edge[bend left] node (1x) {1_{X}} (x) 
        ;
        \draw[2cell] 
        ;
      \end{scope}
    \end{tikzpicture}
  \end{equation}

  \begin{equation}\label{lax-morphism-structure-axiom}
    \begin{tikzpicture}[x=.4ex,y=.4ex,xscale=1.2]
      \def\boundary{
        \draw[0cell] 
        (0,0) node (LL) {\T^2X}
        (20,20) node (TL) {\T X}
        (20,-20) node (BL) {\T^2Y}
        (50,20) node (TR)  {X}
        (50,-20) node (BR) {\T Y}
        (70,0) node (RR)  {Y}
        ;
        \draw[1cell] 
        (LL) edge node {\mu_{X}} (TL)
        (TL) edge node {\Xmul} (TR)
        (TR) edge node {f} (RR)
        (LL) edge[swap] node {\T^2 f} (BL)
        (BL) edge[swap] node {\T \Ymul} (BR)
        (BR) edge[swap] node {\Ymul} (RR)
        ;
      }
      
      \begin{scope}
        \boundary
        \draw[0cell]
        (30,0) node (CC)  {\T X}
        ;
        \draw[1cell] 
        (LL) edge node {\T \Xmul} (CC)
        (CC) edge node {\Xmul} (TR)
        (CC) edge node {\T f} (BR)
        ;
        \draw[2cell] 
        node[between=BL and CC at .5, rotate=60, 2label={below,\T f_\bullet}] {\Rightarrow}
        node[between=CC and RR at .55, rotate=90, 2label={below,f_\bullet}] {\Rightarrow}
        ;
        \draw (82.5,2.5) node[font=\LARGE] {=};
      \end{scope}

      \begin{scope}[shift={(95,0)}]
        \boundary
        \draw[0cell]
        (40,0) node (CC)  {\T Y}
        ;
        \draw[1cell] 
        (TL) edge[swap] node {\T f} (CC)
        (BL) edge node {\mu_{Y}} (CC)
        (CC) edge node {\Ymul} (RR)
        ;
        \draw[2cell] 
        node[between=CC and TR at .5, rotate=60, 2label={above,f_\bullet}] {\Rightarrow}
        ;
      \end{scope}
    \end{tikzpicture}
  \end{equation}
  We often abbreviate the pair $(f,f_\bullet)$ as $f$.
  We say that $f$ is a \emph{strict $\T$-map} if $f_\bullet$ is an identity 2-cell, so that \cref{eq:Tmap-2cell} commutes.
  We will sometimes say ``map'' or ``strict map'' when $\T$ is clear from context.
\end{defn}
\begin{rmk}\label{rmk:strictTmaps-underlying}
  In the context of \cref{defn:t-maps}, let $\K_0$ denote the underlying 1-category of $\K$ and let $\T_0$ denote the monad on $\K_0$ underlying $\T$.
  Suppose that $(X,\Xmul)$ and $(Y,\Ymul)$ are $\T$-algebras.
  Then a 1-cell $f\cn X \to Y$ in $\K$ is a strict $\T$-map if and only if $f$ is a morphism of $\T_0$-algebras.
\end{rmk}

\begin{rmk}\label{rmk:Tmaps}
  Our terms ``$\T$-map'', respectively ``strict $\T$-map,'' are convenient abbreviations for what are called
  \emph{pseudo} or \emph{strong $\T$-morphism}, respectively \emph{strict $\T$-morphism},
  in the literature.
  The more general notion of \emph{lax $\T$-morphism}, where $f_\bullet$ is not assumed to be invertible, will not be used in this present work.
\end{rmk}

\begin{defn}\label{defn:T2cell}
  Suppose $(f,f_\bullet)$ and $(g,g_\bullet)$ are two $\T$-maps $(X,\Xmul) \to (Y,\Ymul)$ for $\T$-algebras $X$ and $Y$ in $\K$.
  A \emph{$\T$-algebra 2-cell}
  \[
    \al\cn f \to g
  \]
  is a 2-cell $\al\cn f \to g$ in $\K$ such that the following equality holds.
  \[
    \begin{tikzpicture}[x=12ex,y=10ex,xscale=1.2]
      \newcommand\boundary{
        \draw[0cell] 
        (0,0) node (tx) {\T X}
        (1,0) node (tx') {\T Y}
        (0,-1) node (x) {X}
        (1,-1) node (x') {Y}
        ;
        \draw[1cell] 
        (tx) edge[bend left] node (Tf1) {\T f} (tx')
        (x) edge[',bend right] node (f2) {g} (x')
        (tx) edge['] node {\Xmul} (x)
        (tx') edge node {\Ymul} (x')
        ;
      }
      \draw[font=\Large] (1.75,-.4) node {=};
      \begin{scope}[shift={(0,0)}]
        \boundary
        \draw[1cell] 
        (x) edge[bend left] node {f} (x')
        ;
        \draw[2cell] 
        node[between=Tf1 and f2 at {.4}, rotate=225, 2label={below,f_\bullet}] {\Rightarrow}
        node[between=x and x' at {.46}, rotate=-90, 2label={above,\,\alpha}] {\Rightarrow}
        ;
      \end{scope}
      \begin{scope}[shift={(2.5,0)}]
        \boundary
        \draw[1cell] 
        (tx) edge[',bend right] node {\T g} (tx')
        ;
        \draw[2cell] 
        node[between=Tf1 and f2 at {.66}, rotate=225, 2label={below,g_\bullet}] {\Rightarrow}
        node[between=tx and tx' at {.44}, rotate=-90, 2label={above,\T \alpha}] {\Rightarrow}
        ;
      \end{scope}
    \end{tikzpicture}
  \]
  We will also say that $\al$ is an \emph{algebra 2-cell} when $\T$ is clear from context.
\end{defn}

\begin{defn}\label{defn:Tmap-comp}
  The composite of $\T$-maps
  \[
    X \fto{f} X' \fto{f'} X''
  \]
  is defined as follows.
  \begin{itemize}
  \item The underlying 1-cell of $f' \circ f$ is the composite of underlying 1-cells.
  \item The algebra constraint $(f' \circ f)_\bullet$ is given by the pasting in $\K$ indicated below.
    \begin{equation}\label{eq:Tmap-comp-2cell}
      \begin{tikzpicture}[x=11ex,y=9ex,vcenter,xscale=1.2]
        \draw[0cell] 
        (0,0) node (tx) {\T X}
        (1,0) node (tx') {\T Y}
        (2,0) node (tx'') {\T Z}
        (0,-1) node (x) {X}
        (1,-1) node (x') {Y}
        (2,-1) node (x'') {Z}
        ;
        \draw[1cell] 
        (tx) edge node {\T f} (tx')
        (tx') edge node {\T f'} (tx'')
        (tx) edge['] node {\Xmul} (x)
        (tx') edge node {\Ymul} (x')
        (tx'') edge node {\Zmul} (x'')
        (x) to['] node {f} (x')
        (x') to['] node {f'} (x'')
        ;
        \draw[2cell] 
        node[between=x and tx' at {.5}, rotate=225, 2label={below,f_\bullet}, 2label={above,\iso}] {\Rightarrow}
        node[between=x' and tx'' at {.5}, rotate=225, 2label={below,f'_\bullet}, 2label={above,\iso}] {\Rightarrow}
        ;
      \end{tikzpicture}
    \end{equation} 
    That is, 
    \begin{equation}\label{eq:Tmap-comp-formula}
      (f' \circ f)_\bullet = (f' * f_\bullet) \circ (f'_\bullet * \T f).
    \end{equation}
  \end{itemize}
  Horizontal and vertical composition of algebra 2-cells is given by that of the underlying 2-category, $\K$.
\end{defn}

\begin{defn}\label{defn:Talg-Talgs}
  Suppose $\K$ is a 2-category and $\T$ is a 2-monad on $\K$.
  We use the notations
  \[
    \Talg \andspace \Talgs
  \]
  to denote the 2-categories consisting of
  \begin{itemize}
  \item $\T$-algebras as 0-cells,
  \item $\T$-maps, respectively strict $\T$-maps, as 1-cells, and
  \item $\T$-algebra 2-cells as 2-cells.
  \end{itemize}
  Because every strict $\T$-map is a $\T$-map with identity algebra constraints,
  there is an identity-on-objects, locally full and faithful inclusion denoted
  \begin{equation}\label{eq:incl-i}
    \itt \cn \Talgs \hookrightarrow \Talg.  
  \end{equation}
  Moreover, each $\T$-algebra, $\T$-map, or $\T$-algebra 2-cell has an underlying object, 1-cell, or 2-cell in $\K$, respectively.
  We let $\usf$ denote the forgetful 2-functors as indicated in the following diagram, with $\usf = \usf \circ \itt$.
  \begin{equation}\label{eq:iuu}
    \begin{tikzpicture}[x=8ex,y=8ex,vcenter,xscale=1.2]
      \draw[0cell] 
      (0,0) node (k) {\K}
      (k)++(-1,1) node (tas) {\Talgs}
      (k)++(1,1) node (ta) {\Talg}
      ;
      \draw[1cell] 
      (tas) edge[right hook->] node {\itt} (ta)
      (ta) edge node {\usf} (k)
      (tas) edge['] node {\usf} (k)
      ;
    \end{tikzpicture}
  \end{equation} 
  \ 
\end{defn}

\begin{convention}\label{convention:ignorei}
  The 2-functor $\itt \cn \Talgs \hookrightarrow \Talg$ \cref{eq:incl-i} is the identity on objects, 1-cells, and 2-cells. 
  Therefore, we will sometimes leave $\itt$ implicit and omit it from the notation.
  For example, any time that a strict $\T$-map is composed with a general $\T$-map, there may be an implicit usage of $\itt$.
\end{convention}

\begin{defn}\label{defn:Tu-adj}
  In the context of \cref{defn:Talg-Talgs}, we use the notations
  \begin{equation}\label{eq:Tu-adj}
    \begin{tikzpicture}[x=15ex,y=8ex,vcenter]
      \draw[0cell] 
      (0,0) node (x) {\K}
      (1,0) node (y) {\Talgs}
      ;
      \draw[1cell] 
      (x) edge[bend left,transform canvas={yshift=.7mm},out=21,in=167] node (L) {\T} (y) 
      (y) edge[bend left,transform canvas={yshift=-.7mm},in=160,out=13] node (R) {\usf} (x) 
      ;
      \draw[2cell] 
      node[between=L and R at .5] {\bot}
      ;
    \end{tikzpicture}
  \end{equation}
  for the free-forgetful 2-adjunction with left 2-adjoint $\T$ and right 2-adjoint $\usf$.
  We let $\eta$ and $\epz$ denote, respectively, the unit and counit of $\T \dashv \usf$.
  For each $\T$-algebra $(X,\Xmul)$,
  \begin{itemize}
  \item the unit component $\eta_X$ is the unit of the $\T$-algebra structure on $X$ and
  \item the counit component $\epz_X$ is the algebra structure cell $\Xmul\cn \T X \to X$.
  \end{itemize}
\end{defn}

\begin{convention}\label{rmk:zzto}
  Beginning here, and throughout the rest of this document, we will write
  \[
    f \cn X \zzto Y,
  \]
  using a zigzag arrow, to denote that $f$ is the 1-cell part of a $\T$-map $(f,f_\bullet)$.
  If $f_\bullet$ is known to be an identity, so that $f$ is a strict $\T$-map, we use a straight arrow and write
  \[
    f \cn X \to Y.
  \]
\end{convention}

\begin{rmk}[Uniqueness of mates]\label{rmk:LRadj-uniqueness}
  The following elementary detail about 2-adjunctions will be useful below.
  Suppose given a 2-adjunction 
  \[
    \begin{tikzpicture}[x=15ex,y=8ex,vcenter]
      \draw[0cell] 
      (0,0) node (x) {\K}
      (1,0) node (y) {\zA}
      ;
      \draw[1cell] 
      (x) edge[bend left=19] node (L) {L} (y) 
      (y) edge[bend left=19] node (R) {R} (x) 
      ;
      \draw[2cell] 
      node[between=L and R at .5] {\bot}
      ;
    \end{tikzpicture}
  \]
  between 2-categories $\K$ and $\zA$, with unit $\eta$ and counit $\epz$.
  For objects $C \in \K$ and $Y \in \zA$, the isomorphism of categories
  \begin{equation}\label{eq:LRadj-uniqueness}
    \zA(LC,Y) \fto{\iso} \K(C,RY)
  \end{equation}
  is given by the right adjoint $R$ and composition or whiskering with $\eta$:
  \begin{equation}\label{eq:eta-whisk}
    \begin{aligned}
      f & \mapsto Rf \circ \eta\\
      \al & \mapsto R\al * \eta
    \end{aligned}
  \end{equation}
  where $\al \cn f \to f'$ in $\zA(LC,Y)$.
  In particular, if $f$ and $g$ are two 1-cells in $\zA(LC,Y)$ such that $Rf \circ \eta = Rg \circ \eta$ as 1-cells in $\K$, then $f$ and $g$ are equal as 1-cells in $\K$.
\end{rmk}

\section{Cotensors and coequalizers}
\label{sec:bg-lim-colim}

Completeness and cocompleteness for 2-categories generally refers to the $\cat$-enriched sense, meaning not just conical limits and colimits but also including all small $\cat$-weighted limits and colimits.
The only non-conical such we will employ, in \cref{sec:upc,sec:De1-equiv}, is that of a cotensor (also called a power).
Below, we recall their defining property and a key application.
For the more general theory of 2-dimensional limits and colimits, we refer the reader to \cite{Kel89Elementary,borceux2}.

Later in this section we discuss various coequalizers and their relation to $\T$-algebra structures.
These will be used in \cref{sec:udc-implies-Qi}.

\begin{defn}\label{defn:cotensor}
Suppose $\K$ is a 2-category, $X$ is an object of $\K$, and $C$ is a small category.
The \emph{cotensor} of $C$ and $X$ is an object of $\K$, denoted $\{C, X\}$, equipped with a 2-natural isomorphism
\begin{equation}\label{eq:cotensor}
  \cat \bigl( C, \K(-, X) \bigr) \cong \K\bigl( -, \{C, X\} \bigr) 
\end{equation}
of 2-functors $\K^{op} \to \cat$.
If the cotensor $\{C, X\}$ exists in $\K$ for every object $X$ and every small category $C$, we say that \emph{$\K$ has all cotensors}.
\end{defn}

\begin{rmk}\label{rmk:cotensor-for-Talgs}
If $\K = \Talgs$ or $\K = \Talg$ for a 2-monad $\T$ on $\cat$, then $\{C, X \}$ will be the ordinary functor category $\cat \bigl(C, \usf X\bigr)$ equipped with the pointwise $\T$-algebra structure.
\end{rmk}

\begin{notn}\label{notn:underlying}
  If $\K$ is a 2-category, we let $\Kzero$ denote the underlying category of $\K$.
  If $F \cn \K \to \zL$ is a 2-functor, we let $F_0 \cn \Kzero \to \zLzero$ denote the functor obtained by restricting $F$ to the underlying categories.
\end{notn}

\begin{notn}\label{notn:arrow}
  We let $\bii = \{0 \to 1\}$ denote the free arrow category, consisting of two objects and one non-identity morphism.
  Similarly, let $\II = \{0 \iso 1\}$ denote the free isomorphism category, consisting of two objects and an isomorphism between them.
\end{notn} 

\begin{rmk}\label{rmk:special-cotensors}
  Most of our work below will depend only on cotensors of the form $\{\bii,X\}$ and $\{\II,X\}$.
  Unpacking \cref{defn:cotensor,notn:arrow} in these cases gives the following direct descriptions of \cref{eq:cotensor} on 1-cells.
  \begin{enumerate}
  \item 1-cells $f\cn W \to \{\bii,X\}$ in $\K$ are in bijection with triples $(f_1,f_2,\al)$ where $f_1, f_2 \cn W \to X$ are 1-cells and $\alpha \cn f_1 \to f_2$ is a 2-cell in $\K$.
  \item 1-cells $f\cn W \to \{\II,X\}$ in $\K$ are described similarly, with $\alpha$ being an isomorphism.
  \end{enumerate}
\end{rmk}

Recall from \cref{eq:iuu} the forgetful functors 
$\usf$ from $\Talgs$ and $\Talg$ to $\K$ and the inclusion
$\itt$ from $\Talgs$ to $\Talg$.
We need the following two facts about cotensor products; proofs of both can be found in {\cite{BKP1989Two}}.
\begin{prop}\label{prop:cotensor-facts}
  \ 
\begin{enumerate}
\item\label{it:cotensor-facts-i} {\cite[Proposition~2.5]{BKP1989Two}}
  Suppose $\K$ is a 2-category, and $\T$ is a 2-monad on $\K$.
  If $C$ is a small category and $\K$ admits all cotensors of the form $\{C, X\}$, then so do $\Talgs$ and $\Talg$.
  Moreover, the inclusion $\itt$ and both forgetful functors $\usf$ preserve those cotensors.
\item\label{it:cotensor-facts-ii} {\cite[Proposition~3.1]{BKP1989Two}}
  Suppose $\zA$ and $\zB$ are 2-categories such that $\zA$ admits cotensors of the form $\{ \bii, X \}$.
  Suppose $V \cn \zA \to \zB$ is a 2-functor that preserves those cotensors.
  Then the underlying functor $V_0 \cn \zAzero \to \zBzero$ has a left adjoint if and only if $V$ has a left 2-adjoint.
\end{enumerate}
\end{prop}

Now we turn to a discussion of various coequalizers and their relation to $\T$-algebra structures.
\begin{defn}[Split coequalizers and $\usf$-split pairs]\label{defn:split-coeq}
Suppose $C$ is a category, and $\usf \cn C \to C'$ is a functor.
\begin{enumerate}
\item A \emph{split coequalizer} in $C$ is a diagram of the form below,
  \begin{equation}\label{eq:split-coeq}
    \begin{tikzpicture}[x=15ex,y=8ex,vcenter]
      \def\halfsep{.7mm}
      \draw[0cell] 
      (0,0) node (a) {X}
      (1,0) node (b) {Y}
      (2,0) node (c) {Z}
      ;
      \draw[1cell] 
      (a) edge[transform canvas={yshift=\halfsep}] node (f) {f} (b)
      (a) edge[swap,transform canvas={yshift=-\halfsep}] node (g) {g} (b)
      (b) edge[swap] node (h) {h} (c)
      (c) edge[swap, bend right=30] node (s) {s} (b)
      (b) edge[swap, bend right=45] node (t) {t} (a)
      ;
    \end{tikzpicture}
  \end{equation}
  such that the following equations hold.
  \begin{equation}\label{eq:split-coeq-axioms}
    \begin{split}
      \begin{aligned}
        hf & = hg \\
        hs & = 1_Z
      \end{aligned}
      \qquad\qquad
      \begin{aligned}
        sh & = gt \\
        ft & = 1_Y
      \end{aligned}
    \end{split}
  \end{equation}
  In this case, $h$ is said to be \emph{a split coequalizer} of $f$ and $g$.
\item Suppose $f, g \cn X \to Y$ are parallel arrows in $C$.
This pair is called a \emph{$\usf$-split pair} if there exists an object $Z'$ together with morphisms $h'$, $s'$, and $t'$ in $C'$ such that
\begin{equation}\label{eq:u-split-coeq}
    \begin{tikzpicture}[x=15ex,y=8ex,vcenter]
      \def\halfsep{.7mm}
      \draw[0cell] 
      (0,0) node (a) {\usf X}
      (1,0) node (b) {\usf Y}
      (.96,0.1) node (bt) {}
      (.96,-0.1) node (bb) {}
      (2,0) node (c) {Z'}
      ;
      \draw[1cell] 
      (a) edge[transform canvas={yshift=\halfsep}] node (f) {\usf f} (b)
      (a) edge[swap,transform canvas={yshift=-\halfsep}] node (g) {\usf g} (b)
      (b) edge[swap] node (h) {h'} (c)
      (c) edge[swap, bend right=30] node (s) {s'} (b)
      (b) edge[swap, bend right=45] node (t) {t'} (a)
      ;
    \end{tikzpicture}
  \end{equation}
  is a split coequalizer in $C'$.
\end{enumerate}
\end{defn}

\begin{rmk}[Split coequalizers are coequalizers]\label{rmk:splits}
Suppose given a split coequalizer as in \cref{eq:split-coeq} and a morphism $p \cn Y \to W$ such that $pf = pg$.
Then the unique morphism $\wt{p} \cn Z \to W$ such that $p = \wt{p} h$ is given by the formula
\[
\wt{p} = ps.
\]
Therefore, $h$ is the coequalizer of $f$ and $g$.
\end{rmk}

\begin{rmk}[Split coequalizers are absolute]\label{rmk:absolute}
Suppose given a split coequalizer in $C$ as in \cref{eq:split-coeq}, and a functor $F \cn C \to D$.
Then applying $F$ to the entire diagram gives a split coequalizer in $D$.
\end{rmk}

\begin{example}[The canonical $\usf$-split pair for a $\T$-algebra]\label{example:usplit}
Suppose $\T$ is a monad on a category $C$, and $x \cn \T X \to X$ is a $\T$-algebra structure on an object $X$.
Then the pair $\mu, \T x \cn \T^2X \to \T X$ has $x \cn \T X \to X$ as its coequalizer in $\Talgs$, and is $\usf$-split for $\usf$ the forgetful functor from $\Talgs$ back to $C$.
An explicit splitting in $C$, with the forgetful functor $\usf$ suppressed, is given below.
\begin{equation}\label{eq:u-split-alg}
    \begin{tikzpicture}[x=15ex,y=8ex,vcenter]
      \def\halfsep{.7mm}
      \draw[0cell] 
      (0,0) node (a) {\T^2X}
      (1,0) node (b) {\T X}
      (2,0) node (c) {X}
      ;
      \draw[1cell] 
      (a) edge[transform canvas={yshift=\halfsep}] node (f) {\mu} (b)
      (a) edge[swap,transform canvas={yshift=-\halfsep}] node (g) {\T x} (b)
      (b) edge[swap] node (h) {x} (c)
      (c) edge[swap, bend right=30] node (s) {\eta_X} (b)
      (b) edge[swap, bend right=45] node (t) {\eta_{\T X}} (a)
      ;
    \end{tikzpicture}
  \end{equation}
  This observation is a key component of Beck's Monadicity Theorem \cite{Bec1967Triples} and related variants.
  See \cite[Section~VI.7]{ML98Categories} and \cite[Section~5.5]{Rie2017CTC}.
\end{example}

We require an analogue of the previous example in the 2-category $\Talg$ for a 2-monad $\T$ on a 2-category $\K$.
\begin{lem}\label{lem:2d-splitcoeq}
Suppose $\K$ is a 2-category, and that
  \begin{equation}\label{eq:split-coeq-lem}
    \begin{tikzpicture}[x=15ex,y=8ex,vcenter]
      \def\halfsep{.7mm}
      \draw[0cell] 
      (0,0) node (a) {X}
      (1,0) node (b) {Y}
      (2,0) node (c) {Z}
      ;
      \draw[1cell] 
      (a) edge[transform canvas={yshift=\halfsep}] node (f) {f} (b)
      (a) edge[swap,transform canvas={yshift=-\halfsep}] node (g) {g} (b)
      (b) edge[swap] node (h) {h} (c)
      (c) edge[swap, bend right=30] node (s) {s} (b)
      (b) edge[swap, bend right=45] node (t) {t} (a)
      ;
    \end{tikzpicture}
  \end{equation}
  is a split coequalizer in $\K_0$, the underlying category of $\K$.
  Then $Z$ is also the $\cat$-enriched colimit of the same diagram, meaning it also satisfies the following 2-dimensional universal property.
  \begin{description}
  \item[2-dimensional universality of split coequalizers:] 
    Suppose given 1-cells
    \[
      p, q \cn Y \to W
    \]
    such that $pf = pg$ and $qf = qg$.
    Let $\wt{p}, \wt{q} \cn Z \to W$ be the unique 1-cells induced by the universal property of $h$ as the coequalizer of $f, g$ in $\K_0$.
    Then the functions given by whiskering with $h$ and $s$
    \[
      (-* h) \cn \K(Z, W)(\wt{p}, \wt{q}) \lradj \K(Y, W)(p, q) \bacn (- * s)
    \]
    induce inverse bijections between the set of 2-cells $\wt{\al}\cn \wt{p} \to \wt{q}$ and the subset
    \[
      \bigl\{ \alpha \cn p \to q \, \big|  \, \alpha * f = \alpha * g \bigr\}
      \subseteq
      \K(Y, W)(p, q).
    \]
  \end{description}
\end{lem}
\begin{proof}
Suppose $\alpha \cn p \to q$ such that $ \alpha * f = \alpha * g$.
Recall (\cref{rmk:splits}) that $\wt{p} = ps$ and $\wt{q} = qs$, and define $\wt{\alpha} \cn \wt{p} \to \wt{q}$ to be $\alpha * s$.
Then
 \begin{equation}\label{eq:split-coeq-2d}
\wt{\alpha} * h = \alpha * sh = \alpha * gt = \alpha * ft = \alpha * 1_Y = \alpha
  \end{equation}
  by the definition of $\wt{\alpha}$, the assumption $ \alpha * f = \alpha * g$, and the equations in \cref{eq:split-coeq-axioms}.
  It remains to prove that $\alpha * s$ is the only 2-cell $\beta \cn \wt{p} \to \wt{q}$ such that $\beta * h = \alpha$.
  Indeed, if $\beta * h = \alpha$, then 
  \[
  \beta = \beta * 1_Z = \beta * hs = \alpha * s = \wt{\alpha}.
  \]
\end{proof}

\begin{rmk}\label{rmk:split-coeq-2d}
  Note, in the context of \cref{lem:2d-splitcoeq} above, that the 2-cell
  \[
    \al = \wt{\al} * h
  \]
  is invertible if and only if $\wt{\al}$ is invertible.
  This follows because the inverse bijection to $(- * h)$ is $(- * s)$ and whiskering preserves invertibility of 2-cells.
\end{rmk}

We adopt the following temporary notation to distinguish between the two different versions of $\usf$ for a 2-monad $\T$.

\begin{notn}\label{notn:u-s}
Suppose $\T$ is a 2-monad on a 2-category $\K$.
We write $\usf_s \cn \Talgs \to \K$ for the forgetful functor when considering only the strict $\T$-maps, and $\usf \cn \Talg \to \K$ when considering all $\T$-maps.
In this notation, the commutative diagram \cref{eq:iuu} is an equality $\usf \circ \itt = \usf_s$ as 2-functors $\Talgs \to \K$.
\end{notn}

\begin{prop}\label{prop:hstrict}
  Suppose $f,g\cn (X,\Xmul) \to (Y,\Ymul)$ is a $\usf_s$-split pair of strict $\T$-maps.
  Let $h \cn Y \to Z$ be the split coequalizer of $\usf_s f, \usf_s g$ in $\K$.
  Then $h$ is the underlying 1-cell of a strict $\T$-map, also denoted $h$, and is the coequalizer in $\Talgs$ of the pair $f,g$.
\end{prop}
\begin{proof}
  This follows from the analogous standard result for 1-monads, e.g., \cite[Proposition~5.4.9]{Rie2017CTC}, and \cref{rmk:strictTmaps-underlying}.
\end{proof}

\begin{lem}\label{lem:wtk-Tmap}
  Suppose given $f$, $g$, and $h$ as in \cref{prop:hstrict} and
  suppose given a 1-cell $\wt{k}$ and a 2-cell $\wt{k}_\bullet$ in $\K$
  \[
    \wt{k} \cn Z \to W 
    \andspace
    \wt{k}_{\bullet} \cn \Wmul \circ \T \wt{k} \to \wt{k} \circ \Zmul
  \]
  for some $\T$-algebra $(W,\Wmul)$.
  Then $(\wt{k},\wt{k}_\bullet)$ is a $\T$-map $(Z,\Zmul) \zzto (W,\Wmul)$ if and only if the composite
  \begin{equation}\label{eq:wtk-h}
    (k,k_\bullet) = (\wt{k},\wt{k}_\bullet) \circ h = (\wt{k} \circ h, \wt{k}_\bullet * \T h)
  \end{equation}
  is a $\T$-map $(Y,\Ymul) \zzto (W,\Wmul)$.
\end{lem} 
\begin{proof}
  If $(\wt{k},\wt{k}_\bullet)$ is a $\T$-map, then the composite $(\wt{k},\wt{k}_\bullet) \circ h$ is a $\T$-map.
  In this case, the composition formula \cref{eq:Tmap-comp-formula} simplifies to the right hand side of \cref{eq:wtk-h} because $h$ is a strict $\T$-map.

  For the reverse implication, let $(k,k_\bullet)$ be defined via the formula on the right hand side of \cref{eq:wtk-h}.
  Since $h$ is a split coequalizer, recall from \cref{rmk:absolute} that $\T h$ is too.
  Therefore, applying \cref{rmk:split-coeq-2d} to $\T h$, invertibility of $k_\bullet$ implies that of $\wt{k}_\bullet$.
  Now it remains to show that the $\T$-map axioms \cref{lax-morphism-unit-axiom,lax-morphism-structure-axiom} for $(k,k_\bullet)$ imply those for $(\wt{k},\wt{k}_\bullet)$.
  This verification uses the hypothesis that $h$ is a split coequalizer in $\K$ and, separately, the implication that $\T^2h$ is also a split coequalizer in $\K$ by \cref{rmk:absolute}.
  The applications of both of these facts use the 2-dimensional universality from  \cref{lem:2d-splitcoeq}.

  For the unit axiom \cref{lax-morphism-unit-axiom}, we must verify that $\wt{k}_\bullet * \eta_Z = 1_{\wt{k}}$.
  Note that the source and target of $\wt{k}_\bullet * \eta_Z$ are both equal to $\wt{k}$ by naturality of $\eta$ and the unit axioms for $(Z, z)$ and $(W, w)$, respectively:
  \[
    \begin{aligned}
      z \circ \eta_Z & = 1_Z, \\
      w \circ \eta_W & = 1_W.
    \end{aligned}
  \]
  The two-dimensional part of the universal property of the split coequalizer $h \cn Y \to Z$ (\cref{lem:2d-splitcoeq}) implies that the 2-cell $\wt{k}_\bullet * \eta_Z$ is an identity if and only if it is the identity $1_k$ after applying $- * h$.
  The following computation uses naturality of $\eta$, the defining equality $k_\bullet = \wt{k}_\bullet * \T h$ \cref{eq:wtk-h}, and the unit axiom for $(k, k_\bullet)$, respectively:
  \begin{align*}
    \wt{k}_\bullet * \eta_Z * h & = \wt{k}_\bullet * \T h * \eta_Y \\
                                & = k_\bullet * \eta_Y \\ 
                                & = 1_k.
  \end{align*}
  This verifies the unit axiom \cref{lax-morphism-unit-axiom} for $(\wt{k},\wt{k}_\bullet)$.

  For the multiplication axiom \cref{lax-morphism-structure-axiom}, we must check the equality of pastings below.
  \begin{equation}\label{wtk-morphism-structure-axiom}
    \begin{tikzpicture}[x=.4ex,y=.4ex,xscale=1.2]
      \def\boundary{
        \draw[0cell] 
        (0,0) node (LL) {\T^2Z}
        (20,20) node (TL) {\T Z}
        (20,-20) node (BL) {\T^2W}
        (50,20) node (TR)  {Z}
        (50,-20) node (BR) {\T W}
        (70,0) node (RR)  {W}
        ;
        \draw[1cell] 
        (LL) edge node {\mu_{Z}} (TL)
        (TL) edge node {\Zmul} (TR)
        (TR) edge node {\wt{k}} (RR)
        (LL) edge[swap] node {\T^2 \wt{k}} (BL)
        (BL) edge[swap] node {\T \Wmul} (BR)
        (BR) edge[swap] node {\Wmul} (RR)
        ;
      }
      \begin{scope}
        \boundary
        \draw[0cell]
        (30,0) node (CC)  {\T Z}
        ;
        \draw[1cell] 
        (LL) edge node {\T \Zmul} (CC)
        (CC) edge node {\Zmul} (TR)
        (CC) edge node {\T \wt{k}} (BR)
        ;
        \draw[2cell] 
        node[between=BL and CC at .5, rotate=60, 2label={below,\T \wt{k}_\bullet}] {\Rightarrow}
        node[between=CC and RR at .55, rotate=90, 2label={below,\wt{k}_\bullet}] {\Rightarrow}
        ;
        \draw (85,2.5) node {and};
      \end{scope}
      \begin{scope}[shift={(103,0)}]
        \boundary
        \draw[0cell]
        (40,0) node (CC)  {\T W}
        ;
        \draw[1cell] 
        (TL) edge[swap] node {\T \wt{k}} (CC)
        (BL) edge node {\mu_{W}} (CC)
        (CC) edge node {\Wmul} (RR)
        ;
        \draw[2cell] 
        node[between=CC and TR at .5, rotate=60, 2label={above,\wt{k}_\bullet}] {\Rightarrow}
        ;
      \end{scope}
    \end{tikzpicture}
  \end{equation}
  Once again using that $h$ is a split coequalizer, and therefore $\T^2 h$ is also (\cref{rmk:absolute}), the desired equality holds if and only if it holds after applying $-* \T^2h$.

  {
    \def\tmpxscale{.49ex}
    \def\tmpyscale{.47ex}
    \def\tmppicscale{.75}
    \def\nodescale{.7}
    \def\onescale{.7}
    \def\twoscale{.8}
    \def\hshift{135}
    \def\vshift{0}
    \def\requals{\draw (93,8) node[font=\Large,rotate=0] {=};}
    \def\lequals{\draw (-40,-10) node[font=\Large,rotate=15] {=};}
    \def\boundary{
      \draw[0cell=\nodescale] 
      (-30,0) node (TTY) {\T^2Y}
      (-10,25) node (TYa) {\T Y}
      (20,35) node (Y) {Y}
      (60,25) node (Z)  {Z}
      (20,-20) node (TTW) {\T^2W}
      (60,-20) node (TWa) {\T W}
      (80,0) node (W)  {W}
      ;
      \draw[1cell=\onescale] 
      (TTY) edge node {\mu_Y} (TYa)
      (TYa) edge node {\Ymul} (Y)
      (Y) edge node {h} (Z)
      (TTY) edge['] node {\T^2k} (TTW)
      (Z) edge node {\wt{k}} (W)
      (TTW) edge[swap] node {\T \Wmul} (TWa)
      (TWa) edge[swap] node {\Wmul} (W)
      ;
    }
    \def\drawTYb{(15,15) node (TYb) {\T Y}}
    \def\drawTZb{(50,0) node (TZb) {\T Z}}
    \def\drawTTZ{(8,0) node (TTZ) {\T^2Z}}
    \def\drawTWb{(35,0) node (TWb) {\T W}}
    \def\drawTZc{(5,8) node (TZc) {\T Z}}
    \def\drawTTZc{(-4,-3) node (TTZc) {\T^2Z}}

    Whiskering the left pasting diagram in \cref{wtk-morphism-structure-axiom} with $\T^2h$ gives the left diagram below, where the additional regions commute because $h$ is a strict $\T$-map by \cref{prop:hstrict}, $k = \wt{k}h$ by definition \cref{eq:wtk-h}, $\mu$ is 2-natural, and $\T$ is 2-functorial.
    The equality of pastings is immediate as the only difference between the diagrams is how commutative regions are displayed.
    \[
      \begin{tikzpicture}[x=\tmpxscale,y=\tmpyscale,scale=\tmppicscale]
        \begin{scope}
          \boundary
          \draw[0cell=\nodescale]
          (20,20) node (TZa) {\T Z}
          \drawTTZ
          \drawTZb
          ;
          \draw[1cell=\onescale] 
          (TTZ) edge node {\T \Zmul} (TZb)
          (TZb) edge node {\Zmul} (Z)
          (TZb) edge node {\T \wt{k}} (TWa)
          (TTY) edge node[scale=.9,pos=.65] {\T^2h} (TTZ)
          (TTZ) edge[] node {\T^2 \wt{k}} (TTW)
          (TZa) edge node {\Zmul} (Z)
          (TYa) edge node {\T h} (TZa)
          (TTZ) edge node {\mu_{Z}} (TZa)
          ;
          \draw[2cell=\twoscale] 
          node[between=TTW and TZb at .5, rotate=60, 2label={below,\T \wt{k}_\bullet}] {\Rightarrow}
          node[between=TZb and W at .45, shift={(0,3)}, rotate=90, 2label={below,\wt{k}_\bullet}] {\Rightarrow}
          ;
          \requals
        \end{scope}
        \begin{scope}[shift={(\hshift,-.5*\vshift)}]
          \boundary
          \draw[0cell=\nodescale]
          \drawTTZ
          \drawTYb
          \drawTZb
          ;
          \draw[1cell=\onescale] 
          (TTY) edge node[scale=.9,pos=.65] {\T^2h} (TTZ)
          (TTZ) edge[] node {\T^2 \wt{k}} (TTW)
          (TTZ) edge node {\T \Zmul} (TZb)
          (TZb) edge node {\Zmul} (Z)
          (TZb) edge node {\T \wt{k}} (TWa)
          (TTY) edge node {\T \Ymul} (TYb)
          (TYb) edge node {\Ymul} (Y)
          (TYb) edge[] node {\T h} (TZb)
          ;
          \draw[2cell=\twoscale] 
          node[between=TTW and TZb at .5, rotate=60, 2label={below,\T \wt{k}_\bullet}] {\Rightarrow}
          node[between=TZb and W at .45, shift={(0,3)}, rotate=90, 2label={below,\wt{k}_\bullet}] {\Rightarrow}
          ;
        \end{scope}
      \end{tikzpicture} 
    \]
    The pasting in the diagram at right above is equal to that of the diagram at left below by applying $\T$ to the defining equality $k_\bullet = \wt{k}_\bullet * \T h$ \cref{eq:wtk-h}.
    Another application of the same equality shows that the two pastings below are equal.
    \[
      \begin{tikzpicture}[x=\tmpxscale,y=\tmpyscale,scale=\tmppicscale]
        \begin{scope}
          \boundary
          \draw[0cell=\nodescale]
          \drawTYb
          \drawTZb
          ;
          \draw[1cell=\onescale] 
          (TZb) edge node {\Zmul} (Z)
          (TZb) edge node {\T \wt{k}} (TWa)
          (TTY) edge node {\T \Ymul} (TYb)
          (TYb) edge node {\Ymul} (Y)
          (TYb) edge[] node {\T h} (TZb)
          (TYb) edge['] node {\T k} (TWa)
          ;
          \draw[2cell=\twoscale] 
          node[between=TTW and TYb at .5, rotate=30, 2label={above,\T k_\bullet}] {\Rightarrow}
          node[between=TZb and W at .45, shift={(0,3)}, rotate=90, 2label={below,\wt{k}_\bullet}] {\Rightarrow}
          ;
          \requals
        \end{scope}
        \begin{scope}[shift={(\hshift,-.5*\vshift)}]
          \boundary
          \draw[0cell=\nodescale]
          \drawTYb
          ;
          \draw[1cell=\onescale] 
          (TTY) edge node {\T \Ymul} (TYb)
          (TYb) edge node {\Ymul} (Y)
          (TYb) edge['] node {\T k} (TWa)
          (Y) edge[] node {k} (W)
          ;
          \draw[2cell=\twoscale] 
          node[between=TTW and TYb at .5, rotate=30, 2label={above,\T k_\bullet}] {\Rightarrow}
          node[between=TYb and W at .55, shift={(0,-5)}, rotate=90, 2label={below,k_\bullet}] {\Rightarrow}
          ;
        \end{scope}
      \end{tikzpicture}
    \]
    Lastly, the pasting in the diagram at right above is equal to that of the diagram at left below by the multiplication axiom \cref{lax-morphism-structure-axiom} for $(k,k_\bullet)$.
    Equality of the two pastings below holds by another application of \cref{eq:wtk-h}.
    \[
      \begin{tikzpicture}[x=\tmpxscale,y=\tmpyscale,scale=\tmppicscale]
        \begin{scope}
          \boundary
          \draw[0cell=\nodescale]
          \drawTWb
          \drawTZc
          \drawTTZc
          ;
          \draw[1cell=\onescale] 
          (Y) edge[] node {k} (W)
          (TYa) edge node {\T k} (TWb)
          (TTW) edge['] node {\mu_W} (TWb)
          (TWb) edge['] node {\Wmul} (W)
          (TTY) edge node[pos=.3] {\T^2h} (TTZc)
          (TTZc) edge node[pos=.4] {\T^2\wt{k}} (TTW)
          (TYa) edge['] node {\T h} (TZc)
          (TZc) edge['] node[pos=.6] {\T \wt{k}} (TWb)
          (TTZc) edge['] node[scale=.8] {\mu_Z} (TZc)
          ;
          \draw[2cell=\twoscale] 
          node[between=TWb and Y at .5, rotate=90, 2label={below,k_\bullet}] {\Rightarrow}
          ;
          \requals
        \end{scope}
        \begin{scope}[shift={(\hshift,-.5*\vshift)}]
          \boundary
          \draw[0cell=\nodescale]
          \drawTWb
          \drawTZc
          \drawTTZc
          ;
          \draw[1cell=\onescale] 
          (TTW) edge['] node {\mu_W} (TWb)
          (TWb) edge['] node {\Wmul} (W)
          (TTY) edge node[pos=.3] {\T^2h} (TTZc)
          (TTZc) edge node[pos=.4] {\T^2\wt{k}} (TTW)
          (TYa) edge['] node {\T h} (TZc)
          (TZc) edge['] node[pos=.6] {\T \wt{k}} (TWb)
          (TTZc) edge['] node[scale=.8] {\mu_Z} (TZc)
          (TZc) edge node {\Zmul} (Z)
          ;
          \draw[2cell=\twoscale] 
          node [between=TWb and Z at .45, rotate=60, 2label={below,\wt{k}_\bullet}] {\Rightarrow}  
          ;
        \end{scope}
      \end{tikzpicture}
    \]
    The final pasting at right above is the whiskering of the right hand diagram in \cref{wtk-morphism-structure-axiom} with $\T^2h$.
  }
  
  This shows that the two sides of \cref{wtk-morphism-structure-axiom} are equal after applying $- * \T^2h$, and hence completes the proof that the two pastings in \cref{wtk-morphism-structure-axiom} are equal.
  This completes the proof that $(\wt{k},\wt{k}_\bullet)$ satisfies the axioms of a $\T$-map.
\end{proof}

\begin{prop}\label{prop:itt-and-u-split-coeq}
Suppose $\T$ is a 2-monad on a 2-category $\K$.
The 2-functor $\itt \cn \Talgs \to \Talg$ sends coequalizers of $\usf_s$-split pairs to coequalizers of $\usf$-split pairs.
\end{prop}
\begin{proof}
  Suppose $h\cn (Y,\Ymul) \to (Z,\Zmul)$ is the coequalizer in $\Talgs$ of a $\usf_s$-split pair $f,g\cn (X,\Xmul) \to (Y,\Ymul)$.
  Let $h' \cn Y \to Z'$ be the split coequalizer in $\K$ of $\usf_s f$ and $\usf_s g$.
  By \cref{prop:hstrict}, $h'$ is the underlying 1-cell of a strict $\T$-map, so by uniqueness of coequalizers we assume $Z' = Z$ and $h' = \usf_s h$.
 
  Thus, there are 1-cells $s$ and $t$ in $\K$ such that the following is a split coequalizer in $\K$.
  \begin{equation}\label{eq:usplit-coeq-Talgs}
    \begin{tikzpicture}[x=17ex,y=8ex,vcenter,xscale=1.2]
      \def\halfsep{.7mm}
      \draw[0cell] 
      (0,0) node (a) {\usf_s (X,\Xmul)}
      (1,0) node (b) {\usf_s (Y,\Ymul)}
      (2,0) node (c) {\usf_s (Z,\Zmul)}
      ;
      \draw[1cell] 
      (a) edge[transform canvas={yshift=\halfsep}] node (f) {\usf_s f} (b) 
      (a) edge[swap,transform canvas={yshift=-\halfsep}] node (g) {\usf_s g} (b) 
      (b) edge[swap] node (h) {h' = \usf_s h} (c)
      (c) edge[swap, bend right=30] node (s) {s} (b)
      (b) edge[swap, bend right=45] node (t) {t} (a)
      ;
    \end{tikzpicture}
  \end{equation}
  We will show that $\itt h$ is the coequalizer of $\itt f$ and $\itt g$ in $\Talg$.
  Since $\usf_s = \usf \circ \itt$, the same $s$ and $t$ will then make $\itt f, \itt g$ a $\usf$-split pair.

  To prove that $\itt h$ is the coequalizer of $\itt f$ and $\itt g$ in $\Talg$,
  suppose given a $\T$-map
  \[
    (k, k_\bullet) \cn (Y, y) \zzto (W, w)
  \]
  such that
  \begin{equation}\label{eq:k-hypothesis}
    (k, k_\bullet) \circ \itt f = (k, k_\bullet) \circ \itt g.
  \end{equation}
  We will show that there exists a unique $\T$-map 
  \begin{equation}\label{eq:wtk}
    (\wt{k}, \wt{k}_\bullet) \cn (Z, \Zmul) \zzto (W, \Wmul)
  \end{equation}
  such that 
  \begin{equation}\label{eq:wtk-conclusion}
    (k, k_\bullet) = (\wt{k}, \wt{k}_\bullet) \circ \itt h.
  \end{equation}

  Applying $\usf$ to \cref{eq:k-hypothesis}, we have $k f = k g$.
  Since $h$ is the coequalizer of $f, g$ in $\K$, we define $\wt{k}$ as the unique 1-cell in $\K$ induced by the universal property of the coequalizer.
  Thus, we have an equality in $\K$:
  \begin{equation}\label{eq:wtk-1cell-conclusion}
    k = \wt{k} \circ h.
  \end{equation}

  Next we note that, because \cref{eq:k-hypothesis} is an equality of $\T$-maps, the two sides have the same algebra constraints.
  Recalling the formula \cref{eq:Tmap-comp-formula} for algebra constraints of a composite, we have
  \begin{equation}\label{eq:kbullet3}
    k_\bullet *\T f = k_\bullet *\T g
  \end{equation}
  because both $f$ and $g$ are strict $\T$-maps.
  The algebra constraint $k_\bullet$ is shown in the rectangle below, where each of the triangles commutes by the equality \cref{eq:wtk-1cell-conclusion}.
  \begin{equation}\label{eq:kbullet2}
    \begin{tikzpicture}[x=15ex,y=10ex,vcenter]
      \draw[0cell] 
      (-1,0) node (TY) {\T Y}
      (0,.5) node (TZ) {\T Z}
      (1,0) node (TW) {\T W}
      (-1,-1) node (Y) {Y}
      (0,-1.5) node (Z) {Z}
      (1,-1) node (W) {W}
      ;
      \draw[1cell] 
      (TY) edge[] node (th) {\T h} (TZ) 
      (TZ) edge[] node (twtk) {\T \wt{k}} (TW) 
      (TW) edge[] node (w) {w} (W)
      (TY) edge[swap] node (y) {\Ymul} (Y) 
      (Y) edge['] node (h) {h} (Z) 
      (Z) edge['] node (wtk) {\wt{k}} (W) 
      (TY) edge node {\T k} (TW)
      (Y) edge['] node {k} (W)
      ;
      \draw[2cell] 
      node[between=TZ and Z at {.5}, rotate=225, 2label={below,k_\bullet}] {\Rightarrow}
      ;
    \end{tikzpicture}
  \end{equation}
  Since $h$ is a strict $\T$-map, we have $\Zmul \circ \T h = h \circ \Ymul$ and, therefore, $k_\bullet$ has target 
  \begin{equation}\label{eq:kbullet4}
    \wt{k} \circ h \circ \Ymul = \wt{k} \circ \Zmul \circ \T h.
  \end{equation}

  Since $h$ is a split coequalizer in $\K$, so is $\T h$ by \cref{rmk:absolute}.
  Therefore, by \cref{lem:2d-splitcoeq}, $\T h$ satisfies an additional two-dimensional aspect to its universal property: the whiskering function $- * \T h$ induces an isomorphism between the set of 2-cells $\K(\T Z, W)(\Wmul \circ \T \wt{k}, \wt{k} \circ \Zmul)$ and the subset
  \begin{align*}
    S &  = \bigl\{ \alpha \cn \Wmul \circ \T \wt{k} \circ \T h \to  \wt{k} \circ \Zmul \circ \T h \, \big|  \, \alpha * \T f = \alpha * \T g \bigr\}\\
    & \subseteq \K\bigl(\T Y, W\bigr)\bigl(\Wmul \circ \T \wt{k} \circ \T h, \wt{k} \circ \Zmul \circ \T h\bigr).
  \end{align*}

  Combining \cref{eq:kbullet3,eq:kbullet2,eq:kbullet4} shows that the algebra constraint $k_\bullet$ is a member of the subset $S$.
  Therefore, by the two-dimensional aspect of the universal property for $\T h$, there is a unique 2-cell in $\K$
  \[
    \wt{k}_\bullet\cn w \circ \T \wt{k} \to \wt{k} \circ \Zmul
  \] 
  such that 
  \begin{equation}\label{eq:wtkbullet1}
    k_\bullet = \wt{k}_\bullet * \T h.
  \end{equation}
  Since $(k,k_\bullet)$ is a $\T$-map, the equalities \cref{eq:wtk-1cell-conclusion,eq:wtkbullet1} imply, by \cref{lem:wtk-Tmap}, that $(\wt{k},\wt{k}_\bullet)$ is a $\T$-map.
  
  The calculation above verifies that there is a unique $\T$-map $(\wt{k}, \wt{k}_\bullet)$ such that 
  \[
    (k, k_\bullet) = (\wt{k}, \wt{k}_\bullet) \circ \itt h.
  \]
  This completes the proof that $\itt h$ is the coequalizer of $\itt f$ and $\itt g$, as desired.
\end{proof}

\section{Pseudomorphism classifiers}
\label{sec:bg-psmor-class}

For many 2-monads $\T$ of interest, the inclusion \cref{eq:incl-i}
\[
  \itt \cn \Talgs \hookrightarrow \Talg
\]
has a left 2-adjoint.
In such cases, the left 2-adjoint can be used to develop strictification and coherence results, as we will do in \cref{sec:finitary-udc}.

This section and the next recall the basic terminology and related properties.
Much of this content comes from \cite{BKP1989Two}, and we refer the reader there for further development.
Examples, in the special case of monads that encode strict monoidal structures, are explained in \cref{sec:psmorclass}.

\begin{defn}[Pseudomorphism Classifier]
  Suppose given a 2-monad $\T$ on a 2-category $\K$.
  A \emph{pseudomorphism classifier} for $\T$ is a left 2-adjoint $\Q \dashv \itt$ as shown below.
  \begin{equation}\label{eq:Qi-adj}
    \begin{tikzpicture}[x=15ex,y=8ex,vcenter]
      \draw[0cell] 
      (0,0) node (x) {\Talg}
      (1,0) node (y) {\Talgs}
      ;
      \draw[1cell] 
      (x) edge[bend left=12,transform canvas={yshift=.7mm}] node (L) {\Q} (y) 
      (y) edge[bend left=12,transform canvas={yshift=-.7mm}] node (R) {\itt} (x) 
      ;
      \draw[2cell] 
      node[between=L and R at .5] {\bot}
      ;
    \end{tikzpicture}
  \end{equation}
  The unit $\zeta\cn 1 \to \itt \Q$ has components that are $\T$-maps
  \[
    \zeta_X \cn X \zzto \itt \Q X \forspace X \in \Talg.
  \]
  The counit $\delta\cn \Q \itt \to 1$ has components that are \emph{strict} $\T$-maps 
  \[
    \delta_Y \cn \Q \itt Y \to Y \forspace Y \in \Talgs.
  \]
\end{defn}

The unit and counit of a pseudomorphism classifier $\Q$ satisfy triangle identities that lead to a 2-natural isomorphism of categories
\[
  \Talgs(\Q X, Y) \cong \Talg(X, \itt Y)
\]
for every pair of $\T$-algebras $X$ and $Y$.
This is the standard translation between the hom-set and unit/counit expressions for an adjunction.
In this context, we use the following notation.
\begin{defn}\label{defn:fbot}
  For each $\T$-map $f \cn X \zzto \itt Y$, let $f^\bot \cn \Q X \to Y$ be the strict $\T$-map that is the mate of $f$.
  Thus, $f$ factors uniquely as follows.
  \begin{equation}\label{eq:fbot}
    \begin{tikzpicture}[x=12ex,y=8ex,vcenter,xscale=1.2]
      \draw[0cell] 
      (0,0) node (iqx) {\itt \Q X}
      (iqx)++(0,-1) node (x) {X}
      (iqx)++(1,0) node (x') {Y}
      ;
      \draw[1cell] 
      (iqx) edge node {f^\bot} (x')
      ;
      \draw[zz1cell]
      (x) to['] node {f} (x')
      ;
      \draw[zz1cell]
      (x) to node {\ze_X} (iqx)
      ;
    \end{tikzpicture}
  \end{equation}
  \ 
\end{defn}

\begin{rmk}\label{rmk:Qi-triang}
  The triangle identities for $\Q \dashv \itt$ consist of the following equalities for each $Y \in \Talg$ and $X \in \Talgs$:
  \[
    \itt\de_Y \circ \ze_{\itt Y} = 1_{\itt Y} \andspace
    \de_{\Q X} \circ \Q \ze_X = 1_{\Q X}.
  \]
  Thus, omitting the inclusion $\itt$, as discussed in \cref{convention:ignorei}, we have $\de_Y \ze_Y = 1_Y$ for each $\T$-algebra $Y$.

  The composite $\ze_Y \de_Y$ is generally not equal to $1_Y$, but it often has other useful structure.
  This additional structure is described in \cref{defn:adj-surj-equiv,thm:BKP42} below.
\end{rmk}

We will use the following terminology in the 2-categories $\zA = \Talg$ and $\zA = \Talgs$.
\begin{defn}\label{defn:adj-surj-equiv}
  Suppose given a pair of 1-cells
  \[
    \ze \cn Y \to Z \andspace \de \cn Z \to Y
  \]
  in a 2-category $\zA$.
  \begin{description}
  \item[Surjective equivalence:] We say that $(\ze,\de)$ is a \emph{surjective equivalence in $\zA$} if $\de$ is a retraction, so that $\de \ze = 1_Y$, and there is 2-cell isomorphism
    \[
      \Theta \cn \ze \de \fto{\iso} 1_Z \inspace \zA.
    \]
    Thus, $(\ze,\de)$ is a surjective equivalence in $\zA$ if and only if there is a 2-cell isomorphism $\Theta$ such that $(\ze,\de,1_{1_Y},\Theta)$ is an internal equivalence in $\zA$.
    We say that $\de$ is a surjective equivalence if it has a section $\ze$ such that $(\ze,\de)$ is a surjective equivalence.
  \item[Adjoint surjective equivalence:] We say that $(\ze,\de,\Theta)$ is an \emph{adjoint surjective equivalence} if $(\ze,\de)$ is a surjective equivalence with $\Theta \cn \ze \de \iso 1_Z$ such that $\Theta * \ze = 1_\ze$.
    Thus, $(\ze,\de,\Theta)$ is an adjoint surjective equivalence if and only if $(\ze,\de,1_{1_Y},\Theta)$ is an internal adjoint equivalence in $\zA$.
  \end{description}
\end{defn}

\begin{defn}\label{defn:effectiveQi}
  Suppose $\T$ has a pseudomorphism classifier $(\Q,\itt,\ze,\de)$.
  We say that $(\Q,\itt)$ is \emph{effective} if, for each $\T$-algebra $Y$, there is a $\T$-algebra 2-cell isomorphism
  \[
    \Theta\cn \ze_Y \de_Y \fto{\iso} 1_{\Q Y}  
  \]
  such that $(\ze_Y,\de_Y,\Theta)$ is an adjoint surjective equivalence in $\Talg$.
  In this case, $\Theta$ is sometimes called the \emph{efficacy} of $(\Q,\itt)$.
\end{defn}

\begin{thm}[{\cite[Theorem~3.13]{BKP1989Two}}]\label{thm:BKP313}
  Suppose that $\K$ is a complete and cocomplete 2-category and suppose that $\T$ is a finitary monad on $\K$.
  Then $\T$ has a pseudomorphism classifier.
\end{thm}
\begin{rmk}
  The hypotheses of \cref{thm:BKP313} are convenient, but not necessary.
  See \cite[Remark~3.14]{BKP1989Two} for a discussion of the completeness hypothesis.
  The results of Power \cite{power1989coherence} and Lack \cite{Lac02Codescent} give an alternate approach under varying hypotheses, studying a more general coherence for pseudo-algebras.
  \Cref{rmk:power-lack} below discusses aspects of their work in relation to the applications in \cref{sec:psmorclass}.
\end{rmk}

\begin{thm}[{\cite[Theorem~4.2]{BKP1989Two}}]\label{thm:BKP42}
  Suppose $\T$ is a 2-monad on a 2-category $\K$ and suppose that $\K$ admits pseudolimits of 1-cells.
  If $\T$ has a pseudomorphism classifier $(\Q,\itt,\ze,\de)$ then it is an effective pseudomorphism classifier in the sense of \cref{defn:effectiveQi}.
\end{thm}

\begin{rmk}\label{note:delta-Talg}
  We note that even though $\delta_Y$ in \cref{defn:effectiveQi,thm:BKP42} is a strict $\T$-map, it is not guaranteed to have a strict $\T$-map for a pseudoinverse, a condition that would make $\delta_Y$ an equivalence in $\Talgs$.
  When there is a strict $\T$-map that is pseudoinverse to $\de$, then $Y$ is said to be a \emph{flexible} $\T$-algebra.
  While the full theory of flexible algebras will not be necessary in this work, we will use several results related to flexibility of free algebras from \cite[Section~4]{BKP1989Two}.
  These results are described in \cref{sec:Qi-strictification}.
\end{rmk}

\section{Effective pseudomorphism classifiers}\label{sec:Qi-strictification}

Throughout this section we suppose that $\T$ has an effective pseudomorphism classifier $(\Q,\itt,\ze,\de)$ in the sense of \cref{defn:effectiveQi}.
So, for each $\T$-algebra $Y$ there is a $\T$-algebra 2-cell isomorphism
\[
  \Theta\cn \ze_Y \de_Y \fto{\iso} 1_{\Q Y}  
\]
such that the following equalities hold, making $(\ze_Y,\de_Y,\Theta)$ an adjoint surjective equivalence in $\Talg$:
\begin{equation}\label{eq:ze-de-surjeq}
  \de_Y\ze_Y = 1_Y
  \andspace
  \Theta * \ze_Y = 1_{\ze_Y}.
\end{equation}
In this section, we prove a number of elementary properties that will be used in \cref{sec:finitary-udc,sec:De1-equiv}.

\begin{lem}\label{lem:zeflat}
  Suppose $C$ is an object of $\K$.
  There is a strict $\T$-map
  \[
    \ze_{\T C}^\flat \cn \T C \to \Q\T C
  \]
  together with an isomorphism
  \[
    \Theta^\flat \cn \ze_{\T C}^\flat \de_{\T C} \fto{\iso} 1_{\Q\T C}
  \]
  such that $(\ze_{\T C}^\flat,\de_{\T C},\Theta^\flat)$ is an adjoint surjective equivalence in $\Talgs$.
\end{lem}
\begin{proof}
  Consider the composite
  \begin{equation}\label{eq:pre-zetaflat}
    C \fto{\eta_C} \usf \T C \zzfto{\usf \ze_{\T C}} \usf \itt \Q\T C
  \end{equation}
  and define $\ze_{\T C}^\flat$ as the indicated composite in the following diagram.
  \begin{equation}\label{eq:zeflat}
    \begin{tikzpicture}[x=18ex,y=8ex,vcenter]
      \draw[0cell] 
      (0,0) node (a) {\T C}
      (a)++(.5,-1) node (c) {\T \usf \T C}
      (c)++(1.3,0) node (d) {\T \usf \itt \Q \T C}
      (d)++(.5,1) node (b) {\itt \Q\T C}
      ;
      \draw[1cell] 
      (a) edge['] node {\T \eta_C} (c)
      (c) edge node {\T \usf \ze_{\T C}} (d)
      (d) edge['] node {\epz_{\itt \Q\T C}} (b)
      (a) edge node {\ze_{\T C}^\flat} (b)
      ;
    \end{tikzpicture}
  \end{equation}
  That is, $\ze_{\T C}^\flat$ is the mate of \cref{eq:pre-zetaflat} under the adjunction $\T \dashv \usf$ \cref{eq:Tu-adj}.
  For the remainder of this proof, we omit the subscripts $\T C$ on $\de_{\T C}$, $\ze_{\T C}$, and $\ze_{\T C}^\flat$.
  
  Next we consider the composite $\de \ze^\flat$.
  Using the definition of $\ze^\flat$ in \cref{eq:zeflat}, naturality of $\epz$ with respect to the strict $\T$-map $\de$ gives the first equality below.
  The second follows from 2-functoriality of $\usf \T$, the left hand side of \cref{eq:ze-de-surjeq} with $Y = \T C$, and a triangle identity for $\eta$ and $\epz$.
  \begin{equation}\label{eq:de-zeflat}
    \de \ze^\flat = \epz_{\T C} \circ (\T \usf \de) \circ (\T \usf \ze) \circ (\T \eta_C) = 1_{\T C}.
  \end{equation}

  Next, define
  \begin{equation}\label{eq:Ga-ze}
    \Ga_{\ze} = \Theta * \ze^\flat\cn \ze \fto{\iso} \ze^\flat
  \end{equation}
  as shown in the following diagram.
  Here and below, we omit the notation $\itt$, as discussed in \cref{convention:ignorei}.
  \[
    \begin{tikzpicture}[x=18ex,y=8ex,xscale=1.2]
      \draw[0cell] 
      (0,0) node (a) {\T C}
      (a)++(1,0) node (b) {\Q\T C}
      (b)++(-.5,-1) node (c) {\T C}
      (c)++(1,0) node (d) {\Q\T C}
      ;
      \draw[1cell] 
      (a) edge node {\ze^\flat} (b)
      (b) edge['] node {\de} (c)
      (a) edge['] node {1} (c)
      (b) edge node (one) {1} (d)
      ;
      \draw[zz1cell]
      (c) -- node['] (zeta) {\ze} (d)
      ;
      \draw[2cell]
      node[between=b and zeta at .5, rotate=35, 2label={above,\Theta}] {\Rightarrow}
      ;
    \end{tikzpicture}
  \]
  The $\T$-algebra 2-cell isomorphism $\Ga_\ze$ has the following two properties.
  \begin{enumerate}
  \item The following diagram commutes in $\K$.
    \begin{equation}\label{eq:Ga-ze-eta-1}
      \begin{tikzpicture}[x=12ex,y=8ex,vcenter,xscale=1.2]
        \draw[0cell] 
        (0,0) node (a) {C}
        (0,-1) node (b) {\usf \T C}
        (1,0) node (c) {\usf \T C}
        (1,-1) node (d) {\usf \Q\T C}
        ;
        \draw[1cell] 
        (a) edge['] node {\eta_C} (b)
        (a) edge node {\eta_C} (c)
        (c) edge node {\usf \ze^\flat} (d)
        ;
        \draw[zz1cell]
        (b) to node {\usf \ze} (d)
        ;
      \end{tikzpicture}
    \end{equation}
    This holds by definition of $\ze^\flat$ as the mate of $\usf \ze \circ \eta_C$ \cref{eq:pre-zetaflat}.
    Let $\ess$ denote the two equal composites in \cref{eq:Ga-ze-eta-1}:
    \begin{equation}\label{eq:s-square}
      \ess = \usf \zeta \circ \eta_C = \usf \ze^\flat \circ \eta_C.
    \end{equation}
  \item The whiskering $\usf \Ga_\ze * \eta_C$ is equal to the identity 2-cell in $\K$ of the 1-cell $\ess$ \cref{eq:s-square}.
    This follows from the definition of $\Ga_\ze$ \cref{eq:Ga-ze}, the commutativity of \cref{eq:Ga-ze-eta-1}, and the right hand side of \cref{eq:ze-de-surjeq}:
    \begin{equation}\label{eq:Ga-ze-eta-2}
      \begin{split}
      \usf \Ga_\ze * \eta_C
      & = \usf \Theta * (\ze^\flat \circ \eta_C)\\
      & = \usf \Theta * (\ze \circ \eta_C)\\
      & = 1_\ze * \eta_C = 1_\ess.
    \end{split}
  \end{equation}
  \end{enumerate}
  
  Now we define
  \[
    \Theta^\flat = \Theta \circ \bigl( \Ga^\inv_{\ze} * \de \bigr)
    = \Theta \circ \bigl( \Theta^\inv_{\ze} * (\ze^\flat \de) \bigr)
    \cn \ze^\flat \de \fto{\iso} 1_{\Q\T C}
  \]
  as shown in the following diagram.
  \begin{equation}\label{eq:Thetaflat}
    \begin{tikzpicture}[x=19ex,y=15ex,vcenter,xscale=1.2]
      \draw[0cell] 
      (0,0) node (a) {\Q\T C}
      (a)++(.25,1) node (b) {\T C}
      (b)++(1,0) node (c) {\Q\T C}
      (c)++(.25,-1) node (d) {\Q\T C}
      (a)++(.75,.5) node (x) {\T C}
      ;
      \draw[1cell] 
      (a) edge node {\de} (b)
      (b) edge node {\ze^\flat} (c)
      (c) edge node (cd) {1} (d)
      (a) edge['] node (ad) {1} (d)
      (a) edge node {\de} (x)
      (b) edge['] node {1} (x)
      (c) edge['] node {\de} (x)
      ;
      \draw[zz1cell]
      (x) -- node {\ze} (d)
      ;
      \draw[2cell=.9] 
      node[between=x and cd at .6, rotate=180, 2label={below,\ \Theta^\inv}] {\Rightarrow}
      node[between=x and ad at .5, rotate=270, 2label={above,\,\Theta}] {\Rightarrow}
      ;
    \end{tikzpicture}
  \end{equation}
  Using the definition of $\Theta^\flat$ and \cref{eq:de-zeflat}, we have
  \begin{align*}
    \Theta^\flat * \ze^\flat
    & = \bigl( \Theta * \ze^\flat \bigr) \circ \bigl( \Theta^\inv * (\ze^\flat \de \ze^\flat) \bigr)\\
    & =  \bigl( \Theta * \ze^\flat \bigr) \circ \bigl( \Theta^\inv * \ze^\flat \bigr)\\
    & = 1_{\ze^\flat}.
  \end{align*}
  This completes the proof that $(\ze^\flat, \de, \Theta^\flat)$ is an adjoint surjective equivalence in $\Talgs$.
\end{proof}

\begin{rmk}\label{rmk:free-flex}
  Recalling \cref{note:delta-Talg}, the conclusion of \cref{lem:zeflat} implies that each free $\T$-algebra $\T C$ is flexible.
  Beyond this, it identifies the adjoint surjective equivalence $(\ze_{\T C}^\flat, \de_{\T C}, \Theta^\flat)$ that will be necessary in \cref{sec:finitary-udc,sec:De1-equiv} below.
  Moreover, \Cref{lem:flex} makes use of $\Ga_\ze$ and the two properties noted in \cref{eq:Ga-ze-eta-1,eq:Ga-ze-eta-2}.
\end{rmk}

\begin{lem}\label{lem:flex}
  Suppose given an object $C \in \K$ and a $\T$-algebra $Y$ together with a $\T$-map
  \[
    \psi \cn \T C \zzto Y.
  \]
  Then there is a unique pair $(\psi^\flat, \Ga_\psi)$ consisting of a strict $\T$-map $\psi^\flat$ together with an invertible $\T$-algebra 2-cell $\Ga_\psi$
  \[
    \psi^\flat \cn \T C \to Y
    \andspace
    \Ga_{\psi} \cn \psi \fto{\iso} \psi^\flat
  \]
  such that the following statements hold.
  \begin{enumerate}
  \item\label{lem:flex:it1} The following diagram commutes in $\K$;
    \begin{equation}\label{eq:Ga-ze-eta-1-psi}
      \begin{tikzpicture}[x=10ex,y=8ex,vcenter,xscale=1.2]
        \draw[0cell] 
        (0,0) node (a) {C}
        (0,-1) node (b) {\usf \T C}
        (1,0) node (c) {\usf \T C}
        (1,-1) node (d) {\usf Y}
        ;
        \draw[1cell] 
        (a) edge['] node {\eta_C} (b)
        (a) edge node {\eta_C} (c)
        (c) edge node {\usf \psi^\flat} (d)
        ;
        \draw[zz1cell]
        (b) to node {\usf \psi} (d)
        ;
      \end{tikzpicture}
    \end{equation}
    let $\ess_\psi$ denote either of the two equal composites in \cref{eq:Ga-ze-eta-1-psi}:
    \begin{equation}\label{eq:s-square-psi}
      \ess_\psi = \usf \psi \circ \eta_C = \usf \psi^\flat \circ \eta_C.
    \end{equation}
  \item\label{lem:flex:it2} The whiskering $\usf \Ga_\psi * \eta_C$ is equal to the identity 2-cell in $\K$ of the 1-cell $\ess_\psi$ \cref{eq:s-square-psi}.
  \end{enumerate} 
\end{lem}
\begin{proof}
  In the case $Y = \Q\T C$ and $\psi = \ze_{\T C}\cn \T C \zzto \Q\T C$, the proof of \cref{lem:zeflat} defines $\ze_{\T C}^\flat$ and $\Ga_{\ze_{\T C}}$ in \cref{eq:zeflat,eq:Ga-ze}.
  The desired conditions are \cref{eq:Ga-ze-eta-1,eq:Ga-ze-eta-2}.

  For general $\psi\cn \T C \zzto Y$, let $\psi^\bot \cn \Q\T C \to Y$ be the strict $\T$-map factoring $\psi$, as in \cref{eq:fbot}.
  This provides the commutative triangle at right in the diagram below.
  \[
    \begin{tikzpicture}[x=17ex,y=12ex,vcenter,xscale=1.2]
      \draw[0cell] 
      (0,0) node (iqx) {\itt \Q\T C}
      (iqx)++(0,-1) node (x) {\T C}
      (iqx)++(1,0) node (x') {Y}
      ;
      \draw[1cell] 
      (iqx) edge node {\psi^\bot} (x')
      (x) edge[bend left=31,looseness=1.5] node[pos=.52] (zef) {\ze_{\T C}^\flat} (iqx)
      ;
      \draw[zz1cell]
      (x) to['] node {\psi} (x')
      ;

      \draw (x) ++(.25,.5) ++(-3pt,0) node (M0) {};
      \draw[1cell]
      (x) to [bend right=5] (M0)
      decorate [decoration={zigzag, segment length=\midzzSegLen, amplitude=\midzzAmp}] 
      { to ++(0,\midzzTotLen) } node[right,shift={(0,2pt)}] (ze) {\ze_{\T C}} to [bend right=8,in=179] (iqx)
      ;

      \draw[2cell]
      node[between=zef and ze at .5, shift={(0,-1ex)}, rotate=180, 2label={below,\Ga_{\ze_{\T C}}}] {\Rightarrow}
      ;
    \end{tikzpicture}
  \]
  We now define
  \[
    \psi^\flat = \psi^\bot \circ \ze_{\T C}^\flat
    \andspace
    \Ga_{\psi} = \psi^\bot * \Ga_{\ze_{\T C}}.   
  \]
  Thus, $\Ga_\psi$ provides a $\T$-algebra 2-cell isomorphism
  \[
    \psi = \psi^\bot \ze_{\T C} \fto[\iso]{\Ga_{\psi}} \psi^\bot \ze_{\T C}^\flat = \psi^\flat
  \]
  as desired.
  The required conditions for $\psi^\flat$ and $\Ga_\psi$ now follow from the corresponding ones for $\ze^\flat$ and $\Ga_\ze$ in \cref{eq:Ga-ze-eta-1,eq:Ga-ze-eta-2}:
  \begin{equation*}
    \begin{split}
      \bigl(\usf \psi^\flat\bigr) \, \eta_C
      & = \bigl(\usf \psi^\bot\bigr) \, \bigl(\usf \ze^\flat\bigr) \, \eta_C\\
      & = \bigl(\usf \psi^\bot\bigr) \, \bigl(\usf \ze\bigr) \, \eta_C\\ 
      & = \bigl(\usf \psi\bigr) \, \eta_C\\ 
    \end{split}
    \qquad
    \andspace
    \qquad
    \begin{split}
      \bigl(\usf \Ga_\psi\bigr) * \eta_C
      & = \bigl(\usf \psi^\bot\bigr) * \bigl(\usf \Ga_{\ze_{\T C}}\bigr) * \eta_C\\
      & = \bigl(\usf \psi^\bot\bigr) * 1_\ess\\ 
      & = 1_{\ess_\psi},\\ 
    \end{split}
  \end{equation*}
  where $\ess$ and $\ess_\psi$ are the composites in \cref{eq:s-square,eq:s-square-psi}, respectively.
  This completes the proof.
\end{proof}

\artpart{Universal pseudomorphisms}

\section{Universal pseudomorphisms}\label{sec:upc}

In this section we provide the definition and basic properties of universal pseudomorphisms
\[
  \wt{\phi} \cn \T C \to \T(C',\phi)
\]
for 1-cells $\phi \cn C \to C'$ in $\K$.
Recall from \cref{notn:arrow} that $\bii = \{0 \to 1\}$ denotes the free arrow category.

\begin{defn}\label{defn:Talgii}
  Suppose $\T = (\T,\mu,\eta)$ is a 2-monad on a 2-category $\K$.
  Recall the free/forgetful adjunction $\T \dashv \usf$ from \cref{eq:Tu-adj}.
  \begin{description}
  \item[Arrow category:]
    The \emph{arrow category} of $\K$ is denoted $\K^\bii$.
    Its objects are 1-cells $\phi \cn C \to C'$ in $\K$ and its morphisms $(R,S)\cn \phi \to \psi$ are pairs of 1-cells such that $\psi R = S \phi$ in $\K$, as in the diagram at left in \cref{eq:arrowcats} below.
  \item[Strict arrow category of $\T$-maps:]
    The \emph{strict arrow category of $\T$-maps} is denoted $\Talgii$. 
    Its objects are $\T$-maps $f\cn X \zzto X'$ in $\Talg$ and its morphisms $(j,k)\cn f \to g$ are pairs of strict $\T$-maps such that $jf = gk$ in $\Talg$, as in the diagram at right in \cref{eq:arrowcats} below.
  \end{description}
  \begin{equation}\label{eq:arrowcats}
    \begin{tikzpicture}[x=13ex,y=8ex,vcenter]
      \draw[0cell] 
      (0,0) node (c) {C}
      (c)++(1,0) node (c') {C'}
      (0,-1) node (d) {D}
      (d)++(1,0) node (d') {D'}
      ;
      \draw[1cell] 
      (c) edge node {\phi} (c')
      (d) edge node {\psi} (d')
      (c) edge['] node {R} (d)
      (c') edge node {S} (d')
      ;
    \end{tikzpicture}
    \qquad \qquad
    \begin{tikzpicture}[x=13ex,y=8ex,vcenter]
      \draw[0cell] 
      (0,0) node (c) {X}
      (c)++(1,0) node (c') {X'}
      (0,-1) node (d) {Y}
      (d)++(1,0) node (d') {Y'}
      ;
      \draw[1cell] 
      (c) edge['] node {j} (d)
      (c') edge node {k} (d')
      ;
      \draw[zz1cell]
      (c) to node {f} (c')
      ;
      \draw[zz1cell]
      (d) to node {g} (d')
      ;
    \end{tikzpicture}
  \end{equation}
  The forgetful $\usf \cn \Talg \to \K$ induces a functor on arrow categories that we also denote
  \begin{equation}\label{eq:usfbii}
    \usf\cn \Talgii \to \K^\bii.
  \end{equation}
\end{defn}

\begin{rmk}\label{rmk:Kbii-Talgii-2cats}
  Both $\K^\bii$ and $\Talgii$ are the underlying 1-categories of 2-categories, with 2-cells given by pairs of 2-cells in $\K$ and $\Talgs$,
  \[
    (\Gamma,\Omega)\cn (R,S) \to (R',S')
    \andspace
    (\al,\ga)\cn (j,k) \to (j',k'),
  \]
  respectively, that satisfy equalities as in \cref{eq:arrowcats}:
  \[
    \Omega * \phi = \psi * \Gamma
    \andspace
    \ga * f = g * \al.
  \]
  Most of our discussion below will restrict to the underlying 1-categories as written in \cref{defn:Talgii}, but we will refer to the ambient 2-categories using the same notation in \cref{lem:2Dupc} below.
\end{rmk}

\begin{defn}\label{defn:udc}
  In the context of \Cref{defn:Talgii}, we say that $\T$ admits \emph{universal pseudomorphisms} if, for each 1-cell
  \[
    \phi \cn C \to C' \inspace \K
  \]
  there is a $\T$-algebra $\T(C',\phi)$ and $\T$-map
  \begin{equation}\label{eq:phitilde}
    \wt{\phi} \cn \T C \zzto \T(C',\phi) \inspace \Talg
  \end{equation}
  together with a \emph{unit morphism} in $\K^{\bii}$
  \begin{equation}\label{eq:etakappa}
    (\eta_C,\ka_{\phi})\cn \phi \to \usf\wt{\phi}
    \inspace \K^\bii,
  \end{equation}
  where $\eta$ is the unit structure transformation of $\T = (\T,\mu,\eta)$, such that the following holds.
  \begin{description}
  \item[Universal property:]
    For each $\T$-map $f\cn X \zzto Y$ there is a bijection of sets
    \begin{equation}\label{eq:univprop}
      \Talgii(\wt{\phi},f)
      \fto{\iso}
      \K^\bii(\phi,\usf f)
    \end{equation}
    induced by $\usf$ and composition with $(\eta_C,\ka_{\phi})$.
  \end{description}
  In this case, we say that $\wt{\phi}\cn \T C \zzto \T(C',\phi)$ is the \emph{universal pseudomorphism} for $\phi$.
\end{defn}

\begin{rmk}\label{rmk:udc}
  In the context of \cref{defn:udc}, the universal property \cref{eq:univprop} is equivalent to the following.
  For each $f\cn X \zzto X'$ in $\Talg$ and each pair of 1-cells $R$ and $S$ such that $(R,S)\cn \phi \to \usf f$ in $\K^\bii$, there are unique strict $\T$-maps $\ol{R}$ and $\ol{S}$ so that $(\ol{R},\ol{S})\cn \wt{\phi} \to f$ in $\Talgii$ and the diagram below commutes in $\K$.
  \begin{equation}\label{eq:rmk-udc}
    \begin{tikzpicture}[x=34ex,y=8ex,vcenter,xscale=1.2]
      \draw[0cell] 
      (0,0) node (c) {C}
      (c)++(1,0) node (c') {C'}
      (c)++(.3,-1) node (kc) {\usf \T C}
      (c')++(-.3,-1) node (kc') {\usf \T (C',\phi)}
      (c)++(0,-2) node (x) {\usf X}
      (c')++(0,-2) node (x') {\usf X'}
      ;
      \draw[1cell] 
      (c) edge node {\phi} (c')
      (c) edge node[pos=.55] {\eta_C} (kc)
      (c') edge['] node[pos=.6] {\ka_{\phi}} (kc')
      (c) edge['] node {R} (x)
      (c') edge node {S} (x')
      (kc) edge[',dashed] node[pos=.3] {\usf \ol{R}} node[',pos=.3] {\exists !} (x)
      (kc') edge[dashed] node[pos=.3] {\usf \ol{S}} node[',pos=.3] {\exists !} (x')
      ;
      \draw[zz1cell]
      (x) to node {\usf f} (x')
      ;
      \draw[zz1cell]
      (kc) to node {\usf \wt{\phi}} (kc')
      ;
    \end{tikzpicture}
  \end{equation}
  Observe that uniqueness and commutativity of the triangle at left above implies that $\ol{R}$ depends only on $R$.
  In contrast, $\ol{S}$ depends on both $S$ and $\phi$.
\end{rmk}

\begin{rmk}\label{rmk:set-bij}
  For a universal pseudomorphism $\wt{\phi}$, the bijection of sets \cref{eq:univprop} implies a certain 1-categorical adjunction that we explain in \cref{lem:etat-epzt} below.
  Then, \cref{lem:2Dupc} shows, under mild additional hypotheses, that the adjunction extends to a 2-adjunction.
  However, as we explain further in \cref{rmk:2Dupc}, such an extension is (a) not needed for this work and (b) more difficult to verify in practice.
  These are the reasons that the universal property \cref{eq:univprop} is defined as a mere bijection of sets.
\end{rmk}

\begin{notn}\label{rmk:ka-adj}
  In the context of \cref{defn:udc,rmk:udc}, the mate of $\ka_\phi$ under the adjunction $\T \dashv \usf$ is denoted $\ka$ and is uniquely determined such that the following commutes.
  \begin{equation}\label{eq:ka-adj}
    \begin{tikzpicture}[x=18ex,y=12ex,vcenter,xscale=1.2]
      \draw[0cell] 
      (0,0) node (a) {C'}
      (a)++(1,0) node (c) {\usf \T(C',\phi)}
      (a)++(.5,-.5) node (b) {\usf \T C'}
      ;
      \draw[1cell] 
      (a) edge['] node {\eta_{C'}} (b)
      (b) edge['] node {\usf \ka} (c)
      (a) edge node {\ka_\phi} (c)
      ;
    \end{tikzpicture}
  \end{equation}
  \ 
\end{notn}

Recall from \cref{defn:Tu-adj} that $\eta$ and $\epz$ denote, respectively, the unit and counit of the adjunction $\T \dashv \usf$.
\begin{lem}\label{lem:etat-epzt}
  Suppose 
   \begin{itemize}
  \item $C,C'$ are objects of $\K$,
  \item $\phi \cn C \to C'$ is an object of $\K^\bii$,
  \item $X,X'$ are objects of $\Talg$, and
  \item $f\cn X \zzto X'$ is an object of $\Talgii$.
  \end{itemize}
  In the context of \cref{defn:udc}, the assignment
  \[
    \phi \mapsto \wt{\phi}
  \]
  is functorial with respect to morphisms in $\K^\bii$ and is left adjoint to the forgetful $\usf\cn \Talgii \to \K^\bii$ from \cref{eq:usfbii}.
  The unit and counit of the adjunction $\wt{(-)} \dashv \usf$ are given, respectively, by
  \begin{equation}\label{eq:etat-epzt}
    \wt{\eta}_{\phi} = (\eta_C,\ka_\phi)
    \andspace
    \wt{\epz}_{f} = (\epz_{X},\ol{1_{\usf X'}}).
  \end{equation}
\end{lem}
\begin{proof}
  First we define $\wt{(-)}$ on morphisms of $\K^\bii$.
  Suppose that
  \[
    (R,S)\cn \phi \to \psi
  \]
  is a morphism of $\K^\bii$, where
  \begin{align*}
    \phi \cn C \to C', 
    & \quad \psi \cn D \to D',\\
    R\cn C \to D,
    & \andspace S \cn C' \to D'
  \end{align*}
  are 1-cells of $\K$.
  Recall from \cref{eq:etakappa} the unit morphisms for $\phi$ and $\psi$ are
  \[
    (\eta_C,\ka_\phi) \cn \phi \to \wt{\phi}
    \andspace
    (\eta_D,\ka_\psi)\cn \psi \to \wt{\psi}.
  \]
  We now use the universal property \cref{eq:univprop} of $\wt{(-)}$, in the form described in \cref{rmk:udc}.
  Composition in $\K^\bii$ yields the outer vertical morphisms in the diagram below, and the universal property gives the two dashed extensions such that the diagram commutes in $\K$.
  \begin{equation}\label{eq:RStilde}
    \begin{tikzpicture}[x=43ex,y=6ex,vcenter,xscale=1.2]
      \draw[0cell] 
      (0,0) node (c) {C}
      (c)++(1,0) node (c') {C'}
      (c)++(.35,-1.3) node (kc) {\usf \T C}
      (c')++(-.35,-1.3) node (kc') {\usf \T (C',\phi)}
      (c)++(0,-3) node (x) {\usf \T D}
      (c')++(0,-3) node (x') {\usf \T(D',\psi)}
      (c)++(0,-2) node (d) {D}
      (c')++(0,-2) node (d') {D'}
      ;
      \draw[1cell] 
      (c) edge node {\phi} (c')
      (c) edge node[pos=.55] {\eta_C} (kc)
      (c') edge['] node[pos=.6] {\ka_{\phi}} (kc')
      (c) edge['] node {R} (d)
      (c') edge node {S} (d')
      (d) edge['] node {\eta_D} (x)
      (d') edge node {\ka_\psi} (x')
      (kc) edge[',dashed] node[pos=.3] {\usf (\ol{\eta_D R})} node[',pos=.3] {\exists !} (x)
      (kc') edge[dashed] node[pos=.3] {\usf (\ol{\ka_\psi S})} node[',pos=.3] {\exists !} (x')
      ;
      \draw[zz1cell]
      (x) to node {\usf \wt{\psi}} (x')
      ;
      \draw[zz1cell]
      (kc) to node {\usf \wt{\phi}} (kc')
      ;
    \end{tikzpicture}
  \end{equation}
  By uniqueness, we have $\ol{\eta R} = \T R$.
  Thus, $\wt{(-)}$ is defined on morphisms by
  \[
    \wt{R} = \ol{\eta_D R} = \T R
    \andspace
    \wt{S} = \ol{\ka_\psi S}.
  \]
  Uniqueness shows that this assignment is functorial, and commutativity of the triangles at left and right of \cref{eq:RStilde} shows that the components $(\eta_C,\ka_\phi)$ define a natural transformation
  \[
    1_{\K^\bii} \to \usf \wt{(-)}.
  \]
  This justifies the name \emph{unit} for $(\eta_C,\ka_\phi)$ in \cref{defn:udc} and we define
  \[
    \wt{\eta}_\phi = (\eta_C,\ka_{\phi}).
  \]
  
  If $f\cn X \zzto X'$ is a $\T$-map, we define the counit component
  \[
    \wt{\epz}_f = (\epz_X,\ol{1_{\usf X'}}).
  \]
  Naturality of $\wt{\epz}$ with respect to morphisms $(j,k)\cn f \to g$ in $\Talgii$ follows from uniqueness in the universal property \cref{eq:univprop}.

  The triangle identities for $\wt{\eta}$ and $\wt{\epz}$ follow from the definitions and the triangle identities for $\eta$ and $\epz$.
  This completes the proof that there is an adjunction $\wt{(-)} \dashv \usf$ with unit and counit given by \cref{eq:etat-epzt}.
\end{proof}

\begin{defn}\label{defn:source}
  Define the \emph{source functors}
  \[
    \ssf\cn \K^{\bii} \to \K
    \andspace
    \ssf\cn \Talgii \to \Talgs 
  \]
  by the assignments
  \begin{align*}
    \ssf(\phi)
    = C \qquad
    & \ssf(R,S)
    = R
    \\
    \ssf(f)
    = X \qquad
    & \ssf(j,k)
    = j
  \end{align*}
  where
  \[
    \phi \cn C \to C', \qquad (R,S)\cn \phi \to \psi
  \]
  are 0-, respectively 1-cells in $\K^\bii$, and 
  \[
    f\cn X \zzto X', \qquad (j,k)\cn f \to g
  \]
  are 0-, respectively 1-cells in $\Talgii$.
\end{defn}

The next result follows from \cref{lem:etat-epzt} along with \cref{defn:Talgii,defn:udc,defn:source}.
\begin{prop}\label{prop:source}
  In the context of \cref{defn:udc}, the following diagram of adjunctions serially commutes.
  \[
    \begin{tikzpicture}[x=20ex,y=13ex]
      \draw[0cell] 
      (0,0) node (a) {\Talgii}
      (a)+(1,0) node (b) {\Talgs}
      (a)+(0,-1) node (c) {\K^\bii}
      (b)+(0,-1) node (d) {\K}
      ;
      \draw[1cell] 
      (a) edge node {\ssf} (b)
      (c) edge node {\ssf} (d)
      (c) edge[bend left=15,transform canvas={xshift=-.7mm}] node (L) {\wt{(-)}} (a) 
      (a) edge[bend left=15,transform canvas={xshift=.7mm}] node (R) {\usf} (c) 
      (d) edge[bend left=15,transform canvas={xshift=-.7mm}] node (L') {\T} (b) 
      (b) edge[bend left=15,transform canvas={xshift=.7mm}] node (R') {\usf} (d) 
      ;
      \draw[2cell] 
      node[between=L and R at .5, rotate=0] {\dashv}
      node[between=L' and R' at .5, rotate=0] {\dashv}
      ;
    \end{tikzpicture}
  \]
  That is, the following equalities hold: 
  \begin{align*}
    \ssf \usf
    & = \usf \ssf
    & \ssf \wt{(-)}
    & = \T \ssf \\
    \ssf * \wt{\eta}
    & = \eta * \ssf
    & \ssf * \wt{\epz}
    & = \epz * \ssf.
  \end{align*}
\end{prop}

For the next result, we let $\K^\bii$ and $\Talgii$ denote the ambient 2-categories, as described in \cref{rmk:Kbii-Talgii-2cats}.
The following 2-dimensional extension is included for completeness and context, but will not be necessary in our further work; \cref{rmk:2Dupc} gives further explanation.
\begin{lem}\label{lem:2Dupc}
  In the context of \cref{defn:udc}, suppose furthermore that $\K$ admits cotensors of the form $\{\bii,-\}$. 
  Then the adjunction $\wt{(-)} \dashv \usf$ of \cref{lem:etat-epzt} extends to a 2-adjunction.
\end{lem}
\begin{proof}
  The hypothesis that $\K$ admits cotensors $\{\bii,-\}$ implies, by \cref{prop:cotensor-facts}~\cref{it:cotensor-facts-i} that $\Talg$ and $\Talgs$ both admit those cotensors and that the functors $\itt$ and $\usf$ preserve them.
  The cotensors $\{\bii,-\}$ in $\K$ induce the cotensors $\{\bii,-\}$ in $\K^\bii$ pointwise, and the cotensors $\{\bii,-\}$ in $\Talg$ and $\Talgs$ induce the cotensors $\{\bii,-\}$ in $\Talgii$ pointwise.
  Moreover, $\usf\cn \Talgii \to \K^\bii$ preserves those cotensors.
  Therefore, by \cref{prop:cotensor-facts}~\cref{it:cotensor-facts-ii}, the universal pseudomorphism functor $\wt{(-)}$ extends uniquely to a left 2-adjoint of $\usf$.
\end{proof}

\begin{rmk}\label{rmk:2Dupc}
  As written in \cref{defn:udc}, the universal property of a universal pseudomorphism is a 1-categorical property.
  The 2-categorical extension that appears in \cref{lem:2Dupc} is not needed for any of our work below.
  It does not appear to simplify any of the proofs of the results used in  \cref{thm:main2}.
  Furthermore, the 1-categorical version is simpler to verify in cases where one proves that a 2-monad has universal pseudomorphisms.
  This occurs, for example, in the proof of \cref{thm:finitary-udc}.
\end{rmk}

Recall from \cref{notn:arrow} that $\II$ denotes the category consisting of two objects and an isomorphism between them.
The following is a generalization of \cite[Lemma~2.22]{Gurski13Coherence}.
\begin{lem}\label{lem:222}
  Suppose that $\T$ is a 2-monad on a 2-category $\K$.
  Suppose that $\T$ admits universal pseudomorphisms (\cref{defn:udc}) and suppose given $C, C' \in \K$ and $X,X' \in \Talg$ together with
  \begin{align*}
    \phi \cn C \to C' & \inspace \K, \\
    R \cn C \to \usf X & \inspace \K, \\
    \be \cn S_1 \to S_2 & \inspace \K(C',\usf X'), \andspace \\
    \al \cn f_1 \to f_2 & \inspace \Talg(X,X')
  \end{align*}
  as shown at left in \cref{eq:lem222} below, such that
  \[
    \be * \phi = (\usf \al) * R.
  \]
  \begin{equation}\label{eq:lem222}
    \begin{tikzpicture}[x=17ex,y=13ex,vcenter,xscale=1.2]
      \draw[0cell] 
      (0,0) node (a) {C}
      (a)++(1,0) node (b) {C'}
      (a)++(0,-1) node (c) {\usf X}
      (b)++(0,-1) node (d) {\usf X'}
      ;
      \draw[1cell] 
      (a) edge node {\phi} (b)
      (a) edge['] node {R} (c)
      (b) edge[bend left=25] node (s2) {S_2} (d)
      (b) edge[',bend right=25] node (s1) {S_1} (d)
      ;

      \draw (c) ++(.5,.2) ++(0pt,0) node (M0) {};
      \draw (c) ++(.5,-.2) ++(0pt,0) node (M1) {}; 
      \draw[1cell]
      (c) to [bend left=5] (M0)
      decorate [decoration={zigzag, segment length=\midzzSegLen, amplitude=\midzzAmp}] 
      { to ++(\midzzTotLen,0) } node[above] (f1) {\usf f_1} to [bend left=5] (d)
      ;
      \draw[1cell]
      (c) to [bend right=5] (M1)
      decorate [decoration={zigzag, segment length=\midzzSegLen, amplitude=\midzzAmp}] 
      { to ++(\midzzTotLen,0) } node[below] (f2) {\usf f_2} to [bend right=5] (d)
      ;

      \draw[2cell]
      node[between=f1 and f2 at .5, shift={(-3pt,0)}, rotate=-90, 2label={above,\usf\al}] {\Rightarrow}
      node[between=s1 and s2 at .5, rotate=0, 2label={above,\be}] {\Rightarrow}
      ;
    \end{tikzpicture}
    \qquad\qquad
    \begin{tikzpicture}[x=17ex,y=13ex,vcenter,xscale=1.2]
      \draw[0cell] 
      (0,0) node (a) {\T C}
      (a)++(1,0) node (b) {\T (C',\phi)}
      (a)++(0,-1) node (c) {X}
      (b)++(0,-1) node (d) {X'}
      ;
      \draw[1cell] 
      (a) edge['] node {\ol{R}} (c)
      (b) edge[bend left=25] node (s2) {\ol{S}_2} (d)
      (b) edge[',bend right=25] node (s1) {\ol{S}_1} (d)
      ;

      \draw[zz1cell]
      (a) to node {\wt{\phi}} (b)
      ;

      \draw (c) ++(.5,.2) ++(0pt,0) node (M0) {};
      \draw (c) ++(.5,-.2) ++(0pt,0) node (M1) {}; 
      \draw[1cell]
      (c) to [bend left=5] (M0)
      decorate [decoration={zigzag, segment length=\midzzSegLen, amplitude=\midzzAmp}] 
      { to ++(\midzzTotLen,0) } node[above] (f1) {\!f_1} to [bend left=5] (d)
      ;
      \draw[1cell]
      (c) to [bend right=5] (M1)
      decorate [decoration={zigzag, segment length=\midzzSegLen, amplitude=\midzzAmp}] 
      { to ++(\midzzTotLen,0) } node[below] (f2) {\!f_2} to [bend right=5] (d)
      ;

      \draw[2cell]
      node[between=f1 and f2 at .5, shift={(-2pt,0)}, rotate=-90, 2label={above,\al}] {\Rightarrow}
      node[between=s1 and s2 at .5, rotate=0, 2label={above,\ol{\be}}] {\Rightarrow}
      ;
    \end{tikzpicture}
  \end{equation}
  Then the following statements hold.
  \begin{enumerate}
  \item\label{it:222-bii} If $\K$ admits cotensors of the form $\{\bii,-\}$, then there is a unique $\T$-algebra 2-cell $\ol{\be}\cn \ol{S}_1 \to \ol{S}_2$ at right in \cref{eq:lem222} such that
  \[
    \ol{\be} * \wt{\phi} = \al * \ol{R}.
  \]
  Here, for $i = 1,2$, 
  \[
    (\ol{R},\ol{S}_i)\cn \wt{\phi} \to f_i
  \]
  is the pair of unique strict $\T$-maps determined by the universal property \cref{eq:univprop} of $\wt{\phi}$.
\item\label{it:222-II} If $\K$ admits cotensors of the form $\{\II, -\}$, and if $\al$ and $\be$ are invertible, then there is a unique $\T$-algebra 2-cell $\ol{\be}$ as above, and $\ol{\be}$ is invertible.
  \end{enumerate}
\end{lem}
\begin{proof}
  We begin with the first assertion.
  Recalling \cref{prop:cotensor-facts}~\cref{it:cotensor-facts-i}, the assumption that $\K$ admits cotensors $\{\bii, -\}$ implies the same for both $\Talgs$ and $\Talg$.
  Furthermore, the inclusion $\itt$ and the forgetful functors $\usf$ preserve those cotensors.
  
  Using the fact that $\usf$ preserves cotensor products and unpacking the definition of cotensors for $\{\bii, X'\}$, as in \cref{rmk:special-cotensors}, the diagrams in \cref{eq:lem222} correspond to the diagrams in \cref{eq:lem222-2} below, with $F$ and $\usf F$ corresponding, respectively, to the triples $(f_1, f_2, \al)$ and $(\usf f_1, \usf f_2, \usf \al)$.
  Likewise, $S$ and $\ol{S}$ correspond, respectively, to the triples $(S_1, S_2, \be)$ and $(\ol{S}_1, \ol{S}_2, \ol{\be})$.
  \begin{equation}\label{eq:lem222-2}
    \begin{tikzpicture}[x=12ex,y=8ex,vcenter,xscale=1.2]
      \draw[0cell] 
      (0,0) node (a) {C}
      (a)++(1,0) node (b) {C'}
      (a)++(0,-1) node (c) {\usf X}
      (b)++(0,-1) node (d) {\{\bii,\usf X'\}}
      ;
      \draw[1cell] 
      (a) edge node {\phi} (b)
      (a) edge['] node {R} (c)
      (b) edge node {S} (d)
      ;
      \draw[zz1cell]
      (c) to node {\usf F} (d)
      ;
    \end{tikzpicture}
    \qquad\qquad
    \begin{tikzpicture}[x=11ex,y=8ex,vcenter,xscale=1.2]
      \draw[0cell] 
      (0,0) node (a) {\T C}
      (a)++(1,0) node (b) {\T (C',\phi)}
      (a)++(0,-1) node (c) {X}
      (b)++(0,-1) node (d) {\{\bii,X'\}}
      ;
      \draw[1cell] 
      (a) edge['] node {\ol{R}} (c)
      (b) edge node {\ol{S}} (d)
      ;
      \draw[zz1cell]
      (a) to node {\wt{\phi}} (b)
      ;
      \draw[zz1cell]
      (c) to node {F} (d)
      ;
    \end{tikzpicture}
  \end{equation}

  With this reformulation, the assertion \cref{it:222-bii} follows by applying the universal property \cref{eq:RStilde} to $(R,S)$ at left in \cref{eq:lem222-2}: there is a unique morphism $(\overline{R}, \overline{S}) \cn \wt{\phi} \zzto F$ in $\Talgii$, at right in \cref{eq:lem222-2}, such that
  \[
    (R, S) = \usf(\overline{R}, \overline{S}) \circ (\eta_C, \kappa_{\phi}).
  \]
  Unpacking the above equation, via \cref{rmk:special-cotensors} and a similar description of 2-cells in \cref{eq:cotensor}, yields the data $(\ol{S_1},\ol{S_2},\ol{\be})$ at right in \cref{eq:lem222} satisfying the desired equations.
  Using cotensors with the category $\II = \{ 0 \cong 1 \}$ instead of $\bii$ yields the second assertion, in which all of the 2-cells are invertible.
\end{proof}

\begin{defn}\label{defn:Delta}
  Suppose $\T$ is a 2-monad on $\K$ that admits universal pseudomorphisms.
  For each 1-cell $\phi\cn C \to C'$ in $\K$, define a strict $\T$-map
  \begin{equation}\label{eq:Delta}
    \De = \ol{\eta_{C'}} \cn \T(C',\phi) \to \T C'
  \end{equation}
  as follows.
  The unit $\eta$ defines a morphism
  \[
    (\eta_C,\eta_{C'})\cn \phi \to \usf \T\phi
    \inspace \K^\bii.
  \]
  Therefore, by the universal property \cref{eq:univprop} there is a unique morphism $(\ol{\eta_C},\ol{\eta_{C'}})$ in $\Talgii$ as shown in the diagram below.
  By uniqueness, $\ol{\eta_C}$ is the identity $1_{\T C}$.
  \begin{equation}\label{eq:Deltadiagram}
    \begin{tikzpicture}[x=43ex,y=8ex,vcenter,xscale=1.2]
      \draw[0cell] 
      (0,0) node (c) {C}
      (c)++(1,0) node (c') {C'}
      (c)++(.3,-1) node (kc) {\usf \T C}
      (c')++(-.3,-1) node (kc') {\usf \T (C',\phi)}
      (c)++(0,-2) node (x) {\usf \T C}
      (c')++(0,-2) node (x') {\usf \T C'}
      ;
      \draw[1cell] 
      (c) edge node {\phi} (c')
      (x) edge node {\usf \T\phi} (x')
      (c) edge node[pos=.55] {\eta_C} (kc)
      (c') edge['] node[pos=.6] {\ka_{\phi}} (kc')
      (c) edge['] node {\eta_C} (x)
      (c') edge node {\eta_{C'}} (x')
      (kc) edge[dashed] node[pos=.3] {\usf \ol{\eta_C} = 1} node[',pos=.3] {\exists !} (x)
      (kc') edge[',dashed] node[pos=.3] {\usf\De = \usf \ol{\eta_{C'}}} node[',pos=.3] {\exists !} (x')
      ;
      \draw[zz1cell]
      (kc) to node {\usf \wt{\phi}} (kc')
      ;
    \end{tikzpicture}
  \end{equation}
  Define $\De = \ol{\eta_{C'}}$.
\end{defn}

\section{Universal pseudomorphisms via pushouts}
\label{sec:finitary-udc}

Throughout \cref{sec:finitary-udc,sec:De1-equiv}, $\T$ is assumed to have an effective pseudomorphism classifier (\cref{defn:effectiveQi}).
The goal of this section is to show that universal pseudomorphisms for $\T$ (\cref{defn:udc}) can be constructed as pushouts in $\Talgs$.
In \cref{sec:applications} we explain applications in the case that $\T$ is one of the 2-monads for strict monoidal structures (\cref{notn:monoidal-variants}).

\begin{defn}\label{defn:finitary-phitilde}
  Suppose $\T$ is a 2-monad on a 2-category $\K$ such that $\T$ has an effective pseudomorphism classifier and $\Talgs$ admits pushouts.
  For each 1-cell $\phi \cn C \to C'$ in $\K$, define a $\T$-algebra $\T(C',\phi)$ together with
  \begin{itemize}
  \item a $\T$-map $\wt{\phi} \cn \T C \zzto \T(C',\phi)$ and
  \item a 1-cell $\ka_\phi \cn C' \to \usf \T(C',\phi)$ in $\K$.
  \end{itemize}
  as follows.

  The unit of \cref{eq:Qi-adj} is a $\T$-map
  \begin{equation}\label{eq:zetaTC}
    \ze_{\T C}\cn \T C \zzto \itt \Q\T C.
  \end{equation}
  By \cref{lem:flex}, $\ze_{\T C}$ is isomorphic to a unique strict $\T$-map
  \begin{equation}\label{eq:olzetaTC}
    \ze^\flat \cn \T C \to \itt \Q\T C
  \end{equation}
  such that the diagram below commutes.
  \begin{equation}\label{eq:olze}
    \begin{tikzpicture}[x=14ex,y=8ex,vcenter,xscale=1.2]
      \draw[0cell] 
      (0,0) node (a) {C}
      (a)++(0,-1) node (b) {\usf \T C}
      (a)++(1,0) node (c) {\usf \T C}
      (b)++(1,0) node (d) {\usf \itt \Q\T C}
      ;
      \draw[1cell] 
      (a) edge['] node {\eta_C} (b)
      (a) edge node {\eta_C} (c)
      (c) edge node {\usf \ze^\flat} (d)
      ;
      \draw[zz1cell]
      (b) to node {\usf \ze_{\T C}} (d)
      ;
    \end{tikzpicture}
  \end{equation}

  Define $\T(C',\phi)$ as the pushout in $\Talgs$ of $\ze^\flat$ and $\T\phi$, with structure morphisms $\wh{\phi}$ and $\ka$ as shown in the square below.
  \begin{equation}\label{eq:udc-pushout}
    \begin{tikzpicture}[x=14ex,y=8ex,vcenter,xscale=1.2]
      \draw[0cell] 
      (0,0) node (a) {\T C}
      (a)++(1,0) node (b) {\T C'}
      (a)++(0,-1) node (c) {\itt \Q\T C}
      (c)++(1,0) node (d) {\T(C',\phi)}
      ;
      \draw[1cell] 
      (a) edge node {\T\phi} (b)
      (a) edge['] node {\ze^\flat} (c)
      (c) edge node {\wh{\phi}} (d)
      (b) edge node {\ka} (d)
      ;
    \end{tikzpicture}
  \end{equation}
  Moreover, define $\wt{\phi}$ and $\ka_\phi$ by the following composites in $\Talg$ and $\K$, respectively.
  \begin{equation}\label{eq:wtphi-kaphi}
    \begin{tikzpicture}[x=18ex,y=12ex,vcenter,xscale=1.2,yscale=1.1]
      \draw[0cell] 
      (0,0) node (a) {\T C}
      (a)++(1,0) node (c) {\T(C',\phi)}
      (a)++(.5,-.5) node (b) {\itt \Q\T C}
      ;
      \draw[1cell] 
      (b) edge['] node {\itt\wh{\phi}} (c)
      ;
      \draw[zz1cell]
      (a) to['] node {\ze_{\T C}} (b)
      ;
      \draw[zz1cell]
      (a) to node {\wt{\phi}} (c)
      ;
    \end{tikzpicture}
    \quad\quad
    \begin{tikzpicture}[x=18ex,y=12ex,vcenter]
      \draw[0cell] 
      (0,0) node (a) {C'}
      (a)++(1,0) node (c) {\usf \T(C',\phi)}
      (a)++(.5,-.5) node (b) {\usf \T C'}
      ;
      \draw[1cell] 
      (a) edge['] node {\eta_{C'}} (b)
      (b) edge['] node {\usf \ka} (c)
      (a) edge node {\ka_\phi} (c)
      ;
    \end{tikzpicture}
  \end{equation} 
  This completes the definition of
  \[
    \wt{\phi}\cn \T C \to \T(C',\phi) \inspace \Talg
  \]
  and the unit
  \[
    (\eta_C,\ka_\phi) \cn \phi \to \usf \wt{\phi} \inspace \K^\bii.
  \]
  We show that these satisfy the universal property \cref{eq:univprop} in \cref{thm:finitary-udc} below.
\end{defn}

In the following, we use \cref{convention:ignorei} and implicitly apply the inclusion $\itt$ to compose a general $\T$-map with a strict one.
\begin{lem}\label{lem:Sbar-unique}
  In the context of \cref{defn:finitary-phitilde}, suppose given
  \begin{itemize}
  \item a $\T$-map $f\cn X \zzto X'$ in $\Talg$,
  \item 1-cells $R \cn C \to \usf X$ and $S \cn C' \to \usf X'$ in $\K$, and
  \item strict $\T$-maps
    \begin{align*}
      \ol{R} & \cn \T C \to X\\
      \ol{S}_1,\,\ol{S}_2 & \cn \T(C',\phi) \to X'
    \end{align*}
  \end{itemize}
  such that, for each $i = 1,2$,
  \begin{equation}\label{eq:talg-trapezoid}
    \ol{S}_i \, \wt{\phi} = f \, \ol{R}\cn \T C \to X' \inspace \Talg
  \end{equation}
  and the following diagram commutes in $\K$.
  \begin{equation}\label{eq:Sbar-unique}
    \begin{tikzpicture}[x=34ex,y=8ex,vcenter,xscale=1.2]
      \draw[0cell] 
      (0,0) node (c) {C}
      (c)++(1,0) node (c') {C'}
      (c)++(.3,-1) node (kc) {\usf \T C}
      (c')++(-.3,-1) node (kc') {\usf \T (C',\phi)}
      (c)++(0,-2) node (x) {\usf X}
      (c')++(0,-2) node (x') {\usf X'}
      ;
      \draw[1cell] 
      (c) edge node {\phi} (c')
      (c) edge node[pos=.55] {\eta_C} (kc)
      (c') edge['] node[pos=.6] {\ka_{\phi}} (kc')
      (c) edge['] node {R} (x)
      (c') edge node {S} (x')
      (kc) edge['] node[pos=.3] {\usf \ol{R}} (x)
      (kc') edge node[pos=.3] {\usf \ol{S}_i} (x')
      ;
      \draw[zz1cell]
      (x) to node {\usf f} (x')
      ;
      \draw[zz1cell]
      (kc) to node {\usf \wt{\phi}} (kc')
      ;
    \end{tikzpicture}
  \end{equation}
  Then $\ol{S}_1 = \ol{S}_2$ in $\Talgs$.
\end{lem}
\begin{proof}
  To use the universal property of the pushout \cref{eq:udc-pushout} defining $\T(C',\phi)$, we will show
  \begin{equation}\label{eq:olS1-olS2}
    \ol{S}_1 \ka = \ol{S}_2 \ka
    \andspace
    \ol{S}_1 \wh{\phi} = \ol{S}_2 \wh{\phi}.
  \end{equation} 
  For the first of these, we obtain
  \[
    \usf \bigl( \ol{S}_1 \ka \bigr) \circ \eta_{C'} = S = \usf \bigl( \ol{S}_2 \ka \bigr) \circ \eta_{C'}
  \]
  using 2-functoriality of $\usf$, the definition $\ka_\phi = \usf \ka \circ \eta_{C'}$ from \cref{eq:wtphi-kaphi}, and commutativity of the triangle at right in \cref{eq:Sbar-unique}.
  Then the uniqueness of mates noted in \cref{rmk:LRadj-uniqueness} implies that $\ol{S}_1 \ka = \ol{S}_2 \ka$.

  For the second desired equality in \cref{eq:olS1-olS2}, we obtain
  \[
    \bigl( \ol{S}_1 \wh{\phi} \bigr) \circ \ze_{\T C} = f \, \ol{R} = \bigl( \ol{S}_2 \wh{\phi} \bigr) \circ \ze_{\T C}
  \]
  using the associativity of 1-cell composition, the definition $\wt{\phi} = \wh{\phi} \ze_{\T C}$ in \cref{eq:wtphi-kaphi}, and the hypothesis \cref{eq:talg-trapezoid}.
  Then uniqueness of mates, for the adjunction $(\Q,\itt,\ze,\de)$, implies that $\ol{S}_1 \wh{\phi} = \ol{S}_2 \wh{\phi}$.
  The result $\ol{S}_1 = \ol{S}_2$ then follows from the universal property of the pushout \cref{eq:udc-pushout}.
\end{proof}

\begin{thm}\label{thm:finitary-udc}
  In the context of \cref{defn:finitary-phitilde}, the pushouts $\T(C',\phi)$ in \cref{eq:udc-pushout} determine universal pseudomorphisms for $\T$.
\end{thm}
\begin{proof}
  We show that $\wt{\phi}$ and $\ka_\phi$, as defined in \cref{eq:wtphi-kaphi}, satisfy the universal property \cref{eq:univprop} for each 1-cell
  \[
    \phi \cn C \to C' \inspace \K
  \]
  and each $\T$-map
  \[
    f\cn X \zzto X' \inspace \Talg.
  \]
  For this purpose, suppose given 1-cells $R$ and $S$ in $\K$, as in the outer diagram \cref{eq:univ-toshow} below.
  Following \Cref{rmk:udc}, we will show that there are unique strict $\T$-maps $\ol{R}$ and $\ol{S}$ such that
  \[
    \ol{S} \, \wt{\phi} = f \, \ol{R} \cn \T C \zzto X' \inspace \Talg  
  \]
  and the following diagram commutes in $\K$.
  \begin{equation}\label{eq:univ-toshow}
    \begin{tikzpicture}[x=34ex,y=8ex,vcenter,xscale=1.2]
      \draw[0cell] 
      (0,0) node (c) {C}
      (c)++(1,0) node (c') {C'}
      (c)++(.3,-1) node (kc) {\usf \T C}
      (c')++(-.3,-1) node (kc') {\usf \T (C',\phi)}
      (c)++(0,-2) node (x) {\usf X}
      (c')++(0,-2) node (x') {\usf X'}
      ;
      \draw[1cell] 
      (c) edge node {\phi} (c')
      (c) edge node[pos=.55] {\eta_C} (kc)
      (c') edge['] node[pos=.6] {\ka_{\phi}} (kc')
      (c) edge['] node {R} (x)
      (c') edge node {S} (x')
      (kc) edge[',dashed] node[pos=.3] {\usf \ol{R}} node[',pos=.3] {\exists !} (x)
      (kc') edge[dashed] node[pos=.3] {\usf \ol{S}} node[',pos=.3] {\exists !} (x')
      ;
      \draw[zz1cell]
      (x) to node {\usf f} (x')
      ;
      \draw[zz1cell]
      (kc) to node {\usf \wt{\phi}} (kc')
      ;
    \end{tikzpicture}
  \end{equation}

  Recall from \cref{defn:t-alg} that $x\cn \T X \to X$ denotes the $\T$-algebra structure 1-cell for $X$.
  We define
  \[
    \ol{R} = x \circ \T R \cn \T C \to X \inspace \Talgs 
  \]
  and note that each of the following diagrams commutes by naturality of $\eta$ and the unit condition \cref{zeta-theta-def} for $X$.
  \begin{equation}\label{eq:RSsetup}
    \begin{tikzpicture}[x=15ex,y=10ex,baseline={(0,-1)},xscale=1.2,yscale=1.1]
      \draw[0cell] 
      (0,0) node (c) {C}
      (c)++(0,-1) node (tc) {\usf \T C}
      (c)++(1,0) node (x) {\usf X}
      (x)++(0,-1) node (tx) {\usf \T X}
      (tx)++(.7,0) node (xb) {\usf X}
      ;
      \draw[1cell] 
      (c) edge node {R} (x)
      (x) edge node {1_{\usf X}} (xb)
      (c) edge['] node (etac) {\eta_C} (tc)
      (tc) edge[bend right=30] node {\usf \ol{R}} (xb)
      (tc) edge node {\usf \T R} (tx)
      (x) edge['] node {\eta_X} (tx)
      (tx) edge node {\usf x} (xb)
      ;
    \end{tikzpicture}
    \quad\quad
    \begin{tikzpicture}[x=15ex,y=10ex,baseline={(0,-1)},xscale=1.2,yscale=1.1]
      \draw[0cell] 
      (0,0) node (c) {C'}
      (c)++(0,-1) node (tc) {\usf \T C'}
      (c)++(1,0) node (x) {\usf X'}
      (x)++(0,-1) node (tx) {\usf \T X'}
      (tx)++(.7,0) node (xb) {\usf X'}
      ;
      \draw[1cell] 
      (c) edge node {S} (x)
      (x) edge node {1_{\usf X'}} (xb)
      (c) edge['] node (etac) {\eta_{C'}} (tc)
      (tc) edge node {\usf \T S} (tx)
      (x) edge['] node {\eta_{X'}} (tx)
      (tx) edge node {\usf x'} (xb)
      ;
    \end{tikzpicture}
  \end{equation}
  The diagram at left above shows that the triangle at left in \cref{eq:univ-toshow} commutes.
  Uniqueness of $\ol{R}$ follows from the uniqueness of mates noted in \cref{rmk:LRadj-uniqueness}. 

  Next, the strict $\T$-map $\ol{S}$ will be defined using the universal property of the pushout \cref{eq:udc-pushout}.
  Consider the following diagram in $\K$, explained below.
  \begin{equation}\label{eq:ustar-eta}
    \begin{tikzpicture}[x=16ex,y=10ex,vcenter,xscale=1.2]
      \draw[0cell] 
      (0,0) node (tc) {\usf \T C}
      (tc)++(1,0) node (tc') {\usf \T C'}
      (tc)++(0,-1) node (iqtc) {\usf \itt \Q\T C}
      (tc')++(0,-1) node (tc'phi) {\star}
      (tc'phi)++(1,0) node (ty) {\usf \T X'}
      (iqtc)++(-1,0) node (tcb) {\usf \T C}
      (tcb)++(0,-2) node (tx) {\usf \T X}
      (ty)++(0,-1) node (y) {\usf X'}
      (iqtc)++(1,-1) node (iqx) {\usf \itt \Q X}
      (iqx)++(0,-1) node (x) {\usf X}
      (tc)++(-1,1) node (c) {C}
      (tc')++(-1,1) node (c') {C'}
      ;
      \draw[1cell] 
      (tc) edge node {\usf \T\phi} (tc')
      (tc) edge['] node {\usf \ze^\flat} (iqtc)
      (tc') edge node {\usf \T S} (ty)
      (ty) edge node {\usf x'} (y)
      (tcb) edge['] node {\usf \T R} (tx)
      (tx) edge['] node {\usf x} (x)
      (iqtc) edge['] node {\usf \itt \Q\ol{R}} (iqx)
      (iqx) edge['] node[pos=.33] {\usf f^{\bot}} (y)
      (tcb) edge['] node {\usf \ol{R}} (x)
      (c) edge['] node {\eta_C} (tc)
      (c') edge node {\eta_{C'}} (tc')
      (c) edge node {\phi} (c')
      (c) edge['] node {\eta_C} (tcb)
      ;
      \draw[zz1cell]
      (x) to node {\usf \ze_X} (iqx)
      ;
      \draw[zz1cell]
      (x) to['] node {\usf f} (y)
      ;
      \draw[zz1cell]
      (tcb) to node {\usf \ze_{\T C}} (iqtc)
      ;
      \draw[1cell]
      (c') -| node[pos=.7] {S} ($(y)+(.4,0)$) -- (y)
      ;
      \draw[1cell]
      (c) -- +(-.5,0) |- node[',pos=.27] {R} ($(x)+(0,-.5)$) -- (x)
      ;
    \end{tikzpicture}
  \end{equation}
  In the above diagram, the two upper-left quadrilateral regions commute by \cref{eq:olze} and naturality of $\eta$, respectively.
  The lower left triangle commutes by definition of $\ol{R}$ in \cref{eq:RSsetup}.
  In the lower right triangle, $f^\bot$ is the strict mate of $f$ in \cref{eq:fbot} and hence the triangle commutes by definition.
  The lower trapezoid region commutes by naturality of $\zeta$, and the two outer regions commute by \cref{eq:RSsetup}.
  The outer diagram commutes by the hypothesis $\usf f \circ R = S \circ \phi$ in \cref{eq:univ-toshow}.

  Referring to the region $\star$ in  \cref{eq:ustar-eta} above, let
  \[
    h_1 = f^\bot \circ (\Q\ol{R}) \circ \ze^\flat
    \andspace
    h_2 = x' \circ (\T S) \circ (\T\phi).
  \]
  The above argument, together with 2-functoriality of $\usf$, shows that $\usf h_1 \circ \eta_{C} = \usf h_2 \circ \eta_C$.
  Therefore, because $h_1 $ and $h_2$ are strict $\T$-maps, we conclude $h_1 = h_2$ by the uniqueness of mates noted in \cref{rmk:LRadj-uniqueness}.

  The strict $\T$-maps $h_1$ and $h_2$ are the two composites around the boundary of the diagram in $\Talgs$ shown below.
  Since these are equal, there is a unique strict $\T$-map $\ol{S}$ induced by the universal property of the pushout \cref{eq:udc-pushout}.
  \begin{equation}\label{eq:olS-pushout}
    \begin{tikzpicture}[x=16ex,y=10ex,vcenter,xscale=1.2]
      \draw[0cell] 
      (0,0) node (tc) {\T C}
      (tc)++(1,0) node (tc') {\T C'}
      (tc)++(0,-1) node (iqtc) {\Q\T C}
      (tc')++(0,-1) node (tc'phi) {\T(C',\phi)}
      (tc'phi)++(1,0) node (ty) {\T X'}
      (ty)++(0,-1) node (y) {X'}
      (iqtc)++(1,-1) node (iqx) {\Q X}
      ;
      \draw[1cell] 
      (tc) edge node {\T\phi} (tc')
      (tc) edge['] node {\ze^\flat} (iqtc)
      (tc') edge['] node {\ka} (tc'phi)
      (iqtc) edge node {\wh{\phi}} (tc'phi)
      (tc') edge node {\T S} (ty)
      (ty) edge node {x'} (y)
      (iqtc) edge['] node {\Q\ol{R}} (iqx)
      (iqx) edge['] node {f^{\bot}} (y)
      (tc'phi) edge[dashed] node {\ol{S}} node['] {\exists !}(y)
      ;
    \end{tikzpicture}
  \end{equation}

  The construction of $\ol{S}$ then shows the following two equalities required for $\ol{S}$.
  First, using the definition $\wt{\phi} = \wh{\phi} \ze_{\T C}$ in \cref{eq:wtphi-kaphi}, the lower left parallelogram in \cref{eq:olS-pushout}, naturality of $\ze$, and the equality $f^\bot \ze_C = f$ in \cref{eq:fbot}, we have
  \begin{align*}
    \ol{S} \, \wt{\phi}
    & = \ol{S} \, \wh{\phi} \, \ze_{\T C} \\
    & = f^\bot \, (\Q \ol{R}) \, \ze_{\T C} \\
    & = f^\bot \, \ze_X \, \ol{R} \\
    & = f \, \ol{R}.
  \end{align*}
  Second, using the definition $\ka_\phi = \usf \ka \circ \eta_{C'}$ from \cref{eq:wtphi-kaphi}, 2-functoriality of $\usf$, the lower right parallelogram in \cref{eq:olS-pushout}, and the equality $S = (\usf x') \circ (\usf \T S) \circ \eta_{C'}$ from the diagram at right in \cref{eq:RSsetup}, we have
  \begin{align*}
    (\usf \ol{S}) \, \ka_\phi
    & = (\usf \ol{S}) \, (\usf \ka) \, \eta_{C'}\\
    & = (\usf x') \, (\usf \T S) \, \eta_{C'}\\
    & = S.
  \end{align*}

  This completes the construction of $\ol{S}$ and the proof that it satisfies the required equalities.
  Uniqueness of $\ol{S}$ is proved in \cref{lem:Sbar-unique}.
  This completes the proof.
\end{proof}

\section{The equivalence \texorpdfstring{$\Delta$}{∆}}
\label{sec:De1-equiv}

In this section we assume that
\begin{itemize}
\item $\T$ has an effective pseudomorphism classifier (\cref{defn:effectiveQi}) and 
\item $\T$ admits universal pseudomorphisms (\cref{defn:udc}).
\end{itemize}
This section contains two results showing that the canonical comparison \cref{eq:Delta}
\[
  \De \cn \T(C',\phi) \to \T C'
\]
is a surjective equivalence in $\Talgs$.
Its inverse is the strict $\T$-map in \cref{rmk:ka-adj}
\[
  \ka \cn \T C' \to \T(C',\phi),
\]
defined as the mate of $\ka_\phi \cn C' \to \usf\T(C',\phi)$.

\begin{thm}\label{thm:udc}
  Suppose $\T$ is a 2-monad on $\K$ that admits an effective pseudomorphism classifier $(\Q,\itt,\ze,\de)$ and universal pseudomorphisms $\wt{\phi}$.
  Suppose, moreover, that $\K$ admits cotensors of the form $\{\II,-\}$.
  Then the strict $\T$-map
  \[
    \De = \ol{\eta_{C'}} \cn \T(C',\phi) \to \T C'
  \]
  in \cref{eq:Delta} is a surjective equivalence in $\Talgs$ with inverse
  \[
    \ka \cn \T  C' \to \T(C',\phi)
  \]
  in \cref{eq:ka-adj}.
\end{thm}
\begin{proof}
  This argument consists of the following two steps.
  \begin{enumerate}
  \item Show that $\De \ka = 1_{\T C'}$.
  \item Define an invertible $\T$-algebra 2-cell
    \[
      \ol{\be}\cn \ka \De \iso 1_{\T(C',\phi)}.
    \]
  \end{enumerate}
  To begin, recall $\ka_\phi$ from \cref{eq:etakappa} is part of the unit morphism for $\wt{\phi}$.
  The strict $\T$-map $\ka$ is uniquely determined such that the outer triangle of the following diagram commutes in $\K$.
  \begin{equation}\label{eq:E}
    \begin{tikzpicture}[x=18ex,y=14ex,vcenter,xscale=1.2]
      \def\w{.05}
      \draw[0cell] 
      (0,0) node (a) {C'}
      (a)++(1,0) node (c) {\usf \T(C',\phi)}
      (a)++(.5,-.5) node (b) {\usf \T C'}
      ;
      \draw[1cell] 
      (a) edge['] node {\eta_{C'}} (b)
      (b) edge[',transform canvas={shift={(180:-\w)}}] node {\usf \ka} (c)
      (a) edge node {\ka_\phi} (c)
      (c) edge[',transform canvas={shift={(180:\w)}}] node {\usf \De} (b)
      ;
    \end{tikzpicture}
  \end{equation}
  The definition of $\De$ \cref{eq:Deltadiagram} implies that the inner triangle above also commutes in $\K$.
  Together these give the following equalities:
  \begin{align*}
    \eta_{C'} & = \usf \De \circ \ka_\phi\\
              & = \usf \De \circ \usf \ka \circ \eta_{C'} \\
              & = \usf (\De \circ \ka) \circ \eta_{C'}.
  \end{align*}
  Since $\De$ and $\ka$ are both strict $\T$-maps, the uniqueness of mates (\cref{rmk:LRadj-uniqueness}) implies
  \begin{equation}\label{eq:De-ka-1}
    \De \circ \ka = 1_{\T C'}
  \end{equation}
  as desired.

  Now we give the construction of $\be$.
  By hypothesis, there is a $\T$-map \cref{eq:phitilde}
  \[
    \wt{\phi} \cn \T C \zzto \T(C',\phi) \inspace \Talg
  \]
  satisfying the universal property \cref{eq:univprop}.
  Applying \cref{lem:flex} gives an isomorphism
  \begin{equation}\label{eq:Gamma}
    \Ga\cn \wt{\phi} \fto{\iso} \wt{\phi}^\flat
  \end{equation}
  such that $\usf \Ga * \eta_C = 1$.
  Now consider the following computation, beginning with \cref{lem:flex}~\cref{lem:flex:it1} and continuing with the indicated justifications.
  \begin{equation}\label{eq:sigma-ETphi-computation}
    \begin{aligned}
      \usf\wt{\phi}^\flat \circ \eta_C
      & = \usf\wt{\phi} \circ \eta_C
      & & \\
      & = \ka_\phi \circ \phi
      & & \mathrm{by}~\cref{eq:rmk-udc}~\mathrm{top}\\
      & = \usf \ka \circ \eta_{C'} \circ \phi
      & & \mathrm{by}~\cref{eq:E}\\
      & = \usf \ka \circ \usf \T\phi \circ \eta_C
      & & \mathrm{by~naturality~of~} \eta\\
      & = \usf (\ka \circ \T\phi) \circ \eta_C
      & & \mathrm{by~functoriality~of~} \usf
    \end{aligned}
  \end{equation}
  Hence, uniqueness of mates implies 
  \begin{equation}\label{eq:sigma-ETphi}
    \wt{\phi}^\flat = \ka \circ \T\phi.
  \end{equation}

  The equalities
  \[
    \usf\wt{\phi}^\flat \circ \eta_C
    = \usf\wt{\phi} \circ \eta_C
    = \ka_\phi \circ \phi
  \]
  in \cref{eq:sigma-ETphi-computation} also show that $(\eta_C,\ka_\phi)$ defines a morphism in $\K^\bii$ from $\phi$ to $\usf \wt{\phi}^\flat$.
  Applying the universal property of $\wt{\phi}$ \cref{eq:univprop} determines a morphism $(1,\ol{\ka_\phi})$ in $\Talgii$, as shown in the following diagram.
  \begin{equation}\label{eq:wtkadiagram}
    \begin{tikzpicture}[x=48ex,y=9ex,vcenter,xscale=1.2]
      \draw[0cell] 
      (0,0) node (c) {C}
      (c)++(1,0) node (c') {C'}
      (c)++(.3,-1) node (kc) {\usf \T C}
      (c')++(-.3,-1) node (kc') {\usf \T (C',\phi)}
      (c)++(0,-2) node (x) {\usf \T C}
      (c')++(0,-2) node (x') {\usf \T (C',\phi)}
      ;
      \draw[1cell] 
      (c) edge node {\phi} (c')
      (x) edge node {\usf \wt{\phi}^\flat} (x')
      (c) edge node[pos=.55] {\eta_C} (kc)
      (c') edge['] node[pos=.6] {\ka_{\phi}} (kc')
      (c) edge['] node {\eta_C} (x)
      (c') edge node {\ka_\phi} (x')
      (kc) edge[',dashed] node[pos=.4] {1} node[',pos=.3] {\exists !} (x)
      (kc') edge[dashed] node[pos=.5] {\usf \ol{\ka_\phi}} node[',pos=.3] {\exists !} (x')
      ;
      \draw[zz1cell]
      (kc) to node {\usf \wt{\phi}} (kc')
      ;
    \end{tikzpicture}
  \end{equation}

  Now observe that the outer diagram above can also be filled as below. 
  \begin{equation}\label{eq:wtkadiagramB}
    \begin{tikzpicture}[x=48ex,y=9ex,vcenter,xscale=1.2]
      \draw[0cell] 
      (0,0) node (c) {C}
      (c)++(1,0) node (c') {C'}
      (c)++(.3,-1) node (kc) {\usf \T C}
      (c')++(-.3,-1) node (kc') {\usf \T (C',\phi)}
      (c)++(0,-2) node (x) {\usf \T C}
      (c')++(0,-2) node (x') {\usf \T (C',\phi)}
      (x')++(-.3,.3) node (x'') {\usf \T C'}
      ;
      \draw[1cell] 
      (c) edge node {\phi} (c')
      (x) edge['] node {\usf \wt{\phi}^\flat} (x')
      (c) edge node[pos=.55] {\eta_C} (kc)
      (c') edge['] node[pos=.6] {\ka_{\phi}} (kc')
      (c) edge['] node {\eta_C} (x)
      (c') edge node {\ka_\phi} (x')
      (kc) edge[',dashed] node[pos=.4] {1} node[',pos=.3] {\exists !} (x)
      (kc') edge['] node {\usf \De} (x'')
      (x'') edge node[pos=.4] {\usf \ka} (x')
      (x) edge node[pos=.7] {\usf \T \phi} (x'')
      (c') edge[bend left=8] node {\eta_{C'}} (x'')
      ;
      \draw[zz1cell]
      (kc) to node {\usf \wt{\phi}} (kc')
      ;
    \end{tikzpicture}
  \end{equation}
  The triangles at right and bottom commute by \cref{eq:E} and \cref{eq:sigma-ETphi}, respectively.
  The remaining interior is that of \cref{eq:Deltadiagram} defining $\De$.
  By universality of $\wt{\phi}$, \cref{eq:univprop}, we conclude
  \[
    \ka \circ \De = \ol{\ka_\phi}.
  \]

  Finally, we use the hypothesis that $\K$ admits cotensors of the form $\{\II,-\}$ and apply \cref{lem:222}~\cref{it:222-II} to the diagram at left below, where $\be$ is the identity 2-cell of $\ka_{\phi}$.
  This application yields a 2-cell $\ol{\be}$ as shown in the diagram at right below.
  \[
    \begin{tikzpicture}[x=20ex,y=13ex,xscale=1.2]
      \draw[0cell] 
      (0,0) node (a) {C}
      (a)++(1,0) node (b) {C'}
      (a)++(0,-1) node (c) {\usf \T C}
      (b)++(0,-1) node (d) {\usf \T(C',\phi)}
      ;
      \draw[1cell] 
      (a) edge node {\phi} (b)
      (c) edge[',bend right=25,out=-20,in=195] node[shift={(3pt,0)}] (f2) {\usf \wt{\phi}^\flat} (d)
      (a) edge['] node {\eta} (c)
      (b) edge[bend left=25] node (s2) {\ka_\phi} (d)
      (b) edge[',bend right=25] node (s1) {\ka_\phi} (d)
      ;

      \draw (c) ++(.5,.2) ++(-3pt,0) node (M0) {};
      \draw[1cell]
      (c) to [bend left=5] (M0)
      decorate [decoration={zigzag, segment length=\midzzSegLen, amplitude=\midzzAmp}] 
      { to ++(\midzzTotLen,0) } node[above,shift={(-2pt,0)}] (f1) {\usf \wt{\phi}} to [bend left=8,in=179] (d)
      ;

      \draw[2cell]
      node[between=f1 and f2 at .5, shift={(-3pt,-1pt)}, rotate=-90, 2label={above,\,\usf\Ga}] {\Rightarrow}
      node[between=s1 and s2 at .5, rotate=0, 2label={above,\be}] {\Rightarrow}
      ;
    \end{tikzpicture}
    \qquad\qquad
    \begin{tikzpicture}[x=20ex,y=13ex]
      \draw[0cell] 
      (0,0) node (a) {\T C}
      (a)++(1,0) node (b) {\T (C',\phi)}
      (a)++(0,-1) node (c) {\T C}
      (b)++(0,-1) node (d) {\T(C',\phi)}
      ;
      \draw[1cell] 
      (c) edge[',bend right=25,in=200] node[shift={(3pt,0)}] (f2) {\wt{\phi}^\flat} (d)
      (a) edge['] node {1} (c)
      (b) edge[bend left=25] node (s2) {\ol{\ka_\phi}} (d)
      (b) edge[',bend right=25] node (s1) {1} (d)
      ;

      \draw[zz1cell]
      (a) to node {\wt{\phi}} (b)
      ;

      \draw (c) ++(.5,.2) ++(-3pt,0) node (M0) {};
      \draw[1cell]
      (c) to [bend left=5] (M0)
      decorate [decoration={zigzag, segment length=\midzzSegLen, amplitude=\midzzAmp}] 
      { to ++(\midzzTotLen,0) } node[above,shift={(-2pt,0)}] (f1) {\wt{\phi}} to [bend left=8,in=179] (d)
      ;

      \draw[2cell]
      node[between=f1 and f2 at .5, shift={(-1pt,0pt)}, rotate=-90, 2label={above,\,\Ga}] {\Rightarrow}
      node[between=s1 and s2 at .45, rotate=0, 2label={above,\ol{\be}}] {\Rightarrow}
      ;
    \end{tikzpicture}
  \]
  Since $\Ga$ is an isomorphism and $\be = 1$, the resulting $\ol{\be}$ is an invertible $\T$-algebra 2-cell
  \[
    \ol{\be}\cn 1 \iso \ol{\ka_\phi} = \ka \circ \De.
  \]
  This completes the proof that $\De$ and $\ka$ are inverse equivalences in $\Talgs$.
\end{proof}

\begin{thm}\label{rmk:alt-udc}
  Suppose $\T$ is a 2-monad on $\K$ that admits an effective pseudomorphism classifier $(\Q,\itt,\ze,\de)$ and universal pseudomorphisms $\wt{\phi}$.
  If $\T(C',\phi)$ is constructed as the pushout \cref{eq:udc-pushout} in $\Talgs$, then
  \[
    \De = \ol{\eta_{C'}} \cn \T(C',\phi) \to \T C'
  \]
  in \cref{eq:Delta} is an adjoint surjective equivalence in $\Talgs$.
\end{thm}
\begin{proof}
  Consider the following diagram, where the upper square is the pushout \cref{eq:udc-pushout} and $\om$ is described below.
  \begin{equation}\label{eq:om-diagram}
  \begin{tikzpicture}[x=14ex,y=8ex,vcenter,xscale=1.2]
    \draw[0cell] 
    (0,0) node (a) {\T C}
    (a)++(1,0) node (b) {\T C'}
    (a)++(0,-1) node (c) {\itt \Q\T C}
    (c)++(1,0) node (d) {\T(C',\phi)}
    (d)++(1,-1) node (y) {\T C'}
    (c)++(0,-1) node (x) {\T C}
    (x)++(0,-1) node (z) {\itt \Q\T C}
    (y)++(1,-1) node (w) {\T(C',\phi)}
    ;
    \draw[1cell] 
    (a) edge node {\T\phi} (b)
    (a) edge['] node {\ze^\flat} (c)
    (c) edge node {\wh{\phi}} (d)
    (b) edge node {\ka} (d)
    (c) edge['] node {\de} (x)
    (d) edge[dashed] node {\exists!} node['] {\om} (y)
    (b) edge[bend left=18] node {1} (y)
    (x) edge node {\T\phi} (y)
    (x) edge['] node {\ze^\flat} (z)
    (y) edge node {\ka} (w)
    (z) edge node {\wh{\phi}} (w)
    (c) edge[',bend right=45,out=-60,in=240] node {1} (z)
    ;
    \draw[2cell]
    (x)++(-.25,0) node[rotate=180, 2label={below,\Theta^\flat}] {\Rightarrow}
    ;
  \end{tikzpicture}
\end{equation}
Here, $(\ze^\flat,\de,\Theta^\flat)$ is the adjoint surjective equivalence of \cref{lem:zeflat,lem:flex}, with $\psi = \ze$ in the latter.
In particular, we have
\begin{equation}\label{eq:zeflat-Thetaflat}
  \de \ze^\flat = 1_{\T C}
  \andspace
  \Theta^\flat * \ze^\flat = 1_{\ze^\flat}.
\end{equation}
The left hand side of \cref{eq:zeflat-Thetaflat} implies that the two solid arrow composites from $\T C$ in the upper left to the lower right instance of $\T C'$ in \cref{eq:om-diagram} are equal.
Hence, we define $\om$ as the induced strict $\T$-map out of the pushout $\T(C',\phi)$, indicated by the dashed arrow in \cref{eq:om-diagram}.
Note that $\om \ka = 1$ by construction.

Next, whiskering $\wh{\phi}$ with the isomorphism $\Theta^\flat$ gives an isomorphism
\begin{equation}\label{eq:whphi-Thetaflat-whiskering}
  \ka \om \wh{\phi} = \wh{\phi} \ze^\flat \de \fto[\iso]{\wh{\phi} * \Theta^\flat} \wh{\phi}
  \withspace
  \bigl(\wh{\phi} * \Theta^\flat\bigr) * \ze^\flat  = \wh{\phi} * 1_{\ze^\flat} = 1_{\wt{\phi}},
\end{equation}
by the right hand side of \cref{eq:zeflat-Thetaflat} and the left hand side of \cref{eq:wtphi-kaphi}.
Thus, the two-dimensional aspect of the pushout implies that there is an isomorphism
\[
  \Psi \cn \ka\om \fto{\iso} 1_{\T(C',\phi)}
\]
such that $\wh{\phi} * \Psi = \Theta^\flat * \wh{\phi}$ and $\Psi * \ka = 1_\ka$.

This shows that $(\ka,\om,\Psi)$ is an adjoint surjective equivalence in $\Talgs$.
From uniqueness of $\De = \usf \ol{\eta_C'}$ in  \cref{eq:Deltadiagram}, it follows that $\om = \De$.
\end{proof}

\begin{rmk}\label{rmk:two-proofs-De}
  Note that \cref{rmk:alt-udc,thm:udc} require slightly different hypotheses.
  \Cref{thm:udc} requires certain limits in $\K$, in the form of cotensors, and \cref{rmk:alt-udc} requires certain colimits in $\Talgs$, in the form of pushouts.
\end{rmk}

\begin{rmk}[Consideration of lax coherence]\label{rmk:no-lax}
  The theory of pseudomorphism classifiers from \cref{sec:bg-psmor-class} has a parallel variant for \emph{lax morphism classifiers}, and some of the development in \cref{sec:Qi-strictification} can be generalized to the lax case.
  One can likewise generalize much of \cref{sec:upc} to a notion of \emph{universal lax morphism}.
  
  However, the efficacy $\Th$ for an effective lax morphism classifier is generally not invertible.
  The construction of $\Th^\flat$ in \cref{eq:Thetaflat} requires invertibility of $\Th$, and this is used in the proofs of \cref{lem:zeflat,lem:flex}.
  The proofs of \cref{thm:udc,rmk:alt-udc} above depend crucially on  \cref{lem:zeflat,lem:flex}, and hence do not apparently generalize to the lax case.
\end{rmk}

\section{Constructing \texorpdfstring{$\Q$}{Q} via universal pseudomorphisms}\label{sec:udc-implies-Qi}

Throughout this section, we suppose that $\T$ admits universal pseudomorphisms (\cref{defn:udc}).
The goal of this section is to show that this hypothesis determines a pseudomorphism classifier for $\T$ via certain coequalizers in $\Talgs$.
We first recall \emph{reflexive pairs} of morphisms, and then introduce the more specialized notion of \emph{$\P$-free pairs} in \cref{defn:P-free-fork}.

\begin{defn}\label{defn:reflexive-pair}
  A \emph{reflexive pair} in a category $C$ is a pair of parallel morphisms $f$ and $g$ with a common section $t$,
  \begin{equation}\label{eq:refl-pair}
    \begin{tikzpicture}[x=15ex,y=8ex,baseline=(a.base)]
      \def\halfsep{.7mm}
      \draw[0cell] 
      (0,0) node (a) {X}
      (1,0) node (b) {Y}
      ;
      \draw[1cell] 
      (a) edge[transform canvas={yshift=\halfsep}] node (f) {f} (b)
      (a) edge[swap,transform canvas={yshift=-\halfsep}] node (g) {g} (b)
      (b) edge[swap, bend right=45] node (t) {t} (a)
      ;
    \end{tikzpicture}
    \sothatspace 
    gt = ft = 1_Y.
  \end{equation}
\end{defn}

\begin{rmk}\label{rmk:mu-Tx-usplit}
Recall from \cref{example:usplit} that each $\T$-algebra $(X,x)$ is the coequalizer of a canonical $\usf$-split pair, with splittings below and the forgetful $\usf$ suppressed.
\[
  \begin{tikzpicture}[x=15ex,y=8ex]
    \def\halfsep{.7mm}
    \draw[0cell] 
    (0,0) node (a) {\T^2X}
    (1,0) node (b) {\T X}
    (2,0) node (c) {X}
    ;
    \draw[1cell] 
    (a) edge[transform canvas={yshift=\halfsep}] node (f) {\mu} (b)
    (a) edge[swap,transform canvas={yshift=-\halfsep}] node (g) {\T\Xmul} (b)
    (b) edge[swap] node (h) {\Xmul} (c)
    (c) edge[swap, bend right=30] node (s) {\eta_X} (b)
    (b) edge[swap, bend right=45] node (t) {\eta_{\T X}} (a)
    ;
  \end{tikzpicture}
\]
Furthermore, $\T\eta_X$ provides a common splitting for $\mu$ and $\T\Xmul$, so that the following is a reflexive pair in $\Talgs$.
\begin{equation}\label{eq:mu-Tx-reflexive}
  \begin{tikzpicture}[x=15ex,y=8ex]
    \def\halfsep{.7mm}
    \draw[0cell] 
    (0,0) node (a) {\T^2X}
    (1,0) node (b) {\T X}
    ;
    \draw[1cell] 
    (a) edge[transform canvas={yshift=\halfsep}] node (f) {\mu} (b)
    (a) edge[swap,transform canvas={yshift=-\halfsep}] node (g) {\T\Xmul} (b)
    (b) edge[swap, bend right=45] node (t) {\T\eta_{X}} (a)
    ;
  \end{tikzpicture}
\end{equation}
\end{rmk}

\begin{defn}\label{defn:QTC-TC1}
  For each object $C \in \K$, define
  \begin{equation}\label{eq:QTC-TC1}
    \P C = \T(C,1_C)
  \end{equation}
  as in \cref{eq:phitilde}, with $\phi = 1_C$.
  For a $\T$-map $f\cn \T C \zzto \T C'$, with $C,C' \in \K$, define
  \begin{equation}\label{eq:QTC-TC1-f}
    \P f = \ol{S}
    \forspace
    S = \wt{1_{C'}} \circ \bigl( \usf f \bigr) \circ \eta_C.
  \end{equation}
  That is, $\ol{S}$ is the unique strict $\T$-map determined by the universal property \cref{eq:rmk-udc}, as shown in the following diagram.
  \begin{equation}\label{eq:QTC-TC1-diagram}
    \begin{tikzpicture}[x=42ex,y=10ex,vcenter,xscale=1.2]
      \draw[0cell] 
      (0,0) node (c) {C}
      (c)++(1,0) node (c') {C}
      (c)++(.3,-1) node (kc) {\usf \T C}
      (c')++(-.3,-1) node (kc') {\usf \T (C,1_C)}
      (c)++(0,-2) node (x) {\usf \T C}
      (c')++(0,-2) node (x') {\usf \T(C',1_{C'})}
      (x)++(.5,0) node (z) {\usf \T C'}
      (c')++(0,-.66) node (c'1) {\usf \T C}
      (c'1)++(0,-.66) node (c'2) {\usf \T C'}
      ;
      \draw[1cell] 
      (c) edge node {1_C} (c')
      (c) edge node[pos=.55] {\eta_C} (kc)
      (c') edge['] node[pos=.6] {\ka_{1_C}} (kc')
      (c) edge['] node {\eta_C} (x)
      (c') edge node {\eta_C} (c'1)
      (c'2) edge node {\usf \wt{1_{C'}}} (x')
      (kc) edge[',dashed] node[pos=.3] {1} (x)
      (kc') edge[dashed] node[pos=.4] {\usf \ol{S}} node[',pos=.15] {\exists !} (x')
      ;
      \draw[zz1cell]
      (c'1) to node {\usf f} (c'2)
      ;
      \draw[zz1cell]
      (x) to node {\usf f} (z)
      ;
      \draw[zz1cell]
      (z) to node {\usf \wt{1_{C'}}} (x')
      ;
      \draw[zz1cell]
      (kc) to node {\usf \wt{1_C}} (kc')
      ;
    \end{tikzpicture}
  \end{equation}
  \ 
\end{defn}

\begin{notn}\label{notn:underlying-Talg}
  Recalling \cref{notn:underlying}, we use
  \[
    \Talgz \andspace \Talgsz
  \]
  below to denote the underlying 1-categories of $\Talg$ and $\Talgs$, respectively.
\end{notn}

\begin{defn}\label{defn:not-kleisli}
  Let $\cF$ denote the category whose objects are 0-cells of $\K$ and with hom sets
  \[
    \cF(C,C') = \Talgz(\T C,\T C') \forspace C,C' \in \K.
  \]
\end{defn}

\begin{rmk}\label{rmk:not-kleisli}
We note that $\cF$ is similar to the Kleisli category for the underlying monad $T_0$ on $\K_0$, but has $\T$-maps as morphisms instead of strict $\T$-maps.
\end{rmk}

\begin{prop}\label{prop:wt1:Id-P}
  Let $I\cn \cF \to \Talgz$ denote the functor given by $\T$ on objects and the identity on morphisms.
  Then, in the context of \cref{defn:QTC-TC1}, $\P$ defines a functor
  \[
    \P \cn \cF \to \Talgsz.
  \]
  Furthermore, the components $\wt{1}_C \cn \T C \to \T(C,1_C) = \P C$ in \cref{eq:QTC-TC1-diagram} define a natural transformation
  \[
    \wt{1}\cn I \to \itt\P.
  \]
\end{prop}
\begin{proof}
  Functoriality of $\P$ follows from uniqueness of $\ol{S}$ in \cref{eq:QTC-TC1-diagram}.
  Naturality of $\wt{1}$ follows from the commutativity of the lower trapezoid in \cref{eq:QTC-TC1-diagram}.
\end{proof}

\begin{defn}\label{defn:P-free-fork}
  Suppose that $(X,\Xmul)$ is a $\T$-algebra.
  Recall from \cref{eq:mu-Tx-reflexive} that $(\mu,\T\Xmul)$ is a reflexive pair in $\cF$. 
  The \emph{$\P$-free pair} associated to $(X,\Xmul)$ is the pair of strict $\T$-maps $(\P\mu,\P\T\Xmul)$:
  \begin{equation}\label{eq:P-free-fork}
    \begin{tikzpicture}[x=15ex,y=8ex,vcenter]
      \def\halfsep{.7mm}
      \draw[0cell] 
      (0,0) node (x) {\P\T X}
      (1,0) node (y) {\P X}
      ;
      \draw[1cell] 
      (x) edge[transform canvas={yshift=\halfsep}] node {\P\mu} (y)
      (x) edge[',transform canvas={yshift=-\halfsep}] node {\P\T x} (y)
      ;
    \end{tikzpicture}
  \end{equation}
  We say that $\Talgsz$ \emph{admits coequalizers of $\P$-free pairs} if there is a coequalizer of \cref{eq:P-free-fork} in $\Talgsz$ for each $\T$-algebra $(X,\Xmul)$.
\end{defn}

\begin{rmk}\label{rmk:P-free-fork}
  In the context of \cref{defn:P-free-fork}, the pair $(\mu,\T\Xmul)$ is a reflexive pair, and thus the same holds for $(\P\mu,\P\T\Xmul)$.
  Thus, if $\Talgsz$ admits coequalizers of reflexive pairs, then $\Talgsz$ admits coequalizers of $\P$-free pairs in particular.
\end{rmk}

\begin{defn}\label{defn:QX-coeq-mk2}
  Suppose that $\Talgsz$ admits coequalizers of $\P$-free pairs, and suppose $(X,\Xmul)$ is a $\T$-algebra.
  Define $\Q X$ as the following coequalizer in $\Talgsz$.
  \begin{equation}\label{eq:QX-coeq-mk2}
    \begin{tikzpicture}[x=15ex,y=8ex,vcenter]
      \def\halfsep{.7mm}
      \draw[0cell] 
      (0,0) node (x) {\P\T X}
      (1,0) node (y) {\P X}
      (1.8,0) node (z) {\Q X}
      ;
      \draw[1cell] 
      (x) edge[transform canvas={yshift=\halfsep}] node {\P\mu} (y)
      (x) edge[',transform canvas={yshift=-\halfsep}] node {\P\T x} (y)
      (y) edge[dashed] node {} (z)
      ;
    \end{tikzpicture}
  \end{equation}
\end{defn}

Recalling \cref{example:usplit,prop:itt-and-u-split-coeq}, each $\T$-algebra $(X,\Xmul)$ is the coequalizer, in both $\Talgs$ and $\Talg$ and their respective underlying categories, of the pair $(\mu,\T\Xmul)$.
\begin{defn}\label{defn:ze-coeq-mk2}
  For each $X \in \Talgz$, define a morphism $\ze_X$ to be the unique $\T$-map induced by the universal property of $(X,\Xmul)$ as the coequalizer in $\Talgz$, as shown in the following diagram with $\itt$ suppressed.
  The squares at left commute by naturality of $\wt{1}$ in \cref{prop:wt1:Id-P}.
  \begin{equation}\label{eq:zeX-mk2}
    \begin{tikzpicture}[x=15ex,y=10ex,vcenter]
      \def\halfsep{.7mm}
      \draw[0cell] 
      (0,0) node (x) {\T^2X}
      (1,0) node (y) {\T X}
      (1.8,0) node (z) {X}
      (0,-1) node (x') {\P\T X}
      (1,-1) node (y') {\P X}
      (1.8,-1) node (z') {\Q X}
      ;
      \draw[1cell] 
      (x) edge[transform canvas={yshift=\halfsep}] node {\mu} (y)
      (x) edge[',transform canvas={yshift=-\halfsep}] node {\T x} (y)
      (y) edge node {x} (z)
      (x') edge[transform canvas={yshift=\halfsep}] node {\P\mu} (y')
      (x') edge[',transform canvas={yshift=-\halfsep}] node {\P\T x} (y')
      (y') edge node {} (z')
      ;
      \draw[zz1cell]
      (x) -- node['] {\wt{1}} (x')
      ;
      \draw[zz1cell]
      (y) -- node {\wt{1}} (y')
      ;
      \draw[zz1cell,dashed]
      (z) -- node['] {\exists!} node {\ze_{X}} (z')
      ;
    \end{tikzpicture}
  \end{equation}
  \ 
\end{defn}

\begin{defn}\label{defn:de-De-mk2}
  For each $Y \in \Talgsz$, define a strict $\T$-map
  \[
    \de_Y \cn \Q Y \to Y
  \]
  as follows.
  Recall from \cref{eq:Delta} the strict $\T$-maps
  \[
    \De = \ol{\eta_{C'}} \cn \T(C',\phi) \to \T C' \forspace \phi \cn C \to C' \in \K.
  \]
  In the case $\phi = 1_C$, this gives a strict $\T$-map
  \begin{equation}\label{eq:DePC-mk2}
    \De_C \cn \P C \to \T C.
  \end{equation}
  Naturality of the components $\De_C$ with respect to strict $\T$-maps $h\cn \T C \to \T C'$ follows from the definition of $\P h$ \cref{eq:QTC-TC1-diagram} and uniqueness of $\ol{S}$ in the universal property \cref{eq:rmk-udc} with $\phi = 1_C$, $f = h$, and $S = \De \circ \wt{1} \circ h \circ \eta_C$.
  
  Define $\de_Y$ as the unique strict $\T$-map induced by the universal property of $\Q$ as the coequalizer in $\Talgsz$, as shown in the following diagram with $\usf$ suppressed.
  The squares at left commute by naturality of $\De$.
  \begin{equation}\label{eq:deY-mk2}
    \begin{tikzpicture}[x=15ex,y=10ex,vcenter]
      \def\halfsep{.7mm}
      \draw[0cell] 
      (0,-1) node (x) {\T^2Y}
      (1,-1) node (y) {\T Y}
      (1.8,-1) node (z) {Y}
      (0,0) node (x') {\P\T Y}
      (1,0) node (y') {\P Y}
      (1.8,0) node (z') {\Q Y}
      ;
      \draw[1cell] 
      (x) edge[transform canvas={yshift=\halfsep}] node {\mu} (y)
      (x) edge[',transform canvas={yshift=-\halfsep}] node {\T y} (y)
      (y) edge node {y} (z)
      (x') edge[transform canvas={yshift=\halfsep}] node {\P\mu} (y')
      (x') edge[',transform canvas={yshift=-\halfsep}] node {\P\T y} (y')
      (y') edge node {} (z')
      (x') edge['] node {\De_{\T Y}} (x)
      (y') edge['] node {\De_Y} (y)
      (z') edge[',dashed] node {\de_Y} (z)
      ;
    \end{tikzpicture}
  \end{equation}
  \ 
\end{defn}

\begin{lem}\label{lem:zeta-triangle-mk2}
  Given a $\T$-map $f\cn X \zzto X'$, there are unique strict $\T$-maps $\ol{f}$ and $f^\bot$ that make the following diagram commute in $\Talgz$, with $\usf$ and $\itt$ suppressed.
  \begin{equation}\label{eq:zeta-triangle-mk2}
    \begin{tikzpicture}[x=18ex,y=10ex,vcenter]
      \def\halfsep{.7mm}
      \draw[0cell] 
      (0,0) node (x) {\T^2X}
      (1,0) node (y) {\T X}
      (1.8,0) node (z) {X}
      (0,-1) node (x') {\P\T X}
      (1,-1) node (y') {\P X}
      (1.8,-1) node (z') {\Q X}
      (2.6,-1) node (w) {X'}
      ;
      \draw[1cell] 
      (x) edge[transform canvas={yshift=\halfsep}] node {\mu} (y)
      (x) edge[',transform canvas={yshift=-\halfsep}] node {\T x} (y)
      (y) edge node {x} (z)
      (x') edge[transform canvas={yshift=\halfsep}] node {\P\mu} (y')
      (x') edge[',transform canvas={yshift=-\halfsep}] node {\P\T x} (y')
      (y') edge node {} (z')
      (y') edge[dashed,bend right] node {\exists!} node['] {\ol{f}}(w)
      ;
      \draw[zz1cell]
      (x) -- node['] {\wt{1}} (x')
      ;
      \draw[zz1cell]
      (y) -- node {\wt{1}} (y')
      ;
      \draw[zz1cell]
      (z) -- node {f} (w)
      ;
      \draw[zz1cell]
      (z) to node['] {\ze_X} (z')
      ;
      \draw[1cell,dashed]
      (z') to node {f^\bot} node['] {\exists!}(w)
      ;
    \end{tikzpicture}
  \end{equation}
  Here,
  \begin{itemize}
  \item $\ol{f}$ is the unique strict $\T$-map such that $\ol{f} \circ \wt{1} = f \circ \Xmul$ and
  \item $f^\bot$ is the unique strict $\T$-map such that $f^\bot \circ \ze_X = f$.
  \end{itemize} 
  In particular, if $f = 1_Y$, then $f^\bot = \de_Y$ by uniqueness.
\end{lem}
\begin{proof}
  The strict $\T$-map $\ol{f}$ in \cref{eq:zeta-triangle-mk2} is induced by the universal property \cref{eq:rmk-udc} with $(R,S) = (\usf 1_X,\usf f)$.
  The asserted uniqueness of $\ol{f}$ is that of \cref{eq:rmk-udc}.

  The universal property for $\P\T X = \T(\T X,1_{\T X})$ implies that $\ol{f}$ coequalizes $\P\mu$ and $\P\T x$.
  The strict $\T$-map $f^\bot$ is thus induced by universality of $\Q X$ as the coequalizer in $\Talgsz$.
  The equality $f^\bot \circ \ze_X = f$, in the triangle at right in \cref{eq:zeta-triangle-mk2}, follows by commutativity of the right-hand square in \cref{eq:zeta-triangle-mk2} and universality of $X$ as the coequalizer of $(\mu,\T\Xmul)$.

  The asserted uniqueness of $f^\bot$ follows from the uniqueness of $\ol{f}$ and uniqueness in the universal property of $\Q X$.
  Indeed, suppose $f^\dagger\cn \Q X \to X'$ is any strict $\T$-map such that $f^\dagger \circ \ze_X = f$, and let $\ell\cn \P X \to \Q X$ denote 
  the structure morphism in \cref{eq:zeta-triangle-mk2}.
  Commutativity of the triangle and square at right in \cref{eq:zeta-triangle-mk2} implies $f^\dagger \circ \ell \circ \wt{1} = f \circ \Xmul$, and so $f^\dagger \circ \ell$ is equal to $\ol{f}$ by uniqueness.
  This, in turn, implies $f^\dagger = f^\bot$ by uniqueness in the universal property of $\Q X$.
\end{proof}

\begin{defn}\label{defn:Qf-mk2}
  Given a $\T$-map $f \cn X \zzto X'$, define a strict $\T$-map
  \begin{equation}\label{eq:Qf-mk2}
    \Q f = \bigl( \ze_{X'} \circ f \bigr)^\bot \cn \Q X \to \Q X'
  \end{equation}
  as the unique strict $\T$-map of \cref{lem:zeta-triangle-mk2} associated to the composite $\ze_{X'} \circ f$.
  Thus, $\Q f$ is the unique strict $\T$-map such that the following diagram commutes in $\Talgz$.
  \begin{equation}\label{eq:Qf-triangle-mk2}
    \begin{tikzpicture}[x=20ex,y=10ex,vcenter,xscale=1.2]
      \def\halfsep{.7mm}
      \draw[0cell] 
      (0,0) node (x) {X}
      (.475,-.475) node (x') {X'}
      (0,-1) node (qx) {\Q X}
      (1,-1) node (qx') {\Q X'}
      ;
      \draw[zz1cell]
      (x) -- node {f} (x')
      ;
      \draw[zz1cell]
      (x) to node['] {\ze_X} (qx)
      ;
      \draw[zz1cell]
      (x') to node {\ze_{X'}} (qx')
      ;
      \draw[1cell,dashed]
      (qx) to node['] {\Q f = \bigl( \ze_{X'} \circ f \bigr)^\bot} node {\exists!} (qx')
      ;
    \end{tikzpicture}
  \end{equation}
  \ 
\end{defn}

\begin{prop}\label{prop:Qfun-mk2}
  There is a functor
  \begin{equation}\label{eq:Qfun-mk2}
    \Q \cn \Talgz \to \Talgsz
  \end{equation}
  with object and morphism assignments given respectively by \cref{eq:QX-coeq-mk2,eq:Qf-mk2}.
  Furthermore, the components of \cref{eq:zeX-mk2,eq:deY-mk2} define respective natural transformations
  \begin{equation}\label{eq:ze-de-nat-mk2}
    \ze\cn 1_{\Talgz} \to \itt \Q
    \andspace
    \de\cn \Q \itt \to 1_{\Talgsz}
  \end{equation}
\end{prop}
\begin{proof}
  Functoriality of $\Q$ follows from uniqueness of the strict $\T$-maps $\Q f = \bigl( \ze_{X'} \circ f \bigr)^\bot$ in \cref{eq:Qf-triangle-mk2}.
  Naturality of $\ze$ with respect to $\T$-maps $f$ holds by definition of $\Q f$, since the triangle \cref{eq:Qf-triangle-mk2} is the naturality square for $\ze$.
  Naturality of $\de$ with respect to strict $\T$-maps $g \cn Y \to Y'$ follows from naturality of $\ze$, the equality $\de_X \circ \ze_X = 1_X$ in \cref{lem:zeta-triangle-mk2}, and uniqueness of the strict $\T$-maps $f^\bot$ in \cref{eq:zeta-triangle-mk2}.
\end{proof}

\begin{thm}\label{prop:UDC-Qi-adj}
  Suppose $\T$ is a 2-monad on $\K$ that admits universal pseudomorphisms $\wt{\phi}$.
  Suppose that $\K$ admits cotensors of the form $\{\bii,X\}$ and suppose that $\Talgsz$ admits coequalizers of $\P$-free pairs (\cref{defn:P-free-fork}).
  Then the functor $\Q$, together with unit $\ze$ and counit $\de$, in \cref{prop:Qfun-mk2} extends to a 2-functor that is left 2-adjoint to $\itt$.
\end{thm}
\begin{proof}
  Recalling \cref{prop:cotensor-facts}~\cref{it:cotensor-facts-i} and \cref{it:cotensor-facts-ii}, with $V = \itt$, it suffices to show $(\Q,\itt,\ze,\de)$ is an adjunction of underlying 1-categories.
  \[
    \begin{tikzpicture}[x=20ex,y=8ex,vcenter]
      \draw[0cell] 
      (0,0) node (x) {\Talgz}
      (1,0) node (y) {\Talgsz}
      ;
      \draw[1cell] 
      (x) edge[bend left=12,transform canvas={yshift=.7mm}] node (L) {\Q} (y) 
      (y) edge[bend left=12,transform canvas={yshift=-.7mm}] node (R) {\itt} (x) 
      ;
      \draw[2cell] 
      node[between=L and R at .5] {\bot}
      ;
    \end{tikzpicture}
  \]
  To do this, first recall from
  \cref{lem:zeta-triangle-mk2} that, for each $\T$-map $f \cn X \zzto X'$ there is a unique strict $\T$-map $f^\bot\cn \Q X \to X'$ such that $f^\bot \circ \ze_X = f$.
  The existence and uniqueness $f^\bot$ shows that composition with components of $\ze$ induces a bijection of morphism sets
  \[
    \Talgsz(\Q X,X') \fto[\iso]{- \circ \ze_X} \Talgz(X,\itt X')
  \]
  for each pair of $\T$-algebras $X$ and $X'$.
  Naturality of such a bijection follows from associativity of 1-cell composition and naturality of $\ze$.
  Therefore, $(\Q,\itt,\ze,\de)$ is an adjunction of underlying 1-categories, as desired.
\end{proof}

\artpart{Applications to strict monoidal structures}

\section{Formal diagrams}\label{sec:formal-diagrams}

This section develops the context for formal diagrams in the case $\K = \Cat$, the 2-category of small categories.
Recall, for a monad $\T$ that admits universal pseudomorphisms, the counit \cref{eq:etat-epzt} at a $\T$-map $f \cn X \zzto X'$ is
\[
  \wt{\epz}_f = (\epz_X,\ol{1_{X'}}).
\]
Here, $\epz_X = \Xmul$ is the algebra structure morphism for $X$ and $\ol{1_{X'}}$ is the unique strict $\T$-map such that the following diagram commutes.
\begin{equation}\label{eq:epzf-diagram}
  \begin{tikzpicture}[x=43ex,y=8ex,vcenter,xscale=1.2]
    \draw[0cell] 
    (0,0) node (c) {\usf X}
    (c)++(1,0) node (c') {\usf X'}
    (c)++(.3,-1) node (kc) {\usf \T X}
    (c')++(-.3,-1) node (kc') {\usf \T (X',f)}
    (c)++(0,-2) node (x) {\usf X}
    (c')++(0,-2) node (x') {\usf X'}
    ;
    \draw[1cell] 
    (c) edge node {\usf f} (c')
    (c) edge node[pos=.55] {\eta_X} (kc)
    (c') edge['] node[pos=.6] {\ka_{f}} (kc')
    (c) edge['] node {1_X} (x)
    (c') edge node {1_{X'}} (x')
    (kc) edge[dashed] node[pos=.3] {\usf \epz_X = \usf \Xmul} node[',pos=.3] {\exists !} (x)
    (kc') edge[',dashed] node[pos=.3] {\usf \ol{1_{X'}}} node[',pos=.3] {\exists !} (x')
    ;
    \draw[zz1cell]
    (x) to node {\usf f} (x')
    ;
    \draw[zz1cell]
    (kc) to node {\usf \wt{f}} (kc')
    ;
  \end{tikzpicture}
\end{equation}

Over $\K = \Cat$, each $\T$-algebra $X$ has an underlying set of objects, $\ob X$.
Thus, we have the following.
\begin{defn}\label{defn:La}
  Suppose $\T$ is a 2-monad on $\Cat$ that admits universal pseudomorphisms (\cref{defn:udc}).
  For each $\T$-map
  \[
    f\cn (X,\Xmul) \zzto (X',\Xmul'), 
  \]
  define a strict $\T$-map $\La$ as the composite below,
  \begin{equation}\label{eq:La}
    \begin{tikzpicture}[x=13ex,y=10ex,vcenter,xscale=1.2,yscale=1.1]
      \draw[0cell] 
      (0,0) node (a) {\T(\ob X', f_\ob)}
      (a)++(1.05,-.5) node (b) {\T(X',f)}
      (a)++(2,0) node (c) {X'}
      ;
      \draw[1cell] 
      (a) edge['] node {} (b)
      (b) edge['] node {\ol{1_{X'}}} (c)
      (a) edge node {\La} (c)
      ;
    \end{tikzpicture}
  \end{equation}
  where $f_\ob$ denotes the restriction of $f$ to objects, the unlabeled strict $\T$-map is induced by the inclusion of objects $\ob X' \hookrightarrow X'$, and $\ol{1_{X'}}$ is part of the counit $\wt{\epz}_f$ in \cref{eq:epzf-diagram}.
  Equivalently, $\La$ is the unique strict $\T$-map induced by the universal property \cref{eq:rmk-udc} in the following diagram, where the unlabeled arrows are induced by inclusion of objects.
  \begin{equation}\label{eq:La-diagram}
    \begin{tikzpicture}[x=48ex,y=7ex,vcenter,xscale=1.2]
      \draw[0cell] 
      (0,0) node (c) {\ob X}
      (c)++(1,0) node (c') {\ob X'}
      (c)++(.3,-1) node (kc) {\usf \T (\ob X)}
      (c')++(-.3,-1) node (kc') {\usf \T (\ob X',f_\ob)}
      (c)++(0,-3) node (x) {\usf X}
      (c')++(0,-3) node (x') {\usf X'}
      (kc)++(.05,-1.0) node (kc2) {\usf \T X}
      (kc')++(-.05,-1.0) node (kc'2) {\usf \T (X',f)}
      ;
      \draw[1cell] 
      (c) edge node {f_\ob} (c')
      (c) edge node[pos=.55] {\eta_X} (kc)
      (c') edge['] node[pos=.6] {\ka_{f}} (kc')
      (c) edge['] node {} (x)
      (c') edge node {} (x')
      (kc) edge[',dashed] node[',pos=.3] {\exists !} (x)
      (kc') edge[dashed] node[pos=.3] {\La} node[',pos=.3] {\exists !} (x')
      (kc) edge node {} (kc2)
      (kc2) edge node {\Xmul} (x)
      (kc') edge node {} (kc'2)
      (kc'2) edge['] node[scale=.8,pos=.3] {\ol{1_{X'}}} (x')
      ;
      \draw[zz1cell]
      (x) to node {\usf f} (x')
      ;
      \draw[zz1cell]
      (kc) to node {\usf \wt{f_\ob}} (kc')
      ;
      \draw[zz1cell]
      (kc2) to node {\usf \wt{f}} (kc'2)
      ;
    \end{tikzpicture}
  \end{equation}
  \ 
\end{defn}
\begin{rmk}\label{rmk:La-vs-De}
  Note, in the context of \cref{defn:La}, that
  \begin{equation}\label{eq:La-vs-De-1}
    \La \cn \T(\ob X',f_\ob) \to X'
  \end{equation}
  is generally distinct from the following composite of $\Xmul'$ with the canonical comparison $\De$ of \cref{eq:Delta}, where the unlabeled arrow is again induced by inclusion of objects:
  \begin{equation}\label{eq:La-vs-De-2}
    \T(\ob X', f_\ob ) \to \T(X',f) \fto{\De} \T X' \fto{\Xmul'} X'.
  \end{equation}
  Indeed, if $f$ is a strict $\T$-map, so that the algebra constraint $f_\bullet$ in \cref{eq:Tmap-2cell} is an identity, then uniqueness of $\La$ in \cref{eq:La-diagram} will imply that \cref{eq:La-vs-De-1} and \cref{eq:La-vs-De-2} are equal.
  In general however, they are distinct, and their difference is a key feature of our examples in \cref{sec:other}.
\end{rmk}

\begin{defn}\label{defn:formal-diagram}
  Suppose $\T$ is a 2-monad on $\Cat$ and $(X,\Xmul)$ is a $\T$-algebra.
  In the following, the unlabeled arrows are induced by inclusions of objects
  \[
    \ob X \hookrightarrow X
    \andspace
    \ob X' \hookrightarrow X'.
  \]
  \begin{description}
  \item[Diagram:] A \emph{diagram} $(\DD, D)$ in $X$ consists of a small category $\DD$ and a functor $D \cn \DD \to X$. 
    We consider a morphism $s \cn a \to b$ in $X$ as a diagram by taking $\DD = \bii$, with $D$ sending the unique morphism of $\bii$ to $s$.
  \item[Formal diagram:] 
    A diagram $(\DD, D)$ in $X$ is called a \emph{formal diagram for $X$} or an \emph{$X$-formal diagram} if there is a lift $\wt{D}$ such that the following commutes in $\Cat$.
    In this case, $\wt{D}$ is called an \emph{$X$-formal lift of $(\DD,D)$}.
    \begin{equation}\label{eq:wtD}
      \begin{tikzpicture}[x=15ex,y=7ex,vcenter,xscale=1.2]
        \draw[0cell] 
        (0,0) node (d) {\DD}
        (1,2) node (tox) {\T ( \ob X )}
        (1,1) node (tx) {\T X}
        (1,0) node (x) {X}
        ;
        \draw[1cell] 
        (d) edge[dashed] node {\wt{D}} (tox)
        (tox) edge node {} (tx)
        (d) edge node {D} (x)
        (tx) edge node {\Xmul} (x)
        ;
      \end{tikzpicture}
    \end{equation}
  \item[Formal diagram for a $\T$-map:]
    Suppose that $\T$ admits universal pseudomorphisms \cref{eq:univprop}, and suppose that $f \cn (X,\Xmul) \zzto (X',\Xmul')$ is a $\T$-map.
    A diagram $(\DD,D)$ in $X'$ is called a \emph{formal diagram for $f$} or an \emph{$f$-formal diagram} if there is a lift $\wt{D}$ such that the triangle at left below commutes in $\Cat$, where $f_\ob$ denotes the restriction of $f$ to objects and $\La$ is defined in \cref{eq:La}.
    In this case, $\wt{D}$ is called an \emph{$f$-formal lift of $(\DD,D)$}.
    \begin{equation}\label{eq:wtD-f}
      \begin{tikzpicture}[x=15ex,y=7ex,vcenter,xscale=1.2]
        \draw[0cell] 
        (0,0) node (d) {\DD}
        (d)++(1,0) node (x') {X'}
        (x')++(1,0) node (x'2) {X'}
        (x')++(0,1) node (tx'f) {\T(X',f)}
        (tx'f)++(0,1) node (tox'f) {\T(\ob X',f_\ob)}
        (tx'f)++(1,0) node (tx') {\T X'}
        (tox'f)++(1,0) node (tox') {\T (\ob X')}
        ;
        \draw[1cell] 
        (d) edge node {D} (x')
        (tox'f) edge node {} (tx'f)
        (tx'f) edge['] node {\ol{1_{X'}}} (x')
        (tox'f) edge[bend left=60] node {\La} (x')
        (d) edge[dashed] node {\wt{D}} (tox'f)
        (tox'f) edge node {\De} (tox')
        (tox') edge node {} (tx')
        (tx') edge node {\Xmul'} (x'2)
        ;
      \end{tikzpicture}
    \end{equation}
  \item[Dissolution:]
    If $(\DD,D)$ is a formal diagram for $f$ with lift $\wt{D}$ as in \cref{eq:wtD-f}, the \emph{dissolution} of $\wt{D}$, denoted $\abs{D}$, is the composite
    \[
      \abs{D} = \De \circ \wt{D} \cn \DD \fto{\hspace{5ex}} \T(\ob X').
    \]
  \item[Finite generation:]
    In the above contexts, a lift $\wt{D}$ for a formal diagram is said to be \emph{finitely generated} if there is a finite set of objects $G \subset \ob X$ such that $\wt{D}$ factors through, respectively, the strict $\T$-map
    \[
      \T G \to \T(\ob X)
      \orspace
      \T(G',f_G) \to \T(\ob X', f_{\ob}),
    \]
    induced by inclusion of objects, where $f_G$ denotes the restriction of $f_\ob$ to $G$.
  \end{description}
  In any of the above cases, we say that a diagram $(\DD,D)$ \emph{commutes} if we have $D(u) = D(v)$ for every parallel pair of morphisms $u$ and $v$ in $\DD$.
\end{defn}

\begin{rmk}[Using dissolution diagrams]\label{rmk:what-it-means}
  Suppose, in the context of \cref{defn:formal-diagram}, that $(\DD,D)$ is a formal diagram for $f$, with lift $\wt{D}$ to $\T(\ob X', f_\ob)$.
  Suppose, furthermore, that $\De$ is an equivalence, as in \cref{thm:main1,thm:main2}.

  Then, for each pair of parallel morphisms $u$ and $v$ in $\DD$,
  the lifts $\wt{D}(u)$ and $\wt{D}(v)$ are equal in $\T(\ob X',\phi)$ if and only if their dissolutions $\abs{D}(u)$ and $\abs{D}(v)$ are equal in $\T (\ob X')$.
  Hence, the diagram $(\DD,\wt{D})$ commutes in $\T(\ob X',\phi)$ if and only if the dissolution diagram $(\DD,\abs{D})$ commutes in $\T(\ob X')$.
  Furthermore, commutativity of $(\DD,\wt{D})$ implies that of the original diagram $(\DD,D)$.

  Note, however, that the distinction in \cref{rmk:La-vs-De} implies $D$ and $\abs{D}$ generally give \emph{distinct} diagrams in $X'$.
  That is, for each morphism $u$ in $\DD$, the morphisms in
  $X'$ determined by $D(u)$ and $\abs{D}(u)$---composing the latter along the right hand side of \cref{eq:wtD-f}---are generally not equal in $X'$.

  Thus, if $\De$ is an equivalence, the dissolution diagram $(\DD,\abs{D})$ is a diagram that is generally different from the given diagram $(\DD,D)$, and yet commutativity of the former implies that of the latter.
  \Cref{sec:other} contains a variety of examples that demonstrate this phenomenon.
\end{rmk}

\begin{rmk}[Formal diagrams that factor through $\ka$]\label{rmk:f-formal-vs-X-formal}
  In the context of \cref{defn:formal-diagram}, recall from \cref{eq:ka-adj} the strict $\T$-map
  \[
    \ka\cn \T(\ob X') \to \T(\ob X', f_\ob)
  \]
  is the mate of
  \[
    \ka_{f_\ob} \cn \ob X' \to \T(\ob X', f_\ob).
  \]
  Note that the composite $\De \circ \ka$ is equal to the identity $1_{\T(\ob X')}$, as in \cref{eq:De-ka-1}.

  Each $X'$-formal diagram is trivially an $f$-formal diagram by composing its lift $\wt{D}$ with $\ka$.
  In such a case, for the dissolution diagram $\abs{D}$ obtained by composing with $\De$, we have
  \[
    \abs{D} = \De \circ (\ka \circ \wt{D}) = \wt{D}.
  \]
  We will say that an $f$-formal lift \emph{reduces to an $X'$-formal lift} if it factors through $\ka$.
\end{rmk}

\section{Strict monoidal structures}\label{sec:Mv-and-Tv}
We use the following notations for the 2-monads on $\K=\Cat$ whose algebras are general or strict monoidal structures in the plain, symmetric, and braided monoidal cases.
For basic definitions and properties, we refer the reader to \cite[Chapter~XI]{ML98Categories}, \cite{JS1993Braided}, and \cite[Chapter~1]{JYringII}.

Here, we give a brief description of the relevant 2-monads.
See, e.g., \cite[Section~4]{Lac02Codescent}.
More detailed descriptions will not be required, but can be found in operadic presentations such as, e.g., \cite[Part~4]{Yau2021Infinity} or \cite[Chapters~11 and 12]{JYringIII}.
We use a superscript $\scriptstyle{\mathsf{g}}$ to denote the \emph{general} monoidal cases, and use unadorned notation for the strict monoidal cases.
\begin{notn}[Monads for monoidal structures]\label{notn:monoidal-variants}
  \ 
  \begin{description}
  \item[Plain monoidal:] Let $\Mg$ denote the 2-monad whose algebras are monoidal categories.
    Let $\M$ denote the 2-monad whose algebras are strict monoidal categories.

    For a category $C$, the free strict monoidal category $\M C$ has objects given by tuples $\ang{a} = (a_1,\ldots,a_n)$, for $n \ge 0$, with $a_i \in C$ for $i \in \{1,\ldots, n\}$.
    The morphisms of $\M C$ are tuples of morphisms, so that the underlying category of $\M C$ is $\coprod_n C^n$.
    The monoidal product is given by concatenation and the monoidal unit is the empty tuple. 
  \item[Symmetric monoidal:] Let $\Sg$ denote the 2-monad whose algebras are symmetric monoidal categories.
    Let $\S$ denote the 2-monad whose algebras are symmetric strict monoidal categories, also known as \emph{permutative categories}.

    For a category $C$, the free symmetric strict monoidal category $\S C$ has the same objects and monoidal structure as $\M C$.
    The morphisms of $\S C$ are generated by those of $\M C$, together with permutations of the tuples $\ang{a}$.
    In particular, for a single object $a$, the free symmetric strict monoidal category $\S\{a\}$ has an object for each natural number $n$, corresponding to the $n$-tuple $(a,\ldots,a)$. 
    The hom sets are given by
    \begin{equation}\label{eq:Tsa}
      \bigl( \S\{a\} \bigr)(m, n) \cong
      \begin{cases}
        \emptyset, & \text{if\ \ } m \neq n, \\
        \Sigma_m, & \text{if\ \ } m = n,
      \end{cases}
    \end{equation}
    where the symmetry isomorphism $\beta_{a,a}$ is identified with the transposition $(1\ 2)$.
    
  \item[Braided monoidal:] Let $\Bg$ denote the 2-monad whose algebras are braided monoidal categories.
    Let $\B$ denote the 2-monad whose algebras are braided strict monoidal categories.

    For a category $C$, the free braided strict monoidal category $\B C$ has the same objects and monoidal structure as $\M C$ and $\S C$.
    The morphisms of $\B C$ are generated by those of $\M C$ together with braidings of strands labeled by the entries of the tuples $\ang{a}$.
  \end{description}

  In the cases $\T = \M, \S, \B$, respectively, the $\Tg$-maps and $\T$-maps are plain, symmetric, and braided monoidal functors.
  These are also sometimes called plain/symmetric/braided \emph{strong} monoidal functors.
  We will suppress the additional adjective except where it is useful to emphasize the distinction with \emph{strict} $\T$- or $\Tg$-maps.
  The latter are the plain/symmetric/braided \emph{strict} monoidal functors, so they have identity monoidal and unit constraints.

  In both the symmetric and braided cases, a $\T$-map $f \cn A \zzto B$ satisfies an additional \emph{braid axiom}, expressed as commutativity of the following diagram for $a,a' \in A$.
  Here, $\bcdot$ and $\beta$ denote the monoidal products and symmetry/braid isomorphisms, respectively, in both $A$ and $B$.
  \begin{equation}\label{eq:braid-axiom}
    \begin{tikzpicture}[x=25ex,y=8ex,vcenter]
      \draw[0cell] 
      (0,0) node (a) {f(a)\bcdot f(a')}
      (a)++(1,0) node (b) {f(a')\bcdot f(a)}
      (a)++(0,-1) node (c) {f(a \bcdot a')}
      (b)++(0,-1) node (d) {f(a' \bcdot a)}
      ;
      \draw[1cell] 
      (a) edge node {\beta_{f(a),f(a')}} (b)
      (c) edge node {f(\beta_{a,a'})} (d)
      (a) edge['] node {f_2} (c)
      (b) edge node {f_2} (d)
      ;
    \end{tikzpicture}
  \end{equation}
  In all three cases $\T \in \{\M,\S,\B\}$, the $\T$-algebra 2-cells are monoidal transformations.
\end{notn}

In each case of \cref{notn:monoidal-variants}, algebras for the strict monoidal monads $\T$ are also algebras for the general monoidal monads $\Tg$, with $\T \in \{\M, \S, \B\}$.
There is a morphism of monads
\[
  \theta^{\T}\cn \Tg \to \T
\]
for each $\T$, and changing monad structure along this morphism is the forgetful functor from the strict to general variants.

The statements in the following result are equivalent to the general coherence theorems
\cite[VII.2, Corollary]{ML98Categories}, \cite[XI.1, Theorem~1]{ML98Categories}, and \cite[Corollary~2.6]{JS1993Braided}, respectively.
\begin{thm}[Monoidal Strictification]\label{thm:strictification}
  Suppose $C$ is a category.
  Each of 
  \begin{align*}
    \theta^{\M} & \cn \Mg C \to \M C\\
    \theta^{\S} & \cn \Sg C \to \S C\\
    \theta^{\B} & \cn \Bg C \to \B C
  \end{align*}
  is a plain, respectively symmetric, respectively braided, strict monoidal functor, and is an equivalence.
\end{thm}

\begin{cor}\label{cor:str}
  For each monad $\T \in \{\M,\S,\B\}$, commutativity of a formal diagram $(\DD,D)$ with lift
  \[
    \wt{D} \cn \DD \to \Tg(\ob X)
  \]
  is determined by that of the composite
  \[
    \DD \fto{\wt{D}} \Tg(\ob X) \fto{\theta^{\T}} \T(\ob X).
  \]
\end{cor}

\subsection*{Diagrammatic coherence for strict monoidal structures}

Our applications to coherence for strong monoidal functors in \cref{sec:other} will make use of the corresponding coherence theorems for monoidal structures on categories.
We recall these in \cref{thm:diagrcoh} below, making use of the following concepts.
\begin{defn}\label{defn:underlying}
  Suppose $G$ is a set, regarded as a discrete category.
  \begin{description}
  \item[Underlying braids:]
    Each morphism $s\cn \ang{a} \to \ang{b}$ in the braided strict monoidal category $\B G$ has an \emph{underlying braid} $\ups(s)$ determined as follows.
    \begin{itemize}
    \item For an identity morphism, $\ups(1) = 1$, the identity braid.
    \item For a composite, $\ups(s's) = \ups(s')\ups(s)$, the composition of braids.
    \item For a concatenation, $\ups(s' + s) = \ups(s') \oplus \ups(s)$, the block sum of braids.
    \item For the braid isomorphism, $\ups(\beta_{\ang{a},\ang{a'}})$ is the elementary block braid that passes the block of strands labeled by $\ang{a}$ under the block of strands labeled by $\ang{a'}$, without braiding within either block.
    \end{itemize}
  \item[Underlying permutations:]
    Each morphism $s \cn \ang{a} \to \ang{b}$ in the symmetric strict monoidal category $\S G$ has an \emph{underlying permutation} $\pi(s)$ defined as the underlying permutation of the underlying braid $\ups(s)$.
  \end{description}
  Underlying permutations, respectively braids, in the more general $\Sg G$, respectively $\Bg G$, are defined via the equivalences $\theta^{\S}$, respectively $\theta^{\B}$.
\end{defn}

\begin{notn}\label{notn:parallel-diagram}
  Let
  \begin{equation}\label{eq:parallel-diagram}
    \PP = \Bigl\{ 
    \begin{tikzpicture}[x=3em,baseline={(0,-3pt)}]
      \draw[0cell]
      (0,0) node (A) {0}
      (1,0) node (B) {1};
      \draw[1cell]
      (A) edge[transform canvas={yshift=2pt}] node {s} (B)
      (A) edge[transform canvas={yshift=-2pt}] node['] {t} (B)
      ;
    \end{tikzpicture}
    \Bigr\}
  \end{equation}
  denote the free parallel arrow category, consisting of two objects and two parallel morphisms, $s$ and $t$, between them.
\end{notn}

\begin{thm}[Monoidal Coherence]\label{thm:diagrcoh}
  Suppose $A$ is a monoidal, respectively symmetric monoidal, respectively braided monoidal category.
  Suppose $(\PP, D)$ is a formal diagram with lift $\wt{D}$, classifying a pair of parallel morphisms $Ds$ and $Dt$ in $A$.
  \begin{enumerate}
  \item\label{it:M} In the plain monoidal case, $\Mg(\ob A)$ has at most one morphism between any pair of objects, so $\wt{D}s = \wt{D}t$ and hence $Ds = Dt$ {\cite[VII.2]{ML98Categories}}.
  \item\label{it:S} In the symmetric case, if the underlying permutations $\pi(\wt{D}s)$ and $\pi(\wt{D}t)$ are equal, then $\wt{D}s = \wt{D}t$ and hence $Ds = Dt$ {\cite[XI.1]{ML98Categories}}.
  \item\label{it:B} In the braided case, if the underlying braids $\ups(\wt{D}s)$ and $\ups(\wt{D}t)$ are equal, then $\wt{D}s = \wt{D}t$ and hence $Ds = Dt$ {\cite[Corollary~2.6]{JS1993Braided}}.
  \end{enumerate}
\end{thm}

\section{Diagrammatic coherence in the symmetric case}\label{sec:symm-coh}
In the symmetric case $\T = \S$ in \cref{sec:Mv-and-Tv}, there is a simplification for formal diagrams that are finitely generated---a condition which holds in all diagrammatic coherence applications known to the authors.
The simplification makes use of the following result that finite coproducts and finite products of symmetric strict monoidal categories are equivalent.
\begin{thm}[{\cite[Theorem~14.27]{GJOsmbperm}}]\label{thm:smbperm1427}
  Suppose given symmetric strict monoidal categories $A_i$ for $i \in \{1,\ldots,n\}$.
  There is a symmetric strict monoidal functor $I$
  \begin{equation}\label{eq:I}
    \coprod_{i = 1}^n A_i \fto{I} \prod_{i = 1}^n A_i
  \end{equation}
  such that the following statements hold.
  \begin{enumerate}
  \item Each composite with the canonical morphisms
    \[
      A_i \to
      \coprod_{i = 1}^n A_i \fto{I} \prod_{i = 1}^n A_i
      \to A_j
    \]
    is the identity on $A_i$ if $i = j$ and constant at the monoidal unit of $A_j$ otherwise.
  \item $I$ is an equivalence of symmetric strict monoidal categories.
  \end{enumerate} 
\end{thm}
\begin{rmk}
  In \cref{thm:smbperm1427}, $I$ is a symmetric strict monoidal functor, and it is an equivalence, but it does not have a strict monoidal inverse.
  See \cite[Remark~14.25]{GJOsmbperm} for further explanation of this point.
  The proof of \cref{thm:smbperm1427} depends on an analysis of coproducts for symmetric strict monoidal categories that specializes the Gray tensor product of 2-categories.
\end{rmk} 

Recall that $\S$ is left adjoint to the forgetful $\usf$, and therefore commutes with colimits, particularly coproducts.
\begin{defn}\label{defn:ola}
  Suppose $G$ is a finite set.
  Define a strict monoidal functor $\wt{I}$, and
  strict monoidal functors $I_a$ for each $a\in G$, as the composites described below.
  \begin{equation}\label{eq:Ia}
    \begin{tikzpicture}[x=23ex,y=8ex,vcenter]
      \draw[0cell] 
      (0,0) node (a) {
        \S G
      }
      (a)++(0,1) node (b) {
        \S \Bigl( \coprod_{b \in G} \{b\} \Bigr)
      }
      (b)++(.8,0) node (c) {
        \coprod_{b \in G} \S\{b\}
      }
      (c)++(1.2,0) node (d) {
        \prod_{b \in G} \S\{b\}
      }
      (d)++(0,-1) node (e) {
        \S\{a\}
      }
      ;
      \draw[1cell] 
      (a) edge[equal] node {} (b)
      (b) edge node {\iso} (c)
      (c) edge node {I} node['] {\hty} (d)
      (d) edge node {} (e)
      (a) edge['] node {I_a} (e)
      (a) edge['] node {\wt{I}} (d)
      ;
    \end{tikzpicture}
  \end{equation}
  In the above diagram, the isomorphism is given by commuting $\S$ with coproducts, the equivalence $I$ is that of \cref{eq:I}, and the unlabeled arrow is projection from the product.
\end{defn}

Recall from \cref{defn:formal-diagram} that a finitely generated formal diagram is one that factors through a free algebra on a finite set.
\begin{defn}\label{defn:aperm}
  Suppose $A$ is a symmetric strict monoidal category and suppose that $(\DD,D)$ is a diagram in $A$ that is formal and finitely generated, with lift $\wt{D} \cn \DD \to \S G$ for a finite set $G \subset \ob A$.
  For each morphism $s$ in $\DD$ and each $a \in G$, define the permutation $\pidt_a(s)$ as the underlying permutation of the image of $s$ in $\S\{a\}$.
  That is,
  \[
    \pidt_a(s) = \pi\bigl((I_a\wt{D})(s)\bigr).
  \]
  We call $\pidt_a(s)$ the \emph{$a$-permutation of $s$} or the \emph{self-permutation of $a$}.
\end{defn}

\begin{thm}\label{thm:diagrcoh-s}
  Suppose $(\PP, D)$ is a formal diagram classifying a pair of parallel morphisms $Ds$ and $Dt$ in a symmetric strict monoidal category $A$.
  Suppose, moreover, that there is a finitely generated lift $\wt{D}$, factoring through $\S G$ for a finite set $G$, such that
  \begin{equation}\label{eq:pispit}
    \pidt_a(s)=\pidt_a(t) \foreachspace a \in G.
  \end{equation}
Then $Ds = Dt$ in $A$.
\end{thm}
\begin{proof}
  The hypotheses of the theorem establish the following context, where the left hand triangle is that of the formal diagram $(\PP,D)$ and the finitely generated lift $\wt{D}$.
  The right hand triangle is \cref{eq:Ia}.
  Recall that $\wt{I}$ is an equivalence by \cref{thm:smbperm1427}.
  \[
    \begin{tikzpicture}[x=-5ex,y=15ex,rotate=90,xscale=1.15,yscale=1.3]
      \draw[0cell] 
      (0,0) node (tg) {\S G}
      (tg)++(1.14,0) node (toa) {\S (\ob A)}
      (toa)++(.93,0) node (ta) {\S A}
      (ta)++(.93,0) node (a) {A}
      (tg)++(1.4,-1) node (tx) {\S \{a\}}
      (tg)++(0,-1) node (p) {\prod_{b \in G} \S \{b\}}
      (a)++(0,1) node (dd) {\PP}
      ;
      \draw[1cell] 
      (dd) edge node {D} (a)
      (dd) edge node {\wt{D}} (tg)
      (tg) edge['] node {} (toa)
      (toa) edge node {} (ta)
      (ta) edge node {} (a)
      (tg) edge node[pos=.6] {\wt{I}} node[',pos=.6] {\hty} (p) 
      (tg) edge['] node[pos=.7] {I_a} (tx)
      ;
      \draw[1cell]
      (p) edge node {} (tx)
      ;
    \end{tikzpicture}
  \]
  By the universal property of the product, the equalities \cref{eq:pispit} imply that the morphisms $\wt{I}\wt{D}s$ and $\wt{I}\wt{D}t$ are equal in $\txprod_{b \in G} \S \{b\}$.
  Since $\wt{I}$ is an equivalence, we have
  \[
    \wt{D}s = \wt{D}t \inspace \S G,
  \]
  and hence $Ds = Dt$ as desired.
\end{proof}

\begin{rmk}\label{rmk:components}
  It is instructive to compare the statement of \cref{thm:diagrcoh-s} with the more familiar statement for vectors in a vector space $V$ over a field $k$.
  If $V$ has finite dimension $n$, then choosing a basis for $V$ provides an isomorphism $V \iso k^{\oplus n}$.
  Thus, two vectors $v,w \in V$ are equal if an only if their components in $k^{\oplus n}$ are equal.
  The self-permutations $\pidt_a(s)$ provide the same condition: $\wt{I}$ is an equivalence and, therefore, two underlying permutations $\pidt(s)$ and $\pidt(t)$ are equal if and only if their $a$-permutations are equal for each generating object $a$.
\end{rmk}

Several examples of \cref{thm:diagrcoh-s} are given in \cref{sec:doub-quad}.
In particular, see \cref{rmk:symmetric-coherence-example}, \cref{example:cursed-cyclic} and \cref{rmk:cursed-n}.

\section{Explication: Pseudomorphism classifiers}\label{sec:psmorclass}

In this section we give an explicit description of the pseudomorphism classifiers
\[
  \Q\cn \Talg \to \Talgs
\]
for each 2-monad $\T \in \{\M,\S,\B\}$ of \cref{notn:monoidal-variants}.
We present a unified construction, noting minor differences in the three cases where appropriate.
In these applications, we work with the strict monoidal 2-monads $\T$, instead of the general $\Tg$, in order to highlight the essential features.
Equivalent results hold for the general monoidal variants by \cref{cor:str}.
Here and in \cref{sec:applications} we make use of the following.

\begin{notn}\label{notn:TA}
  Suppose $\T \in \{\M,\S,\B\}$ and suppose $(A,\bcdot,e)$ is a $\T$-algebra with monoidal unit $e$ and multiplication denoted as $\bcdot$ or with juxtaposition.
  Recall from \cref{notn:monoidal-variants} that the objects of $\T A$ are given by tuples of objects from $A$.
  The morphisms of $\T A$ are generated by tuples of morphisms from $A$ together with, in the symmetric and braided cases, permutations and braidings, respectively.

  We will use the following notation for objects and morphisms in $\T A$ that are given by tuples of objects and morphisms in $A$:
  \begin{align}
    \begin{split}
    \ang{a} & = \ang{a_i}_{i=1}^n = (a_1,\ldots,a_n)\\
    \ang{s} & = \ang{s_i}_{i=1}^n = (s_1,\ldots,s_n)
    \end{split}
  \end{align}
  where $a_i$ and $s_i$ are objects and morphisms, respectively, in $A$ and $n \ge 0$.
  \begin{itemize}
  \item The number $n$ is called the \emph{length} of $\ang{a}$.
  \item The empty tuple is denoted $\ang{}$ and has length 0.
  \item For a tuple $\ang{a}$ of length $n$, we write
    \[
      a_\bullet = a_1 \!\bcdots a_n
    \]
    to denote the product in $A$ of the entries $a_i$.
  \item For tuples $\ang{a^1}$ and $\ang{a^2}$ of length $n_1$ and $n_2$, respectively, we denote concatenation with a semicolon $;$ and write
    \[
      \ang{a^{1;2}} = \ang{a^1};\ang{a^2} = \ang{a^1;a^2}
    \]
    to denote the tuple whose first $n_1$ entries are those of $\ang{a^1}$ and whose final $n_2$ entries are those of $\ang{a^2}$.
  \end{itemize}
  This same terminology and notation is used for tuples of morphisms $\ang{s}$.

  We also denote the image of a general morphism $t$ under the multiplication $\T A\fto{\bcdot} A$ as $t_\bullet$.
  For example, $t$ may be a permutation or braiding if $\T \in \{\S,\B\}$.
  In such a case, $t_\bullet$ is the corresponding symmetry or braid isomorphism in $A$.

  Thus, the composite
  \[
    \T A \fto{\bcdot} A \fto{\eta_A} \T A
  \]
  is denoted as a length-one tuple with subscript $\scriptstyle{\bullet}$, so we write
  \begin{equation}\label{eq:mul-bullet}
    \begin{split}
      \ang{a} & \mapsto (a_\bullet), \\
      \ang{s} & \mapsto (s_\bullet), \andspace \\
      t & \mapsto (t_\bullet)
    \end{split}
  \end{equation}
  where $\ang{a}$ and $\ang{s}$ are tuples of objects and morphisms, respectively, and $t$ is a general morphism of $\T A$.
\end{notn}

Using \cref{notn:TA}, we now define the pseudomorphism classifier $\Q$ for each of the three cases $\T \in \{\M,\S,\B\}$.
\begin{defn}\label{defn:QA}
  Suppose $A$ is a category.
  Define a $\T$-algebra $\Q A$ as follows.
  \begin{description}
  \item[Objects:] The objects of $\Q A$ are those of $\T A$.
  \item[Free morphisms:] The morphisms of $\T A$ are included as morphisms of $\Q A$, and are called \emph{free morphisms} there.
    The inclusion of objects and free morphisms is denoted
    \begin{equation}\label{eq:iota-free}
      \iota\cn \T A \hookrightarrow \Q A.
    \end{equation}
    When describing individual objects or morphisms, we will often suppress $\iota$ and identify objects and morphisms of $\T A$ with their images in $\Q A$.
  \item[Adjoined isomorphisms:] For each object $\ang{a}$ in $\ob \bigl( \Q A \bigr) = \ob \bigl( \T A \bigr)$, there is an \emph{adjoined isomorphism}
    \[
      \qsf_{\ang{a}} \cn \ang{a} \fto{\iso} (a_\bullet) \inspace \Q A.
    \]
  \end{description}

  The morphisms of $\Q A$ are generated under composition and concatenation by the free morphisms and adjoined isomorphisms, subject to the following axioms.
  In the symmetric or braided cases $\T \in \{\S,\B\}$, the symmetry or braiding isomorphism of $\Q A$ is given by the corresponding free morphism from $\T A$. 
  \begin{description}
  \item[Free composites and products:]
    The inclusion $\iota$ is a strict $\T$-map.
    Thus, composites or products of free morphisms are given by those of $\T A$.
  \item[Naturality of $\qsf$:]
    The adjoined isomorphisms $\qsf$ are natural with respect to free morphisms.
    That is, using the notation \cref{eq:mul-bullet} and suppressing $\iota$, the following diagram commutes for each morphism $t \cn \ang{a_i}_{i=1}^n \to \ang{a'_i}_{i = 1}^n$ in $\T A$.
    \begin{equation}\label{eq:natq}
      \begin{tikzpicture}[x=20ex,y=8ex,vcenter]
        \draw[0cell] 
        (0,0) node (a) {\ang{a}}
        (a)++(1,0) node (a') {\ang{a'}}
        (a)++(0,-1) node (b) {(a_\bullet)}
        (a')++(0,-1) node (b') {(a'_\bullet)}
        ;
        \draw[1cell] 
        (a) edge node {t} (a')
        (b) edge node {(t_\bullet)} (b')
        (a) edge['] node {\qsf_{\ang{a}}} (b)
        (a') edge node {\qsf_{\ang{a'}}} (b')
        ;
      \end{tikzpicture}
    \end{equation}
  \item[Associativity of $\qsf$:]
    The following diagrams commute for tuples $\ang{a^1}$, $\ang{a^2}$, and $\ang{a^3}$ in $\Q A$, where the diagram at left uses the fact that $e$ is a strict unit for $A$.
    \begin{equation}\label{eq:unassocq}
      \begin{tikzpicture}[x=13ex,y=10ex,baseline=(x.base),xscale=1.2]
        \draw[0cell] 
        (0,0) node (a) {\ang{a^1};\ang{};\ang{a^2}}
        (a)++(1,0) node (b)  {\ang{a^1};\ang{a^2}}
        (a)++(-.2,-1) node (c) {\ang{a^1};(e);\ang{a^2}}
        (b)++(.2,-1) node (d) {(a^{1;2}_\bullet)}
        ;
        \draw[1cell] 
        (a) edge[equal] node {} (b)
        (a) edge['] node[pos=.7] (x) {1;\qsf_{\ang{}};1} (c)
        (c) edge node {\qsf} (d)
        (b) edge node[pos=.6] {\qsf} (d)
        ;
      \end{tikzpicture}
      \qquad
      \begin{tikzpicture}[x=21ex,y=10ex,baseline=(x.base),xscale=1.2]
        \draw[0cell] 
        (0,0) node (a) {\ang{a^1};\ang{a^2};\ang{a^3}}
        (a)++(1,0) node (b) {\ang{a^1};(a^{2;3}_\bullet)}
        (a)++(0,-1) node (c) {(a^{1;2}_\bullet);\ang{a^3}}
        (b)++(0,-1) node (d) {(a^{1;2;3}_\bullet)}
        ;
        \draw[1cell] 
        (a) edge node {1;\qsf_{\ang{a^{2;3}}}} (b)
        (b) edge node {\qsf} (d)
        (a) edge['] node (x) {\qsf_{\ang{a^{1;2}}};1} (c)
        (c) edge node {\qsf} (d)
        (a) edge node {\qsf} (d)
        ;
      \end{tikzpicture}
    \end{equation}

  \item[Normality of $\qsf$:]
    For a tuple of length one, $(a)$ with $a \in A$, we have
    \begin{equation}\label{eq:normq}
      \qsf_{(a)} = 1_{(a)} = (1_a).
    \end{equation}
  \end{description}
\end{defn}

\begin{defn}\label{defn:Qf}
  Suppose given a $\T$-map $f \cn A \zzto B$ between $\T$-algebras $A$ and $B$.
  Define a strict $\T$-map
  \[
    \Q f \cn \Q A \to \Q B
  \]
  as follows.
  For a tuple of objects $\ang{a}$, define
  \[
    (\Q f)\ang{a_i}_{i=1}^n = \ang{f(a_i)}_{i=1}^n.
  \]
  For a free morphism $t \cn \ang{a_i}_{i=1}^n \to \ang{a'_i}_{i = 1}^n$, define $\Q f$ as $\T f$.
  That is, define
  \[
    \bigl(\Q f\bigr)(\iota t) = \iota\bigl( \bigl( \T f \bigr)t\bigr) \cn \ang{f(a_i)}_{i=1}^n \to \ang{f(a'_i)}_{i=1}^n.
  \]
  For an adjoined isomorphism $\qsf_{\ang{a}}$, where $\ang{a} = \ang{a_i}_{i=1}^n$, define $\bigl(\Q f\bigr)\qsf_{\ang{a}}$ as the composite
  \begin{equation}\label{eq:fka}
    \ang{f(a_i)} \fto{\qsf_{\ang{f(a_i)}}} ([f(a_i)]_\bullet) \fto{(f_\bullet)} (f(a_\bullet)),
  \end{equation}
  where $[f(a_i)]_\bullet$ denotes the product of the entries $f(a_i)$ and 
  \[
    f_\bullet \cn [f(a_i)]_\bullet \to f(a_\bullet)
  \]
  is the notation of \cref{eq:Tmap-2cell} to indicate any composite of, respectively,
  \begin{itemize}
  \item monoidal constraints $f_2$, if $n \ge 2$,
  \item unit constraints $f_0$, if $n = 0$, or
  \item identities $1_{f(a)}$, if $n = 1$.
  \end{itemize}

  This defines $\Q f$ on the objects and generating morphisms of $\Q A$.
  Then, $\Q f$ is defined to be functorial with respect to formal composition $\circ$ and strict monoidal with respect to concatenation $;$ in $\Q A$ and $\Q B$.
  In the symmetric or braided cases, $\T \in \{\S,\B\}$, the definition of $\Q f$ on free morphisms implies that $\Q f$ satisfies the additional braid axiom \cref{eq:braid-axiom} of a $\T$-map.
  In all three cases for $\T$, we have $\Q f \circ \iota = \iota \circ \T f$ as strict $\T$-maps.

  To verify that $\Q f$ is well defined with respect to the relations \cref{eq:natq,eq:unassocq,eq:normq}, one uses the corresponding relations in the codomain $\T$-algebra $B$ together with functoriality of $f$ and naturality of $f_\bullet$.
  Furthermore, naturality of $\qsf$ and the definition of composition for monoidal functors shows that $\Q$ is functorial with respect to identities and composites of $\T$-maps.
\end{defn} 
\begin{defn}\label{defn:Qal}
In each of the cases $\T \in \{\M, \S, \B\}$, the $\T$-algebra 2-cells are monoidal transformations.
  If $\al\cn f \to f'$ is a monoidal transformation between $\T$-maps $f,f'\cn A \zzto B$, then 
  \[
    \Q \al \cn \Q f \to \Q f'
  \]
  is defined componentwise for objects $\ang{a} = \ang{a_i}_{i=1}^n$ by
  \begin{equation}\label{eq:Qal-a}
    \bigl( \Q \al \bigr)_{\ang{a}} = \ang{\al_{a_i}}_{i=1}^n.
  \end{equation}
  The monoidal transformation axioms for $\Q \al$ hold because concatenation of tuples is strictly associative and unital.
  Similarly, 2-functoriality of $\Q$ with respect to identities and horizontal or vertical composites of monoidal transformations is verified componentwise.
\end{defn}

Together, \cref{defn:QA,defn:Qf,defn:Qal} define a 2-functor
\[
  \Q \cn \Talg \to \Talgs.
\]
Recall from \cref{defn:effectiveQi} that a pseudomorphism classifier $(\Q,\itt,\ze,\de)$ is effective if the unit/counit pair $(\ze,\de)$ is componentwise an adjoint surjective equivalence.
The following will be used in \cref{prop:Qm-effpsmorcl} below to show that $\Q$ is an effective pseudomorphism classifier for $\T$.
\begin{defn}\label{defn:effecQ}
  In the context of \cref{defn:QA,defn:Qf,defn:Qal} above, there are 2-natural transformations $\ze$ and $\de$ together with an invertible monoidal transformation $\Theta$ defined as follows.
  \begin{description}
  \item[Unit:] For a $\T$-algebra $A$, define a $\T$-map
    \[
      \ze_A \cn A \zzto \itt\Q A
    \]
    by sending each object and morphism of $A$ to the corresponding length-one tuple in $\Q A$.
    The monoidal and unit constraints of $\ze$ are given by the adjoined isomorphisms $\qsf$.
    Thus, in the symmetric and braided cases $\T \in \{\S,\B\}$, $\ze_A$ satisfies the braid axiom \cref{eq:braid-axiom}.
    Naturality of $\ze$ with respect to $\T$-maps $f$ holds because $\Q f$ is defined by $\T f$ on tuples $\ang{a}$ and free morphisms $t$.
    Likewise, 2-naturality with respect to monoidal transformations follows from \cref{eq:Qal-a}.
  \item[Counit:] For a $\T$-algebra $B$, define a strict $\T$-map
    \[
      \de_{B} \cn \Q \itt B \to B
    \]
    by sending each tuple of objects $\ang{a}$ to their product $a_\bullet$ in $B$, each free morphism $t$ to $t_\bullet$, and each adjoined isomorphism $\qsf$ to an identity. 
    Thus, in the symmetric or braided cases $\T\in\{\S,\B\}$, $\de_B$ satisfies the braid axiom \cref{eq:braid-axiom}.
    This is a strict $\T$-map because the monoidal product in $B$ is strictly associative and unital.
    Naturality of $\de$ with respect to strict $\T$-maps holds because such $\T$-maps strictly preserve monoidal units and products.
  \item[Efficacy:]
    For each $\T$-algebra $B$, define an invertible monoidal transformation
    \[
      \Theta \cn \ze_B \de_B \fto{\iso} 1_{\Q B}
    \]
    with components
    \[
      \Theta_{\ang{b}} = \qsf^\inv_{\ang{b}} \cn (b_\bullet) \to \ang{b}
      \forspace \ang{b} \in \Q B.
    \]
    Monoidal naturality of $\qsf$, and hence also $\Theta$, is equivalent to the conditions \cref{eq:natq,eq:unassocq}. 
  \end{description}
\end{defn}
  
\begin{prop}\label{prop:Qm-effpsmorcl}
  For each $\T \in \{\M,\S,\B\}$, the 2-functor
  \[
    \Q \cn \Talg \to \Talgs
  \]
  is an effective pseudomorphism classifier for $\T$.
\end{prop}
\begin{proof}
  The 2-functor $\Q$, unit $\ze$, counit $\de$, and isomorphism $\Theta$ are given in \cref{defn:QA,defn:Qf,defn:Qal,defn:effecQ,}.
  For $\T$-algebras $A$ and $B$, the definitions of $\de$ and $\ze$ yield the following computations:
  \begin{align*}
    \de_B \bigl( \ze_B(b) \bigr)
    & = b
    & & \forspace b \in B\\
    \de_{\Q A} \bigl( \bigl( \Q\ze_A \bigr)\ang{a} \bigr)
    & = \de_{\Q A} \ang{(a_i)}_{i=1}^n
    = \ang{a}
    & & \forspace \ang{a} = \ang{a_i}_{i=1}^n \in \Q A.
  \end{align*}
  A similar computation holds for morphisms, using the fact that the monoidal constraints of $\ze$ are the adjoined isomorphisms $\qsf$.
  Thus, $\ze$ and $\de$ satisfy the triangle identities
  \[
    \de_B \circ \ze_B = 1_B
    \andspace
    \de_{\Q A} \circ (\Q \ze_A) = 1_{\Q A}.
  \]
  so that $(\Q,\itt,\ze,\de)$ is a 2-adjunction.
  
  Furthermore, the normality condition \cref{eq:normq} for $\qsf$ implies
  \begin{equation}\label{eq:Theta-ze-1}
    \Theta * \ze_B = 1_{\ze_B}.
  \end{equation}
  This completes the proof.
\end{proof}

The explicit description of $\Q$, above, will be helpful in \cref{sec:applications} below.
The following alternative description of $\Q$ is more abstract, but highlights some of its characteristic properties.
\begin{rmk}\label{rmk:Q-alt}
  The strict $\T$-map $\iota\cn \T A \to \Q A$ from \cref{eq:iota-free} is the identity on objects and factors the monad structure morphism $\T A \to A$ as shown at left below.
  Furthermore, there is a $\T$-map $\ze_A \cn A \zzto \Q A$ such that the adjoined isomorphisms $\qsf$ are the components of an invertible monoidal transformation as shown at right below.
  \begin{equation}\label{eq:MA-QmA}
    \begin{tikzpicture}[x=13ex,y=8ex,vcenter,xscale=1.2]
      \draw[0cell] 
      (0,0) node (a) {\T A}
      (a)++(1,-1) node (d) {\Q A}
      (d)++(0,1) node (b) {A}
      ;
      \draw[1cell] 
      (a) edge node {\bcdot} (b)
      (a) edge['] node {\iota} (d)
      (d) edge['] node {\de_A} (b)
      ;
    \end{tikzpicture}
    \qquad
    \begin{tikzpicture}[x=13ex,y=8ex,vcenter,xscale=1.2]
      \draw[0cell] 
      (0,0) node (a) {\T A}
      (a)++(1,-1) node (d) {\Q A}
      (d)++(-1,0) node (c) {A}
      ;
      \draw[1cell] 
      (a) edge['] node {\bcdot} (c)
      (a) edge node {\iota} (d)
      ;
      \draw[zz1cell]
      (c) to['] node {\ze_A} (d)
      ;
      \draw[2cell]
      (c)++(50:.45) node[rotate=225,2label={below,\qsf}] {\Rightarrow}
      ;
    \end{tikzpicture}
  \end{equation}
  The normality condition \cref{eq:normq} for $\qsf$ is equivalent to the equality
  \begin{equation}\label{eq:normq-alt}
    \qsf * \eta_A = 1_{\ze_A}.
  \end{equation}
  That is, the whiskering of $\qsf$ with the unit $\eta_A \cn A \to \T A$ is the identity transformation of $\ze_A$.
  In this context, the requirement in \cref{defn:effecQ}, that $\de$ sends the adjoined isomorphisms $\qsf$ to identities, is equivalent to the requirement that $\de_A \ze_A = 1_A$ as a strict $\T$-map.
\end{rmk}

\begin{rmk}\label{rmk:power-lack}
  The description in \cref{rmk:Q-alt} indicates how the elementary presentation above in \cref{defn:QA,defn:Qf,defn:Qal} relates to the method of Power \cite[Theorem~3.4]{power1989coherence}, which constructs a pseudomorphism classifier $\Q$ in greater generality by factoring the multiplication morphism of a $\T$-algebra (or pseudo algebra) $(X,\Xmul)$ as a
  \emph{bijective-on-objects} functor $\iota$ followed by a \emph{full and faithful} functor $\delta_A$. 
  In our applications, the left side of \cref{eq:MA-QmA} provides this factorization.
  Power's work can be extended in greater generality via Lack's codescent for pseudo-algebras \cite[Theorem~4.10]{Lac02Codescent}.
\end{rmk}

\begin{defn}\label{defn:QA-etc}
  In the context of \cref{defn:QA,defn:Qf,defn:Qal}, the following associated constructions are of interest.
  These are special cases of the general constructions in \cref{defn:fbot,lem:zeflat,rmk:free-flex}.
  \begin{enumerate}
  \item Referring to the adjunction $\Q \dashv \itt$, each $\T$-map $f \cn A \zzto B$, has a unique strict mate $f^\bot$, making the triangle at left below commute in $\Talg$.
    Recalling \cref{defn:Qf,defn:effecQ}, one verifies that $f^\bot$ is defined by the triangle at right below.
    \[
      \begin{tikzpicture}[x=12ex,y=8ex,vcenter,xscale=1.2]
        \draw[0cell] 
        (0,0) node (a) {A}
        (a)++(0,1) node (qa) {\Q A}
        (qa)++(1,0) node (b) {B}
        ;
        \draw[zz1cell] 
        (a) -- node['] {f} (b)
        ;
        \draw[zz1cell] 
        (a) -- node {\ze_A} (qa)
        ;
        \draw[1cell] 
        (qa) edge node {f^\bot} (b)
        ;
      \end{tikzpicture}
      \qquad
      \begin{tikzpicture}[x=12ex,y=8ex,vcenter,xscale=1.2]
        \draw[0cell] 
        (a)++(0,1) node (qa) {\Q A}
        (qa)++(1,0) node (b) {B}
        (qa)++(1,1) node (qb) {\Q B}
        ;
        \draw[1cell] 
        (qa) edge node {f^\bot} (b)
        (qa) edge node {\Q f} (qb)
        (qb) edge node {\de_B} (b)
        ;
      \end{tikzpicture}
    \]
  \item In the case $A = \T C$ for a category $C$, there is a strict $\T$-map
    \begin{equation}\label{eq:QA-zeflat}
      \ze^\flat \cn \T C \to \Q \T C
    \end{equation}
    that sends a tuple of objects $\ang{a_i}_{i = 1}^n$ in $\T C$ to the corresponding tuple of length-one tuples $\ang{(a_i)}_{i=1}^n$ in $\Q \T C$.
    In the case $n = 0$, $\ze^\flat$ sends the empty tuple $\ang{} \in \T C$ to the empty tuple $\ang{} \in \Q \T C$.
    The assignment on morphisms is given in the same way, and $\ze^\flat$ is a strict monoidal functor.
  \item There is an invertible monoidal transformation
    \[
      \Theta^\flat \cn \ze^\flat_{\T C} \de_{\T C} \fto{\iso} 1_{\Q \T C}
    \]
    defined as in \cref{eq:Thetaflat}.
    For an object
    \[
      \ang{w} = \ang{w_j}_{j=1}^m \in \Q \T C,
    \]
    where $w_j = \ang{a^j_i}_{i=1}^{n_j}$ is an object of $\T C$ for each $j \in \{1,\ldots,m\}$, the component
    \[
      \Theta^\flat_{\ang{w}} \cn \ze^\flat \de \ang{w} \to \ang{w}
    \]
    is given by the composite 
    \begin{equation}\label{eq:Thetaflat-q-qinv}
      \ze^\flat \de \ang{w} \fto{\qsf} (\ang{w_\bullet}) \fto{\qsf^\inv} \ang{w}
    \end{equation} 
    Here, each $\qsf$ is one of the adjoined isomorphisms in $\Q \T C$, the object $(\ang{w_\bullet})$ is the length-one tuple whose entry is the concatenation in $\T C$ of the tuples $w_j = \ang{a^j}$, and $\ze^\flat \de \ang{w} = \ang{(a^j_i)}_{j,i}$ is the tuple of length $N = \sum_j n_j$ whose entries are the length-one tuples $(a^j_i)$.

    If $\ang{w} = \ze^\flat \ang{a}$ for $\ang{a} \in \T C$, then the two components of $\qsf$ appearing in \cref{eq:Thetaflat-q-qinv} are the same.
    Hence, $\Theta^\flat * \ze^\flat = 1_{\ze^\flat}$ as required.
  \end{enumerate}
\end{defn}

\section{Explication: Universal pseudomorphisms}\label{sec:applications}

The hypotheses of \cref{thm:main1} hold for $\K = \Cat$ and each of the 2-monads for monoidal structures $\Tg$ and $\T$ in \cref{notn:monoidal-variants}, with $\T \in \{\M,\S,\B\}$.
Therefore, the comparison strict $\T$-maps
\[
  \Tg(C',\phi) \fto{\De} \Tg C'
  \andspace
  \T(C',\phi) \fto{\De} \T C'
\]
are equivalences for each $\phi \cn C \to C'$ in $\Cat$.

This section gives an explicit description of $\T(G',\phi)$ in \cref{defn:TGpphi}, where $\phi \cn G \to G'$ is a function between sets, treated as discrete categories.
Then, the universal $\T$-map $\wt{\phi}$ for $\T(G',\phi)$ and the comparison $\De$ are described in \cref{defn:wtphi-application,defn:De-application}, respectively.
In applications, $\phi$ is the underlying function-on-objects of a $\T$-map $f$.
In that case, the strict $\T$-map $\La$ of \cref{eq:La} is described in \cref{defn:La-application}.

To begin, it will be useful to record the following.
\begin{defn}\label{defn:disc-ob-ind}
  Let $\Mon$ denote the category of monoids in $\Set$.
  The \emph{set-of-objects} functor
  \[
    \ob\cn \Malg \to \Mon
  \]
  has both left and right adjoints
  \begin{equation}\label{eq:disc-ob-ind}
    \disc \dashv \ob \dashv \indisc
  \end{equation}
  defined as follows.
  \begin{itemize}
  \item $\disc \cn \Mon \to \Malgs$ is the \emph{discrete $\M$-algebra} functor, sending a monoid $G$ to the $\M$-algebra with underlying monoid $G$ and identity morphisms.
  \item $\indisc \cn \Mon \to \Malgs$ is the \emph{indiscrete $\M$-algebra}, sending a monoid $G$ to the $\M$-algebra with underlying monoid $G$ and a unique isomorphism between every pair of objects.
  \end{itemize}
  Below, we will apply $\disc$ implicitly and omit the notation.
\end{defn}

Recall from \cref{thm:finitary-udc} that each universal pseudomorphism for $\T \in \{\M,\S,\B\}$ can be obtained as a pushout of strict $\T$-maps \cref{eq:udc-pushout} shown here.
\begin{equation}\label{eq:udc-pushout-again}
  \begin{tikzpicture}[x=14ex,y=8ex,vcenter,xscale=1.2]
    \draw[0cell] 
    (0,0) node (a) {\T G}
    (a)++(1,0) node (b) {\T G'}
    (a)++(0,-1) node (c) {\itt \Q \T G}
    (c)++(1,0) node (d) {\T (G',\phi)}
    ;
    \draw[1cell] 
    (a) edge node {\T\phi} (b)
    (a) edge['] node {\ze^\flat} (c)
    (c) edge node {\wh{\phi}} (d)
    (b) edge node {\ka} (d)
    ;
  \end{tikzpicture}
\end{equation}
Recall that $\Q$ is described in \cref{defn:QA} using \cref{notn:TA}; recall $\ze^\flat$ from \cref{eq:QA-zeflat}.
Unpacking \cref{eq:udc-pushout-again} yields the following.
\begin{explanation}\label{defn:TGpphi}
  Suppose $\phi\cn G \to G'$ is a functor between discrete categories and $\T \in \{\M,\S,\B\}$.
  The $\T$-algebra $\T(G',\phi)$ in \cref{eq:udc-pushout-again} is given as follows.
  We begin by describing generating objects and their relations.
  Then, we describe generating morphisms and their relations

  The symmetric and braided cases $\T \in \{\S,\B\}$ have the same objects as the plain monoidal case.
  In the monoidal case, $\T = \M$, the functor $\ob$ is left adjoint to $\indisc$ in \cref{eq:disc-ob-ind} and therefore commutes with pushouts.

  Thus, in each of the cases $\T \in \{\M,\S,\B\}$, the objects of $\T(G',\phi)$ are given by the pushout \cref{eq:udc-pushout-again} on objects.
  Hence, the objects are generated under the monoidal product $;$ by those of $\T G'$ and $\Q\T G$, for which we use the following terms.
  \begin{description}
  \item[Free objects:] The \emph{free objects} of $\T(G',\phi)$ are those of $\T G'$.
    They are tuples
    \[
      w' = \ang{a'} = \ang{a'_i}_{i=1}^{n'}
    \]
    where $a_i' \in G'$ and $n' \ge 0$.
    On objects, the functor $\ka\cn \T G' \to \T(G',\phi)$ is the inclusion of free objects.
  \item[$\phi$-Objects:] The \emph{$\phi$-objects} of $\T(G',\phi)$ are tuples denoted
    \[
      \bang{[\phi]w} = \bang{[\phi]w_j}_{j=1}^m
    \]
    where each $\ang{w}$ is an object of $\Q\T G$, so $w_j = \ang{a^j_i}_{i=1}^{n_j}$ is an object of $\T G$, and $m \ge 0$.
    On objects, the functor $\wh{\phi}$ sends an object $\ang{w} \in \Q\T G$ to the $\phi$-object $\bang{[\phi]w}$.
  \end{description}
  These objects are subject to the following relation, identifying the two composites around \cref{eq:udc-pushout-again}. 
  \begin{description}
  \item[Object pushout relation:] If $\ang{w} = \ze^\flat\ang{a} = \bang{(a^j_1)}_{j=1}^m$ is a tuple of length-one tuples, then
    \begin{equation}\label{eq:len1-reln}
      \bang{[\phi](a^j_1)}_{j=1}^m = \ang{\phi(a^j_1)}_{j=1}^m,
    \end{equation}
    where
    \begin{itemize}
    \item the left hand side is the $\phi$-object associated to the tuple $\ang{w}$ whose entries are length-one tuples $(a^j_1)$, and
    \item the right hand side is the free object whose entries are $\phi(a^j_1)$.
    \end{itemize}
    In the case that $m = 0$, the empty $\phi$-object $\bang{[\phi]}$ is identified with the empty free object $\ang{}$.
  \end{description}
  This finishes the description of the objects of $\T(G',\phi)$.

  The morphisms of $\T(G',\phi)$ are likewise generated by those of $\T G'$ and $\Q\T G$ under composition $\circ$ and the product $;$.
  In the symmetric and braided cases $\T \in \{\S,\B\}$, there are additional formal braid isomorphisms.
  Thus, the morphisms of $\T(G',\phi)$ are generated by four types, for which we use the following terms.
  \begin{description}
  \item[Free morphisms:] The \emph{free morphisms} are those of $\T G'$.
    On morphisms, the functor $\ka$ is the inclusion of free morphisms.

  \item[$\phi$-Free morphisms:] The \emph{$\phi$-free morphisms} are denoted
    \begin{equation}\label{eq:phi-free}
      [\phi]u \cn \bang{[\phi]w} \to \bang{[\phi]v}
    \end{equation}
    where $u \cn \ang{w} \to \ang{v}$ is a free morphism of $\Q \T G$.
    Thus, $u$ is either
    \begin{itemize}
    \item a tuple of morphisms $t_j\cn w_j \to v_j$ in $\T G$;
    \item a permutation or braiding, in the symmetric and braided cases $\T \in \{\S,\B\}$; or
    \item a composite of such morphisms. 
    \end{itemize}
    In the former case, since $G$ is discrete, each $t_j$ is either a tuple of identity morphisms or, in the cases $\T \in \{\S,\B\}$, a permutation or braiding in $\T G$.

  \item[$\phi$-Adjoined isomorphisms:] The \emph{$\phi$-adjoined morphisms} are denoted
    \begin{equation}\label{eq:phi-adjoined}
      [\phi]\qsf_{\ang{w}}\cn \bang{[\phi]w} \to \bigl([\phi]w_\bullet\bigr)
    \end{equation}
    where $\ang{w} = \ang{w_j}_{j=1}^m$ is an object of $\Q\T G$ and $w_\bullet$ denotes the concatenation in $\T G$ of the tuples $w_j = \ang{a^j_i}_{i=1}^{n_j}$.
    Thus, $w_\bullet = \ang{a^\bullet}$ is a tuple of length $N = \sum_j n_j$ whose $\ell$th entry, $a^\bullet_\ell$, is $a^J_i$, where 
    \[
      J \in \{1,\ldots,m\} \andspace i \in \{1,\ldots,n_{J}\}
    \]
    are the unique natural numbers such that
    \begin{equation}\label{eq:abullet}
        \ell = \left[\sum_{j=1}^{J-1} n_j\right] + i .
    \end{equation}
  \item[Formal Morphisms:]
    In the symmetric and braided cases, $\T \in \{\S,\B\}$, there are formal permutation and braid morphisms, respectively.
    The formal morphisms between free objects are identified with the corresponding free morphisms given by permutation or braid morphisms in $\T G'$.
    The formal morphisms between $\phi$-objects are identified with the corresponding $\phi$-free morphisms given by permutation or braid morphisms in $\Q \T G$.
  \end{description}

  The morphisms of $\T (G',\phi)$ are freely generated under composition $\circ$ and the product $;$ so that the $\T$-algebra structure on $\T(G',\phi)$ extends that of $\T G'$ and $\Q \T G$, subject to the following axioms.
  \begin{description}
  \item[Composites and products:] The structure morphisms
    \[
      \ka \cn \T G' \to \T(G',\phi)
      \andspace
      \wh{\phi} \cn \Q\T G \to \T(G',\phi)
    \]
    are both strict $\T$-maps.
    Thus, the composites or products of free, respectively $\phi$-, morphisms are given by those of $\T G'$, respectively $\Q\T G$.

  \item[Morphism pushout relation:] For each morphism $t \cn \ang{a} \to \ang{b}$ in $\T G$, where $\ang{a} = \ang{a_i}_{i=1}^n$ and $\ang{b} = \ang{b_i}_{i=1}^n$, the images of $t$ under the two composites around \cref{eq:udc-pushout-again} are identified.
    Thus, the free morphism
    \[
      \bigl(\T\phi\bigr)(t) \cn \bigl(\T\phi\bigr)(\ang{a}) \to \bigl(\T\phi\bigr)(\ang{b})
    \]
    is identified with the $\phi$-free morphism
    \[
      [\phi]\ol{t} \cn \bang{[\phi](a_i)}_{i=1}^n \to \bang{[\phi](b_i)},
    \]
    where $\ol{t} = \ze^\flat t$ is the free morphism induced by $t$, between tuples of length-one tuples $\ang{(a_i)}_{i=1}^n$ and $\ang{(b_i)}_{i=1}^n$.

    Since $G$ is discrete, this relation is trivial if $\T = \M$, in which case $t$ is a tuple of identity morphisms.
    If $\T \in \{\S,\B\}$, then $(\T\phi)t$, $\ol{t}$, and $[\phi]\ol{t}$ are the respective permutation or braiding morphisms determined by $t$.
  \end{description}
  This finishes the description of objects, morphisms, and $\T$-algebra structure of $\T(G',\phi)$.
\end{explanation}

\begin{prop}\label{prop:TGpphi}
  The $\T$-algebra described in \cref{defn:TGpphi} is a model for the pushout $\T(G',\phi)$ in \cref{eq:udc-pushout-again}.
\end{prop}

Now we describe the universal pseudomorphism
\[
  \wt{\phi}\cn \T G \zzto \T(G',\phi).
\]
Recalling \cref{eq:wtphi-kaphi}, $\wt{\phi}$ is equal to the composite $\wh{\phi} \circ \ze$ shown below.
\begin{equation}\label{eq:wtphi-again}
  \begin{tikzpicture}[x=18ex,y=12ex,vcenter,xscale=1.2,yscale=1.1]
    \draw[0cell] 
    (0,0) node (a) {\T G}
    (a)++(1,0) node (c) {\T(G',\phi)}
    (a)++(.5,-.5) node (b) {\Q\T G}
    ;
    \draw[1cell] 
    (b) edge['] node {\wh{\phi}} (c)
    ;
    \draw[zz1cell]
    (a) to['] node {\ze_{\T C}} (b)
    ;
    \draw[zz1cell]
    (a) to node {\wt{\phi}} (c)
    ;
  \end{tikzpicture}
\end{equation}
Recall $\ze$ is the unit in \cref{defn:effecQ}.
\begin{explanation}\label{defn:wtphi-application}
  In the context of \cref{defn:TGpphi,eq:alt-udc-again}, the $\T$-map
  \[
    \wt{\phi}\cn \T G \zzto \T(G',\phi)
  \]
  in \cref{eq:wtphi-again} is given as follows.
  \begin{enumerate}
  \item For a tuple $w = \ang{a_i}_{i=1}^n \in \T G$, with each $a_i \in G$,
    \[
      \wt{\phi}w = ([\phi]w) 
    \]
    is the $\phi$-object of length one whose only entry is $[\phi]w$.
  \item For a morphism $t \cn w \to v$ in $\T G$,
    \[
      \wt{\phi}t = [\phi]t \cn ([\phi]w) \to ([\phi]v)
    \]
    is the $\phi$-free morphism of length one whose entry is either the identity, if $\T = \M$, or the permutation or braid morphism corresponding to $t$ if $\T \in \{\S,\B\}$.
  \item The unit constraint of $\wt{\phi}$ is given by the $\phi$-adjoined isomorphism for the empty tuple:
    \[
      [\phi]\qsf_{\ang{}} \cn \bang{[\phi]} = \ang{} \to \bigl([\phi]\ang{}\bigr) = \bigl(\ang{}\bigr)
    \]
    where, on the right hand side, $\bigl([\phi]\ang{}\bigr) = \bigl(\ang{}\bigr)$ is the $\phi$-object of length one whose single entry is $[\phi]\ang{} = \ang{}$.
  \item The monoidal constraint of $\wt{\phi}$ is given, at a pair of objects $w_1,w_2\in \T G$, by the $\phi$-adjoined isomorphism for the length-two tuple $\ang{w} = \ang{w_i}_{i=1}^2$:
    \[
      [\phi] \qsf_{\ang{w}} \cn \bang{[\phi]w_i}_{i=1}^2 \to
      \bigl( [\phi]w_\bullet \bigr)
      = \bigl( [\phi]\ang{w_{1;2}} \bigr).
    \] 
  \end{enumerate}
  The description of the unit and monoidal constraints of $\wt{\phi}$ follows from those of $\ze$ in \cref{defn:effecQ}.
\end{explanation}

Now we describe the strict $\T$-map
\[
  \De \cn \T(G',\phi) \fto{\hty} \T G'.
\]
By \cref{rmk:alt-udc}, $\De$ is an adjoint surjective equivalence that is determined by the pushout \cref{eq:udc-pushout-again}, as indicated by the dashed arrow below.
\begin{equation}\label{eq:alt-udc-again}
  \begin{tikzpicture}[x=14ex,y=8ex,vcenter,xscale=1.2]
    \draw[0cell] 
    (0,0) node (a) {\T G}
    (a)++(1,0) node (b) {\T G'}
    (a)++(0,-1) node (c) {\itt \Q\T G}
    (c)++(1,0) node (d) {\T(G',\phi)}
    (d)++(1,-1) node (y) {\T G'}
    (c)++(0,-1) node (x) {\T G}
    ;
    \draw[1cell] 
    (a) edge node {\T\phi} (b)
    (a) edge['] node {\ze^\flat} (c)
    (c) edge node {\wh{\phi}} (d)
    (b) edge node {\ka} (d)
    (c) edge['] node {\de} (x)
    (d) edge[dashed] node {\exists!} node['] {\De} (y)
    (b) edge[bend left=18] node {1} (y)
    (x) edge node {\T\phi} (y)
    ;
  \end{tikzpicture}
\end{equation}
Recall $\de$ is the counit in \cref{defn:effecQ}.
Using the description of $\T(G',\phi)$ in \cref{defn:TGpphi} and commutativity of \cref{eq:alt-udc-again} yields the following.
\begin{explanation}\label{defn:De-application}
  In the context of \cref{defn:TGpphi,eq:alt-udc-again}, the strict $\T$-map
  \[
    \De \cn \T (G',\phi) \to \T G'
  \]
  is given as follows.

  \begin{enumerate}
  \item For free objects $\ang{a'}, \ang{b'} \in \T G'$ and free morphisms $t'\cn \ang{a'} \to \ang{b'}$,
    \[
      \De t' = t' \cn \ang{a'} \to \ang{b'}.
    \]
  \item For $\phi$-objects $\bang{[\phi]w} = \bang{[\phi]w_j}_{j=1}^m$ with $w_j = \ang{a^j} = \ang{a^j_i}_{i=1}^{n_j} \in \T G$,
    \[
      \De \bang{[\phi]w} = \bigl( T\phi \bigr)\de\ang{w}
      = \bigl( T\phi \bigr) w_\bullet
      = \bang{\phi(a^\bullet_\ell)}_{\ell = 1}^N
    \]
    where $\ang{a^\bullet} = w_\bullet$ is the concatenation in $\T G$ of the tuples $w_j = \ang{a^j}$ as in \cref{eq:abullet}.
  \item\label{defn:De-application-iii} For $\phi$-free morphisms $[\phi]u\cn \bang{[\phi]w} \to \bang{[\phi]v}$ where $u\cn \ang{w} \to \ang{v}$ is a free morphism of $\Q\T G$,
    \[
      \De \bigl([\phi]u\bigr) = \bigl( T\phi \bigr)\de u
    \]
    is the corresponding identity, permutation, or braiding morphism 
    \[
      \bigl( T\phi \bigr)u \cn \bigl( T\phi \bigr)w_\bullet \to \bigl( T\phi \bigr)v_\bullet
    \]
    in $\T G'$.
  \item\label{defn:De-application-iv} For $\phi$-adjoined isomorphisms $[\phi]\qsf_{\ang{w}}\cn \ang{[\phi]w} \to ([\phi]w_\bullet)$,
    \[
      \De \bigl([\phi]\qsf_{\ang{w}}\bigr) =
      \bigl( T\phi \bigr)\de\qsf_{\ang{w}} =
      1\cn \bigl( T\phi \bigr)w_\bullet \to \bigl( T\phi \bigr)w_\bullet.
    \]
  \item For formal morphisms, in the cases $\T \in \{\S,\B\}$, $\De$ is a strict $\T$-map and so it sends the formal permutation or braiding morphisms of $\T(G',\phi)$ to corresponding permutations or braidings in $\T G'$.
  \end{enumerate}
  This completes the description of $\De$.
\end{explanation}

\begin{example}\label{example:De-phiw}
  In the context of \cref{defn:De-application}, suppose given $\ang{w} = (w_1,w_2)$ with
  \[
    w_1 = \ang{a^1} = (a^1_1,a^1_2,a^1_3)
    \andspace
    w_2 = \ang{a^2} = (a^2_1,a^2_2).
  \]
  Then $w_\bullet = \ang{a^\bullet} = (a^1_1,a^1_2,a^1_3,a^2_1,a^2_2)$ and
  \[
    \De\bang{[\phi]w} = \bigl(\,
    \phi(a^1_1) \scs \phi(a^1_2) \scs \phi(a^1_3) \scs \phi(a^2_1) \scs \phi(a^2_2)
    \,\bigr).
  \]
  Each braiding of tuples $w_j$ or entries $a^j_i$ is sent by $\De$ to the corresponding braiding of entries $\phi(a^j_i)$.
\end{example}

Now we describe the strict $\T$-map from \cref{eq:La}
\[
  \La \cn \T(\ob A',\phi) \to A',
\]
where $f \cn (A,\bcdot) \zzto (A',\bcdot)$ is a $\T$-map and $\phi = f_\ob$ denotes the restriction of $f$ to objects.
Recalling \cref{eq:La-diagram}, $\La$ is the unique strict $\T$-map induced by the universal property \cref{eq:rmk-udc} and the inclusions of objects.
The proof of \cref{thm:finitary-udc} explains how the pushout description of $\T(\ob A',\phi)$, as in \cref{eq:udc-pushout-again}, satisfies the universal property \cref{eq:rmk-udc}.
In particular, recalling \cref{eq:olS-pushout} with $\ol{S} = \La$ and $\ol{R}$ being the composite $\T(\ob A) \to \T A \fto{\bcdot} A$, the following diagram identifies $\La$ via the universal property of the pushout in $\Talgs$.
\begin{equation}\label{eq:La-pushout}
  \begin{tikzpicture}[x=16ex,y=10ex,vcenter,xscale=1.2]
      \draw[0cell] 
      (0,0) node (tc) {\T (\ob A)}
      (tc)++(1,0) node (tc') {\T (\ob A')}
      (tc)++(0,-1) node (iqtc) {\Q\T (\ob A)}
      (tc')++(0,-1) node (tc'phi) {\T(\ob A',\phi)}
      (tc'phi)++(1,0) node (ty) {\T A'}
      (ty)++(0,-1) node (y) {A'}
      (iqtc)++(.5,-.5) node (iqtx) {\Q \T A}
      (iqtx)++(.5,-.5) node (iqx) {\Q A}
      ;
      \draw[1cell] 
      (tc) edge node {\T\phi} (tc')
      (tc) edge['] node {\ze^\flat} (iqtc)
      (tc') edge['] node {\ka} (tc'phi)
      (iqtc) edge node {\wh{\phi}} (tc'phi)
      (tc') edge node {} (ty)
      (ty) edge node {\bcdot} (y)
      (iqtc) edge['] node {} (iqtx)
      (iqtx) edge['] node {\Q \bcdot} (iqx)
      (iqx) edge['] node {f^{\bot}} (y)
      (tc'phi) edge[dashed] node {\La} node['] {\exists !}(y)
      ;
    \end{tikzpicture}
\end{equation}
In the above diagram, $\T(\ob A',\phi)$ is described in \cref{defn:TGpphi}, with $G' = \ob A'$.
The strict $\T$-map $\Q\bcdot$ is an instance of $\Q$ applied to a $\T$-map, as in \cref{defn:Qf}.
The mate $f^\bot$ \cref{eq:fbot} is the unique strict $\T$-map that factors $f$ as below.
\begin{equation}\label{eq:fbot-again}
  \begin{tikzpicture}[x=12ex,y=8ex,vcenter,xscale=1.2]
    \draw[0cell] 
    (0,0) node (iqx) {\Q A}
    (iqx)++(0,-1) node (x) {A}
    (iqx)++(1,0) node (x') {A'}
    ;
    \draw[1cell] 
    (iqx) edge node {f^\bot} (x')
    ;
    \draw[zz1cell]
    (x) to['] node {f} (x')
    ;
    \draw[zz1cell]
    (x) to node {\ze_A} (iqx)
    ;
  \end{tikzpicture}
\end{equation}
For $\T \in \{\M,\S,\B\}$, the unit $\ze$ is described in \cref{defn:effecQ}.
Unpacking these, the following gives an explicit description of $\La$ on objects and morphisms.

\begin{explanation}\label{defn:La-application}
  Suppose $\T \in \{\M,\S,\B\}$ and suppose
  \[
    f\cn (A,\bcdot) \zzto (A',\bcdot)
  \]
  is a $\T$-map.
  Let $\phi = f_\ob$ denote the restriction of $f$ to objects, and recall from \cref{notn:TA} that subscripts $\bullet$ denote the image of free objects or morphisms under the multiplication $\bcdot$.
  Then the strict $\T$-map $\La$ in \cref{eq:La,eq:La-pushout} is given as follows.
  \begin{enumerate}
  \item For free objects $\ang{a'}, \ang{b'} \in \T (\ob A')$ and free morphisms between them, $t'\cn \ang{a'} \to \ang{b'}$, we have
    \[
      \La t' = t'_\bullet \cn a'_\bullet \to b'_\bullet.
    \]
  \item For $\phi$-objects $\bang{[\phi]w} = \bang{[\phi]w_j}_{j=1}^m$ with $w_j = \ang{a^j} = \ang{a^j_i}_{i=1}^{n_j} \in \T (\ob A)$, we have
    \[
      \La \bang{[\phi]w} = f^\bot(\ang{a^j_\bullet}_{j=1}^m) = f(a^1_\bullet)\bcdots f(a^m_\bullet)
    \]
    because $f^\bot$ is strict monoidal.
  \item For $\phi$-free morphisms $[\phi]u\cn \bang{[\phi]w} \to \bang{[\phi]v}$, where
    \[
      \ang{w} = \ang{w_j}_{j=1}^m, \quad \ang{v} = \ang{v_j}_{j=1}^m,
    \]
    and $u\cn \ang{w} \to \ang{v}$ is a free morphism of $\Q\T (\ob A)$, we have
    \[
      \La \bigl([\phi]u\bigr) = f^\bot(u_\bullet) \cn
      f(a^1_\bullet) \bcdots f(a^m_\bullet)
      \to
      f(b^1_\bullet) \bcdots f(b^m_\bullet).
    \]
    Here, each $w_j = \ang{a^j}$ and each $v_j = \ang{b^j}$ as above.

    If $u$ is a tuple of morphisms $t_j \cn w_j \to v_j$ in $\T(\ob A)$, then $u_\bullet$ is their product (concatenation) in $\T(\ob A)$.
    If $u$ is a permutation or braiding in $\T^2(\ob A)$, in the cases $\T \in \{\S,\B\}$, then $u_\bullet$ is the corresponding block permutation or braiding in $\T(\ob A)$.
    In either case, since $f^\bot$ is strict monoidal, $[\phi]u$ is sent to either the product of the morphisms $f(t_j)$ or to the permutation or braiding of $f(a^1_\bullet) \bcdots f(a^m_\bullet)$ determined by $u$.
  \item For $\phi$-adjoined isomorphisms $[\phi]\qsf_{\ang{w}}\cn \ang{[\phi]w} \to ([\phi]w_\bullet)$,
    \[
      \La \bigl([\phi]\qsf_{\ang{w}}\bigr) = f^\bot(\ze_\bullet) = f_\bullet \cn
      f(a^1_\bullet)\bcdots f(a^m_\bullet) \to f(a^1_\bullet \bcdots a^m_\bullet)
    \]
    because the morphisms $\qsf_{\ang{w}}$ are the monoidal and unit constraints of $\ze$ and \cref{eq:fbot-again} is a diagram of $\T$-maps.
  \item For formal morphisms, in the cases $\T \in \{\S,\B\}$, $\La$ is a strict $\T$-map and so it sends the formal permutation or braiding morphisms of $\T(\ob A',\phi)$ to the corresponding permutations or braidings in $A'$.
  \end{enumerate}
  This completes the description of $\La$.
\end{explanation}

\begin{example}\label{example:La-phiw}
  Suppose, as in \cref{defn:La-application}, that $\T \in \{\M,\S,\B\}$ and
  \[
    f\cn (A,\bcdot) \zzto (A',\bcdot)
  \]
  is a $\T$-map.
  Let $\phi = f_\ob$ denote the restriction of $f$ to objects.

  Let $\ang{w} = (w_1,w_2)$ as in \cref{example:De-phiw}, with
  \[
    w_1 = \ang{a^1} = (a^1_1,a^1_2,a^1_3)
    \andspace
    w_2 = \ang{a^2} = (a^2_1,a^2_2).
  \]
  Then $w_\bullet = \ang{a^\bullet} = (a^1_1,a^1_2,a^1_3,a^2_1,a^2_2)$ and
  \[
    \La\bang{[\phi]w} = f(a^1_1 \bcdot a^1_2 \bcdot a^1_3) \,\bcdot\, f(a^2_1 \bcdot a^2_2).
  \]
\end{example}

\section{Examples for symmetric and braided monoidal functors}\label{sec:other}

In this section, we suppose that $\T$ is one of the 2-monads, $\S$ or $\B$, for symmetric or braided monoidal structures, respectively, that are strictly associative and unital (\cref{notn:monoidal-variants}).
In this section we say ``monoidal'' to mean strict monoidal structure for categories $A$ and $A'$.
Note, however, that the discussion here does apply to the corresponding general monoidal structures, $\Sg$ or $\Bg$, by the Monoidal Strictification \cref{thm:strictification}.

The Pseudomorphism Coherence \Cref{thm:main1,thm:main-msb} apply in these cases, and this section provides several examples using the \emph{dissolution}, as in \cref{thm:main-msb}, to determine whether a formal diagram commutes.
In these examples, we suppose given
\[
  f=(f,f_2,f_0) \cn (A, +, 0) \zzto (A', \bcdot, 1)
\]
as follows.
\begin{enumerate}
\item $(A,+,0,\be)$ and $(A',\bcdot,1,\be)$ are $\T$-algebras, i.e., symmetric or braided monoidal categories with the indicated notation for monoidal products, units, and braidings.
\item $f$ is a $\T$-map (\cref{rmk:Tmaps}), and we use the zigzag arrow notation of \cref{rmk:zzto}.
  Thus, $f$ is a symmetric or braided strong monoidal functor.
\end{enumerate}
All of our applications concern functors that are either strong or strict monoidal.
We will say that $f$ is a ``monoidal functor'' to mean strong monoidal functor.

\begin{example}\label{example:mystery1}
  The following diagram in $A'$ appears as \cref{eq:mystery1-intro} in the introduction.
  It involves $f$, the braiding isomorphisms of $A$ and $A'$, and an object $a \in A$.
  The two composites around the diagram apply a cyclic permutation to the objects, but combine with the monoidal constraints of $f$ in different ways.
  \begin{equation}\label{eq:mystery1}
    \begin{tikzpicture}[x=25ex,y=8ex,vcenter,xscale=1.2]
      \draw[0cell] 
      (0,0) node (s1) {f(a) \bcdot f(a) \bcdot f(a)}
      (s1)+(1,0) node (s2) {f(a+a) \bcdot f(a)}
      (s2)++(1,0) node (s3) {f(a) \bcdot f(a+a)}
      (s3)++(0,-1) node (s4) {f(a+a+a)}
      (s1)++(0,-1) node (t2) {f(a+a+a)}
      (t2)++(1,0) node (t3) {f(a+a+a)}
      ;
      \draw[1cell] 
      (s1) edge node {f_2 \bcdot 1} (s2)
      (s2) edge node {\beta} (s3)
      (s3) edge node {f_2} (s4)
      (s1) edge['] node {f_2} (t2)
      (t2) edge node {f(1 + \beta)} (t3)
      (t3) edge node {f(\beta +1)} (s4)
      ;
    \end{tikzpicture}
  \end{equation} 
  One can use the naturality of $f_2$ and $\be$, together with various axioms for $f$ and $\be$, to show that this diagram commutes.

  Alternatively, \cref{eq:mystery1} is an $f$-formal diagram, in the sense of \cref{defn:formal-diagram}.
  The following diagram is a lift to $\T(\ob A',\phi)$, where $\phi = f_\ob$ denotes the restriction of $f$ to objects.
  \begin{equation}\label{eq:mystery1-lift}
    \begin{tikzpicture}[x=27ex,y=10ex,vcenter,xscale=1.2]
      \draw[0cell] 
      (0,0) node (s1) {\Bigl(\; [\phi](a) \scs [\phi](a) \scs [\phi](a) \;\Bigr)}
      (s1)+(1.1,0) node (s2) {\Bigl(\; [\phi](a,a) \scs [\phi](a) \;\Bigr)}
      (s1)++(2,0) node (s3) {\Bigl(\; [\phi](a) \scs [\phi](a,a) \;\Bigr)}
      (s3)++(0,-1) node (s4) {\Bigl(\; [\phi](a,a,a) \;\Bigr)}
      (s1)++(0,-1) node (t2) {\Bigl(\; [\phi](a,a,a) \;\Bigr)}
      (t2)++(1,0) node (t3) {\Bigl(\; [\phi](a,a,a) \;\Bigr)}
      ;
      \draw[1cell] 
      (s1) edge node {[\phi]\qsf \sss 1} (s2)
      (s2) edge node {\beta} (s3)
      (s3) edge node {[\phi]\qsf} (s4)
      (s1) edge['] node {[\phi]\qsf} (t2)
      (t2) edge node {[\phi](1 \scs \beta)} (t3)
      (t3) edge node {[\phi](\beta \scs 1)} (s4)
      ;
    \end{tikzpicture}
  \end{equation} 
  To verify that the above diagram is a lift of \cref{eq:mystery1}, one uses the descriptions of $\T(\ob A',\phi)$ and $\La$ in \cref{defn:TGpphi} and \cref{defn:La-application}, respectively.
  In particular, the terminology of \cref{defn:TGpphi} applies as follows.
  \begin{itemize}
  \item The objects are \emph{$\phi$-objects}; each entry $[\phi](a,\ldots,a)$ is a lift of a term $f(a + \cdots + a)$.
  \item The morphisms $[\phi]\qsf$ are \emph{$\phi$-adjoined isomorphisms} \cref{eq:phi-adjoined} and are lifts of the monoidal constraints for $f$.
  \item The morphisms $[\phi](1,\be)$ and $[\phi](\be,1)$ are \emph{$\phi$-free morphisms} \cref{eq:phi-free} and are lifts of the corresponding morphisms $f(1 + \be)$ and $f(\be + 1)$.
  \item The morphism $\be$ is a \emph{formal morphism} lifting the corresponding braiding in \cref{eq:mystery1}.
  \end{itemize}

  Using the description of $\De$ in \cref{defn:De-application}, the \emph{dissolution} of \cref{eq:mystery1-lift} is the following diagram in $\T(\ob A')$.
  Here, $\phi = f_\ob$ is applied separately to each object, and the $\phi$-adjoined morphisms $[\phi]\qsf$ are sent to identities.
  \begin{equation}\label{eq:mystery1-dissolve}
    \begin{tikzpicture}[x=28ex,y=10ex,vcenter,xscale=1.2]
      \draw[0cell] 
      (0,0) node (s1) {\Bigl(\; f(a) \scs f(a) \scs f(a) \;\Bigr)}
      (s1)+(.84,0) node (s2) {\Bigl(\; f(a) \scs f(a) \scs f(a) \;\Bigr)}
      (s1)++(2,0) node (s3) {\Bigl(\; f(a) \scs f(a) \scs f(a) \;\Bigr)}
      (s3)++(0,-1) node (s4) {\Bigl(\; f(a) \scs f(a) \scs f(a) \;\Bigr)}
      (s1)++(0,-1) node (t2) {\Bigl(\; f(a) \scs f(a) \scs f(a) \;\Bigr)}
      (t2)++(1,0) node (t3) {\Bigl(\; f(a) \scs f(a) \scs f(a) \;\Bigr)}
      ;
      \draw[1cell] 
      (s1) edge node {1} (s2)
      (s2) edge node {\beta_{(\,f(a),f(a)\,)\scs f(a)}} (s3)
      (s3) edge node {1} (s4)
      (s1) edge['] node {1} (t2)
      (t2) edge node {\bigl(\, 1 \scs \beta \,\bigr)} (t3)
      (t3) edge node {\bigl(\, \beta \scs 1 \,\bigr)} (s4)
      ;
    \end{tikzpicture}
  \end{equation} 
  The two composites around the above diagram have the same underlying braid of the object $f(a)$: in the left-bottom composite, the first two instances of $f(a)$ are braided past the third, one at a time, and in the top-right composite they are braided past in one step.

  Therefore, the diagram \cref{eq:mystery1-dissolve} commutes in either case $\T = \S$ or $\T = \B$ by the Symmetric or Braided Coherence \cref{thm:diagrcoh}~\cref{it:S} or~\cref{it:B}, respectively.
  Since $\De$ is an equivalence by \cref{thm:main1}, this implies that the lift \cref{eq:mystery1-lift} commutes in $\T(\ob A',\phi)$ and hence  diagram \cref{eq:mystery1} commutes in $A'$.
\end{example}

The key feature of \cref{example:mystery1}, and of our other examples below, is that the dissolution diagram \cref{eq:mystery1-dissolve} replaces each monoidal constraint of $f$ in \cref{eq:mystery1} with an identity.
Thus, it also replaces objects such as $f(a+a)$ with tuples $\bigl( f(a) \scs f(a) \bigr)$ in the free algebra $\T(\ob A')$.
The lift \cref{eq:mystery1-lift} is what ensures that this can be done coherently.

As noted in \cref{rmk:what-it-means}, the composites around the dissolution diagram \cref{eq:mystery1-dissolve} determine morphisms in $A'$ that are generally \emph{not} equal to the respective morphisms from the original diagram \cref{eq:mystery1}.
Indeed, the morphisms in $A'$ determined by the composites around \cref{eq:mystery1-dissolve} do not have the same codomain as the composites around \cref{eq:mystery1}.
The purpose of the diagrammatic coherence theorems in this work is to determine:
\begin{enumerate}
\item how to construct a lift and corresponding dissolution of an $f$-formal diagram, and
\item conditions under which $\De$ is an equivalence, so that commutativity of the dissolution diagram implies that of the original diagram.
\end{enumerate}

\begin{example}[Monoidal naturality of $f_2$]\label{example:mystery2}
  The monoidal constraint $f_2$ is a natural transformation with components
  \[
    \bigl( f_2 \bigr)_{a,b} \cn  f(a) \bcdot f(b) \to f(a + b)
    \forspace a,b \in A.
  \]
  As a natural transformation, its domain and codomain are the respective composites in the following diagram, in which the products $A \times A$, $A' \times A'$ are given the componentwise monoidal structures.
  \[
    \begin{tikzpicture}[x=15ex,y=8ex,xscale=1.2]
      \draw[0cell] 
      (0,0) node (a) {A \times A}
      (a)++(1,.5) node (b) {A' \times A'}
      (b)++(1,-.5) node (d) {A'}
      (a)++(1,-.5) node (c) {A}
      ;
      \draw[1cell] 
      (a) edge node {f \times f} (b)
      (b) edge node {\bcdot} (d)
      (a) edge['] node {+} (c)
      (c) edge['] node {f} (d)
      ;
      \draw[2cell] 
      node[between=b and c at .5, rotate=-90, 2label={below,f_2}] {\Rightarrow}
      ;
    \end{tikzpicture}
  \]
  The composites above are monoidal functors, with the monoidal constraints of $+$ and $\bcdot$ given by the following for $a,b,c,d \in A$ and $a',b',c',d' \in A'$:
  \[
    a + b + c + d \fto{1 + \be_{b,c} + 1} a + c + b + d
    \andspace
    a' \bcdot b' \bcdot c' \bcdot d' \fto{1 \bcdot \be_{b',c'} \bcdot 1} a' \bcdot c' \bcdot b' \bcdot d'.
  \]
  
  The following diagram in $A'$ is the monoidal naturality axiom at objects $(a,b), (c,d) \in A \times A$, to check whether the natural transformation $f_2$ is a monoidal transformation.
  \begin{equation}\label{eq:mystery2}
    \begin{tikzpicture}[x=30ex,y=8ex,vcenter,xscale=1.2]
      \draw[0cell] 
      (0,0) node (s1) {f(a)\bcdot f(b)\bcdot f(c)\bcdot f(d)}
      (s1)++(0,-1) node (s2) {f(a)\bcdot f(c)\bcdot f(b)\bcdot f(d)}
      (s2)++(0,-1) node (s3) {f(a+c)\bcdot f(b+d)}
      (s3)++(1,0) node (s4) {f(a+c+b+d)}
      (s1)++(1,0) node (t1) {f(a+b)\bcdot f(c+d)}
      (t1)++(0,-1) node (t2) {f(a+b+c+d)}
      ;
      \draw[1cell] 
      (s1) edge['] node {1 \bcdot \be \bcdot 1} (s2)
      (s2) edge['] node {f_2 \bcdot f_2} (s3)
      (s3) edge node {f_2} (s4)
      (s1) edge node {f_2 \bcdot f_2} (t1)
      (t1) edge node {f_2} (t2)
      (t2) edge node {f(1 + \be + 1)} (s4)
      ;
    \end{tikzpicture}
  \end{equation}
  Again using \cref{defn:TGpphi,defn:La-application} followed by \cref{defn:De-application}, one can identify \cref{eq:mystery2} as an $f$-formal diagram and determine the requisite lift followed by its dissolution diagram, shown below.
  \begin{equation}\label{eq:mystery2-dissolve}
    \begin{tikzpicture}[x=30ex,y=9ex,vcenter,xscale=1.2]
      \draw[0cell] 
      (0,0) node (s1) {\Bigl(\; f(a) \scs f(b) \scs f(c) \scs f(d) \;\Bigr)}
      (s1)++(0,-1) node (s2) {\Bigl(\; f(a) \scs f(c) \scs f(b) \scs f(d) \;\Bigr)}
      (s2)++(0,-1) node (s3) {\Bigl(\; f(a) \scs f(c) \scs f(b) \scs f(d) \;\Bigr)}
      (s3)++(1,0) node (s4) {\Bigl(\; f(a) \scs f(c) \scs f(b) \scs f(d) \;\Bigr)}
      (s1)++(1,0) node (t1) {\Bigl(\; f(a) \scs f(b) \scs f(c) \scs f(d) \;\Bigr)}
      (t1)++(0,-1) node (t2) {\Bigl(\; f(a) \scs f(b) \scs f(c) \scs f(d) \;\Bigr)}
      ;
      \draw[1cell] 
      (s1) edge['] node {(1 , \be , 1)} (s2)
      (s2) edge['] node {1} (s3)
      (s3) edge node {1} (s4)
      (s1) edge node {1} (t1)
      (t1) edge node {1} (t2)
      (t2) edge node {(1 , \be , 1)} (s4)
      ;
    \end{tikzpicture}
  \end{equation}
  The two composites around \cref{eq:mystery2-dissolve} have the same underlying braid, given by passing $f(b)$ past $f(c)$.

  Therefore, \cref{eq:mystery2-dissolve} commutes in either case $\T = \S$ or $\T = \B$ by the Symmetric or Braided Coherence \cref{thm:diagrcoh}~\cref{it:S} or~\cref{it:B}, respectively.
  Since $\De$ is an equivalence by \cref{thm:main1}, the commutativity of the dissolution diagram \cref{eq:mystery2-dissolve} in $\T(\ob A')$ implies that the original diagram \cref{eq:mystery2} commutes in $A'$.
\end{example}

If $f$ is a symmetric or braided monoidal functor such that the monoidal constraint $f_2$ has components with nontrivial underlying braids, then the use of dissolution diagrams to determine commutativity of formal diagrams for $f$ is a nontrivial simplification.
For such functors $(f,f_2,f_0)$, the underlying braids of \cref{eq:mystery1,eq:mystery2} may be different from---generally more complex than---those of \cref{eq:mystery1-dissolve,eq:mystery2-dissolve}, respectively.
The significance of our diagrammatic coherence, when $\De$ is an equivalence, is precisely this simplification, summarized in the following variant of \cref{slogan:msb}.
\begin{slogan}
  When $\De$ is an equivalence, commutativity of formal diagrams for $f$ reduces to checking commutativity of the simpler dissolution diagrams, in which the $\T$-algebra constraints of $f$ are replaced by identities.
\end{slogan}

\begin{example}[Monoidal naturality of $\be_{f,f}$]\label{example:mystery3}
  Let $f \bcdot f$ denote the composite monoidal functor shown below, where \emph{diag} is the diagonal functor:
  \[
    A \fto{\text{diag}} A \times A \fto{f \times f} A' \times A' \fto{\bcdot} A'.
  \]
  So, $\bigl( f \bcdot f \bigr)(a) = f(a) \bcdot f(a)$ for objects $a \in A$, and likewise for morphisms.
  The monoidal constraint of $\bcdot$ is $1 \bcdot \be \bcdot 1$, described in \cref{example:mystery2}.
  The diagonal functor is strict monoidal because the monoidal sum in $A \times A$ is given componentwise.

  The braiding isomorphism $\be$ of $A'$ induces a natural transformation
  \begin{equation}\label{eq:betaff}
    \be_{f,f}\cn f \bcdot f \to f \bcdot f
  \end{equation}
  with components $\be_{f(a),f(a)}$ for $a \in A$.
  The diagram below is the monoidal naturality diagram at a pair of objects $a,b \in A$ to check whether $\be_{f,f}$ is a monoidal transformation.
  The left and right vertical composites are the monoidal constraints of $f \bcdot f$.
  \begin{equation}\label{eq:mystery3}
    \begin{tikzpicture}[x=37ex,y=10ex,vcenter,xscale=1.2]
      \draw[0cell] 
      (0,0) node (a') {f(a) \bcdot f(a) \bcdot f(b) \bcdot f(b)}
      (a')++(0,-1) node (c') {f(a) \bcdot f(b) \bcdot f(a) \bcdot f(b)}
      (c')++(0,-1) node (e') {f(a + b) \bcdot f(a + b)}
      (a')++(1,0) node (b') {f(a) \bcdot f(a) \bcdot f(b) \bcdot f(b)}
      (b')++(0,-1) node (d') {f(a) \bcdot f(b) \bcdot f(a) \bcdot f(b)}
      (d')++(0,-1) node (f') {f(a + b) \bcdot f(a + b)}
      (a')++(0,.5) node (a) {\bigl(f \bcdot f\bigr)(a) \bcdot \bigl(f \bcdot f\bigr)(b)}
      (b')++(0,.5) node (b) {\bigl(f \bcdot f\bigr)(a) \bcdot \bigl(f \bcdot f\bigr)(b)}
      (e')++(0,-.5) node (e) {\bigl(f \bcdot f\bigr)(a + b)}
      (f')++(0,-.5) node (f) {\bigl(f \bcdot f\bigr)(a + b)}
      ;
      \draw[1cell] 
      (a') edge node {\be \bcdot \be} (b')
      (e') edge node {\be} (f')
      (a') edge['] node {1 \bcdot \be \bcdot 1} (c')
      (c') edge['] node {f_2 \,\bcdot\, f_2} (e')
      (b') edge node {1 \bcdot \be \bcdot 1} (d')
      (d') edge node {f_2 \,\bcdot\, f_2} (f')
      (a') edge[equal] node {} (a)
      (b') edge[equal] node {} (b)
      (e') edge[equal] node {} (e)
      (f') edge[equal] node {} (f)
      ;
    \end{tikzpicture}
  \end{equation}
  The above is an $f$-formal diagram, and the following is a dissolution diagram for it.
  There, the morphism along the lower edge is the block braiding of the first two terms past the second two.
  \begin{equation}\label{eq:mystery3-dissolve}
    \begin{tikzpicture}[x=37ex,y=10ex,vcenter,xscale=1.2]
      \draw[0cell] 
      (0,0) node (a') {\Bigl( f(a) \scs f(a) \scs f(b) \scs f(b) \Bigr)}
      (a')++(0,-1) node (c') {\Bigl( f(a) \scs f(b) \scs f(a) \scs f(b) \Bigr)}
      (c')++(0,-1) node (e') {\Bigl( f(a) \scs f(b) \scs f(a) \scs f(b) \Bigr)}
      (a')++(1,0) node (b') {\Bigl( f(a) \scs f(a) \scs f(b) \scs f(b) \Bigr)}
      (b')++(0,-1) node (d') {\Bigl( f(a) \scs f(b) \scs f(a) \scs f(b) \Bigr)}
      (d')++(0,-1) node (f') {\Bigl( f(a) \scs f(b) \scs f(a) \scs f(b) \Bigr)}
      ;
      \draw[1cell] 
      (a') edge node {\bigl(\be \scs \be\bigr)}
      node['] {\si_1 \si_3} (b')
      (e') edge['] node {\be}
      node['] {\si_2 \si_1 \si_3 \si_2} (f')
      (a') edge['] node {\bigl(1 \scs \be \scs 1\bigr)}
      node['] {\si_2} (c')
      (c') edge['] node {1} (e')
      (b') edge node {\bigl(1 \scs \be \scs 1\bigr)}
      node['] {\si_2} (d')
      (d') edge node {1} (f')
      ;
    \end{tikzpicture}
  \end{equation}
  In the above diagram, the inner labels on each morphism are the underlying braids, where $\si_i$ is the elementary braiding of strand $i$ under strand $i+1$. 
  The left-bottom and top-right composites around the boundary of \cref{eq:mystery3} are shown in the following braid diagram.
  In these diagrams, the right-to-left composition of elementary braids is ordered bottom-to-top.
  
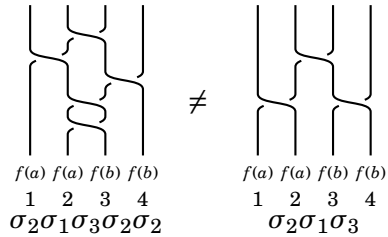
\begin{center}
    \begin{tikzpicture}[baseline=(E.base)]
      \pic[
      stdbraidstyle,
      braid/number of strands = 4,
      name prefix=braid0,
      ] at (0,0)
      {braid={s_2 s_1 s_3 s_2 s_2}}
      ;
      \foreach \x/\y/\z in {1/f(a)/2,2/f(b)/4,3/f(a)/1,4/f(b)/3}
      \draw
      (braid0-\x-e) ++(0,-.5ex) node[text height=.6cm,scale=.8] {$\scriptstyle \y$}
      ++(0,-3ex) node {$\scriptstyle \z$}
      ;
      \node (E) at (0,-1.4) {}; 
      \draw
      node[between=braid0-4-e and braid0-1-e at .5, shift={(0,-5.5ex)}]
      {$\si_2 \si_1 \si_3 \si_2 \si_2$};
    \end{tikzpicture}
    \ {\large $\ne$}\quad 
    \begin{tikzpicture}[baseline=(E.base)]
      \pic[
      stdbraidstyle,
      braid/number of strands = 4,
      name prefix=braid0,
      ] at (0,0)
      {braid={1 s_2 1 s_1-s_3 1}}
      ;
      \foreach \x/\y/\z in {1/f(a)/2,2/f(b)/4,3/f(a)/1,4/f(b)/3}
      \draw
      (braid0-\x-e) ++(0,-.5ex) node[text height=.6cm,scale=.8] {$\scriptstyle \y$}
      ++(0,-3ex) node {$\scriptstyle \z$}
      ;
      \node (E) at (0,-1.4) {}; 
      \draw
      node[between=braid0-4-e and braid0-1-e at .5, shift={(0,-5.5ex)}]
      {$\si_2\si_1\si_3$};
    \end{tikzpicture}
  \end{center}
  Since these braids are not equal, $\be_{f,f}$ in \cref{eq:betaff} is generally not a monoidal transformation when $\T = \B$ and $f$ is a braided monoidal functor.
  For example, when $A = A' = \B\{a,b\}$ is the free braided monoidal category on two objects and $f$ is the identity, then $\be_{f,f}$ is not a monoidal transformation.

  However, since the underlying \emph{permutations} around the diagram \cref{eq:mystery3} are equal, the Symmetric Coherence \cref{thm:diagrcoh}~\cref{it:S} implies that \cref{eq:mystery3} does commute when $\T = \S$.  Thus, $\be_{f,f}$ is a monoidal transformation when $f$ is a symmetric monoidal functor between symmetric monoidal categories.
  Note, again, that this conclusion holds independently of whether the monoidal constraints $f_2$ have nontrivial underlying permutations.
\end{example}
\begin{rmk}\label{rmk:mystery3}
  In each of the above examples, one can also check commutativity directly, using naturality of the monoidal constraints $f_2$.
  Diagram \cref{eq:mystery3} is particularly straightforward, involving a single use of naturality to commute $\be$ with $f_2 \bcdot f_2$.
  Formal diagrams for $f$ are always amenable to such an approach.
  However, it can be a nontrivial task to determine \emph{which} combination of naturality and other axioms will reduce the commutativity of the original diagram to a computation in $\T(\ob A')$.
  The advantage of the dissolution approach is that it formalizes such a reduction by systematically replacing the monoidal constraints with identities.
\end{rmk}

\section{Non-example via quadrupling}\label{sec:doub-quad}
In this section we continue the context and conventions of \cref{sec:other}, so that $\T \in \{\S,\B\}$ is the monad for (strict) symmetric or braided monoidal categories.
We consider two specific functors $f$ and $h$ whose monoidal constraints have nontrivial underlying braids.

This section gives several examples of using \cref{thm:diagrcoh-s}, which is a refinement of the Symmetric Coherence \cref{thm:diagrcoh}~\cref{it:S}.
Then, \cref{example:cursed-cyclic,rmk:cursed-trivialities,lem:cursed-cyclic} discuss a formal diagram \cref{eq:cursed-cyclic} where the only lifts of interest reduce, in the sense of \cref{rmk:f-formal-vs-X-formal}, to $A$-formal lifts.
For such lifts, the dissolution diagram strategy of \cref{sec:other} does not provide any simplification.

\begin{defn}[Doubling functor]\label{defn:dbl}
  The \emph{doubling functor} $f \cn A \zzto A$ with unit and monoidal constraints $f_0$ and $f_2$, respectively, is defined as follows for objects $a,b \in A$ and morphisms $s$ in $A$.
  \begin{equation}\label{eq:double}
    \begin{gathered}
      f(a) = a + a
      \andspace
      f(s) = s + s,\\
      \begin{tikzpicture}[x=13ex,y=5ex,vcenter]
        \draw[0cell] 
        (0,0) node (a) {0}
        (a)++(1,.5) node (b) {f(0)}
        (a)++(1,-.5) node (d) {0+0}
        ;
        \draw[1cell] 
        (b) edge[equal] node {} (d)
        (a) edge node {f_0} (b)
        (a) edge['] node {1_0} (d)
        ;
      \end{tikzpicture}
      \qquad
      \begin{tikzpicture}[x=30ex,y=5ex,vcenter]
        \draw[0cell] 
        (0,0) node (a) {f(a) + f(b)}
        (a)++(1,0) node (b) {f(a + b)}
        (a)++(0,-1) node (c) {a + a + b + b}
        (b)++(0,-1) node (d) {a + b + a + b}
        ;
        \draw[1cell] 
        (a) edge[equal] node {} (c)
        (b) edge[equal] node {} (d)
        (a) edge node {f_2} (b)
        (c) edge node {1_a + \beta_{a,b} + 1_b} (d)
        ;
      \end{tikzpicture}
    \end{gathered}
  \end{equation}
  \ 
\end{defn}
We will show below that $f$ is a \emph{symmetric} monoidal functor in the symmetric case, where $\T = \S$ and $A$ is a symmetric monoidal category.
In the braided case, where $\T = \B$ and $A$ is merely braided monoidal, then $f$ is a monoidal functor, but generally not \emph{braided} monoidal.
Although these conclusions will be familiar to experts, we include them as preparation for the calculations in \cref{example:cursed-cyclic,rmk:cursed-n}.

The unity diagrams for the doubling functor are trivial since $f_0 = 1_0$.
The following example discusses the associativity and braid axioms for $f$.
\begin{example}[Axioms for doubling]\label{example:dbl-monoidal}
  Let $f$ be the doubling functor from \cref{defn:dbl}.
  The associativity diagram, for objects $a,b,c \in A$, is the following.
  \begin{equation}\label{eq:mult2-assoc}
    \begin{tikzpicture}[x=30ex,y=9ex,vcenter,xscale=1.2,yscale=1.1]
      \draw[0cell] 
      (0,0) node (a) {f(a) + f(b) + f(c)}
      (a)++(1,0) node (b) {f(a) + f(b + c)}
      (a)++(0,-1) node (c) {f(a + b) + f(c)}
      (b)++(0,-1) node (d) {f(a + b + c)}
      (a)++(0,.5) node (a') {a + a + b + b + c + c}
      (b)++(0,.5) node (b') {a + a + b + c + b + c}
      (c)++(0,-.5) node (c') {a + b + a + b + c + c}
      (d)++(0,-.5) node (d') {a + b + c + a + b + c}
      ;
      \draw[1cell] 
      (a) edge node {1 + f_2} (b)
      (c) edge node {f_2} (d)
      (b) edge node {f_2} (d)
      (a) edge['] node {f_2 + 1} (c)
      (a) edge[equal] node {} (a')
      (b) edge[equal] node {} (b')
      (c) edge[equal] node {} (c')
      (d) edge[equal] node {} (d')
      (a') edge[',bend right=20,transform canvas={xshift=-10ex}] node {\si_2} (c')
      (b') edge[bend left=20,transform canvas={xshift=10ex}] node {\si_3\si_2} (d')
      (a') edge node {\si_4} (b')
      (c') edge node {\si_3\si_4} (d')
      ;
    \end{tikzpicture}
  \end{equation}
  In the above diagram, the inner arrows are labeled as structure morphisms of $f$ and the outer arrows are labeled by their underlying braids, where $\si_i$ are the elementary braids as in \cref{example:mystery3}.
  The underlying braids for the left-bottom and top-right composites around the boundary of \cref{eq:mult2-assoc} are shown in the following braid diagrams.
  \begin{equation}\label{eq:mult2-assoc-braid}
    \begin{tikzpicture}[baseline=(E.base)]
      \pic[
      stdbraidstyle,
      braid/number of strands = 6,
      name prefix=braid0,
      ] at (0,0)
      {braid={s_3 s_4 s_2}}
      ;
      \foreach \x/\y/\z in {1/a/1,2/b/3,3/c/5,4/a/2,5/b/4,6/c/6}
      \draw
      (braid0-\x-e) node[text height=.6cm] {$\scriptstyle \y$}
      ++(0,-3ex) node {$\scriptstyle \z$}
      ;
      \node (E) at (0,-1) {}; 
      \draw
      node[between=braid0-2-e and braid0-5-e at .5, shift={(0,-5.5ex)}]
      {$\si_3\si_4\si_2$ };
    \end{tikzpicture}
    \ \text{\large =}\ 
    \begin{tikzpicture}[baseline=(E.base)]
      \pic[
      stdbraidstyle,
      braid/number of strands = 6,
      name prefix=braid0,
      ] at (0,0)
      {braid={s_3 s_2 s_4}}
      ;
      \foreach \x/\y/\z in {1/a/1,2/b/3,3/c/5,4/a/2,5/b/4,6/c/6}
      \draw
      (braid0-\x-e) node[text height=.6cm] {$\scriptstyle \y$}
      ++(0,-3ex) node {$\scriptstyle \z$}
      ;
      \node (E) at (0,-1) {}; 
      \draw
      node[between=braid0-2-e and braid0-5-e at .5, shift={(0,-5.5ex)}]
      {$\si_3\si_2\si_4$ };
    \end{tikzpicture}
  \end{equation}
  Since these braids are equal, the Braided Coherence \cref{thm:diagrcoh}~\cref{it:B} implies that \cref{eq:mult2-assoc} commutes when $\T = \B$, and thus also when $\T = \S$.

  The symmetry axiom for the doubling functor, for objects $a,b \in A$, is the following.
  \begin{equation}\label{eq:mult2-symm}
    \begin{tikzpicture}[x=30ex,y=9ex,vcenter,xscale=1.2,yscale=1.1]
      \draw[0cell] 
      (0,0) node (a) {f(a) + f(b)}
      (a)++(1,0) node (b) {f(b) + f(a)}
      (a)++(0,-1) node (c) {f(a + b)}
      (b)++(0,-1) node (d) {f(b + a)}
      (a)++(0,.5) node (a') {a + a + b + b}
      (b)++(0,.5) node (b') {b + b + a + a}
      (c)++(0,-.5) node (c') {a + b + a + b}
      (d)++(0,-.5) node (d') {b + a + b + a}
      ;
      \draw[1cell] 
      (a) edge node {\beta_{f(a),f(b)}} (b)
      (b) edge node {f_2} (d)
      (a) edge['] node {f_2} (c)
      (c) edge node {f(\beta_{a,b})} (d)
      (a) edge[equal] node {} (a')
      (b) edge[equal] node {} (b')
      (c) edge[equal] node {} (c')
      (d) edge[equal] node {} (d')
      (a') edge[',bend right=20,transform canvas={xshift=-6ex}] node {\si_2} (c')
      (b') edge[bend left=20,transform canvas={xshift=6ex}] node {\si_2} (d')
      (a') edge node {\si_2\si_1\si_3\si_2} (b')
      (c') edge node {\si_3\si_1} (d')
      ;
    \end{tikzpicture}
  \end{equation}
  The above diagram is labeled similarly to \cref{eq:mult2-assoc}, with inner arrows labeled via structure morphisms and outer arrows labeled by their underlying braids.
  The following diagrams use the same conventions as above and show the underlying braids for the left-bottom and top-right composites around the boundary of \cref{eq:mult2-symm}.
  \begin{equation}\label{eq:mult2-symm-braid}
    \begin{tikzpicture}[baseline=(E.base)]
      \pic[
      stdbraidstyle,
      braid/number of strands = 4,
      name prefix=braid0,
      ] at (0,0)
      {braid={1 s_3 s_1 s_2 1}}
      ;
      \foreach \x/\y/\z in {1/b/3,2/a/1,3/b/4,4/a/2}
      \draw
      (braid0-\x-e) node[text height=.6cm] {$\scriptstyle \y$}
      ++(0,-3ex) node {$\scriptstyle \z$}
      ;
      \node (E) at (0,-1.2) {}; 
      \draw
      node[between=braid0-4-e and braid0-1-e at .5, shift={(0,-5.5ex)}]
      {$\si_3\si_1\si_2$ };
    \end{tikzpicture}
    \ \text{\large $\ne$}\ 
    \begin{tikzpicture}[baseline=(E.base)]
      \pic[
      stdbraidstyle,
      braid/number of strands = 4,
      name prefix=braid0,
      ] at (0,0)
      {braid={s_2 s_2 s_1 s_3 s_2}}
      ;
      \foreach \x/\y/\z in {1/b/3,2/a/1,3/b/4,4/a/2}
      \draw
      (braid0-\x-e) node[text height=.6cm] {$\scriptstyle \y$}
      ++(0,-3ex) node {$\scriptstyle \z$}
      ;
      \node (E) at (0,-1.2) {}; 
      \draw
      node[between=braid0-4-e and braid0-1-e at .5, shift={(0,-5.5ex)}]
      {$\si_2\si_2\si_1\si_3\si_2$ };
    \end{tikzpicture}
  \end{equation}
  Since these braids are not equal, \cref{eq:mult2-symm} does not generally commute when $\T = \B$.
  That is, for a general braided monoidal category $A$, the doubling functor is not necessarily a \emph{braided} monoidal functor, although it is a plain monoidal functor.

  Note, however, that the underlying permutations of the braids above \emph{are} equal.
  Thus, the Symmetric Coherence \cref{thm:diagrcoh}~\cref{it:S} implies that \cref{eq:mult2-symm} does commute when $\T = \S$.
  That is, the doubling functor is a symmetric monoidal functor when $A$ is a symmetric monoidal category.
\end{example}

\begin{rmk}\label{rmk:symmetric-coherence-example}
  In the case $\T = \S$, there is a refinement of the Symmetric Coherence \cref{thm:diagrcoh}~\cref{it:S}, discussed in \cref{sec:symm-coh}.
  For finitely-generated formal diagrams in a symmetric monoidal category $A$, such as those of \cref{example:dbl-monoidal}, \cref{thm:diagrcoh-s} shows that it suffices to check the self-permutation of $x$, in the sense of \cref{defn:aperm}, for each generating object $x$.

  In \cref{eq:mult2-assoc}, it suffices to check the three self-permutations
  $\pidt_a$, $\pidt_b$, and $\pidt_c$, where $\wt{D}$ denotes the formal lift of \cref{eq:mult2-assoc} to the free monoidal category on three objects, $\S\{a,b,c\}$.
  Each self-permutation $\pidt_x$ is determined by projecting to the free symmetric monoidal category on the single object $x$, for $x \in \{a,b,c\}$.
  In the braid diagram \cref{eq:mult2-assoc-braid}, this corresponds to removing the strands for each object $y \ne x$, and then checking the underlying permutation of the resulting braid.

  In both the left-bottom and top-right composites around \cref{eq:mult2-assoc}, neither instance of $a$ is permuted past the other.
  That is, the strands labeled $a$ in \cref{eq:mult2-assoc-braid} do not cross.
  Thus, $\pidt_a = 1$ for each composite around \cref{eq:mult2-assoc}.
  Likewise, the self-permutations of $b$ and $c$ are identities for both composites.
  This is sufficient for \cref{thm:diagrcoh-s} to imply that \cref{eq:mult2-assoc} commutes.

  The same approach can be used for the composites around \cref{eq:mult2-symm}: the self-permutations of both $a$ and $b$ are trivial, for both composites around \cref{eq:mult2-symm}.
  This is sufficient for \cref{thm:diagrcoh-s} to imply that \cref{eq:mult2-symm} commutes.
\end{rmk}

Recall from \cref{rmk:f-formal-vs-X-formal}, for a symmetric or braided monoidal functor $f\cn A \to A'$, a lift $\wt{D}$ of an $f$-formal diagram is said to \emph{reduce to an $A'$-formal diagram} if it factors through the inclusion of free objects and morphisms 
\begin{equation}\label{eq:ka-again}
  \ka\cn \T(\ob A') \to \T(\ob A',\phi)
\end{equation}
where $\phi = f_\ob$ denotes the restriction to objects.
None of \cref{example:mystery1,example:mystery2,example:mystery3} factor through $\ka$, because the respective lifts involve the $\phi$-adjoined isomorphisms $[\phi]\qsf$, which are then mapped via $\De$ to identities.

The following provides an example of a monoidal naturality diagram that involves only braid isomorphisms and monoidal constraints and yet, except for certain trivialities, any lift to generating morphisms of $\T(\ob A, \phi)$ must factor through $\ka$ and hence reduce to an $A$-formal lift.
\Cref{rmk:cursed-trivialities,lem:cursed-cyclic} below explain some details and additional subtleties related to this case.
\begin{nonexample}[Cyclic braiding]\label{example:cursed-cyclic}
  Let $h$ denote the \emph{quadrupling} functor $h = f \circ f$, where $f$ is the doubling functor from \cref{defn:dbl,example:dbl-monoidal}.
  Thus, we have
  \[
    h(a) = a + a + a + a \forspace a \in A,
  \]
  and $h$ is a monoidal functor in either the symmetric or braided monoidal cases $\T \in \{\S,\B\}$.
  In the symmetric case, $\T = \S$, quadrupling is a symmetric monoidal functor.
  In other words (\cref{rmk:Tmaps}), $h$ is an $\S$-map, but generally not a $\B$-map.

  There is a natural transformation $\ga$ with components given by the cyclic braiding of the first summand past the other three:
  \begin{equation}\label{eq:ga-a}
    \ga_a = \beta_{a,(a+a+a)} \cn h(a) = a + a + a + a
    \to a + a + a + a = h(a).
  \end{equation}
  The following is the monoidal naturality diagram for $a,b\in A$, to determine whether $\ga$ is a monoidal transformation.
  Here, we use the notation
  \[
    \si_{i:k} = \si_{k}\si_{k-1}\cdots\si_i
  \]
  to denote the braiding of strand $i$ under strands $i+1$ through $k+1$.
  \begin{equation}\label{eq:cursed-cyclic}
    \begin{tikzpicture}[x=38ex,y=14ex,vcenter,xscale=1.2,yscale=1.1]
      \draw[0cell] 
      (0,0) node (a) {
        a+a+a+a + b+b+b+b
      }
      (a)++(0,-.8) node (c) {
        a+a + b+b + a+a + b+b
      }
      (c)++(0,-1) node (e) {
        a+b + a+b + a+b + a+b 
      }
      (a)++(1,0) node (b) {
        a+a+a+a + b+b+b+b
      }
      (b)++(0,-.8) node (d) {
        a+a + b+b + a+a + b+b
      }
      (d)++(0,-1) node (f) {
        a+b + a+b + a+b + a+b 
      }
      (a)++(0,.3) node (a') {h(a) + h(b)}
      (b)++(0,.3) node (b') {h(a) + h(b)}
      (e)++(0,-.3) node (e') {h(a+b)}
      (f)++(0,-.3) node (f') {h(a+b)}
      ;
      \draw[1cell] 
      (a) edge node {\ga_a + \ga_b}
      node['] {\si_{1:3}\,\si_{5:7}} (b)
      (e) edge['] node {\ga_{a+b}}
      node['] {\si_{1:6} \, \si_{2:7}} (f)
      (a) edge['] node {1 + \be_{a+a,b+b} + 1}
      node['] {\si_{3:4}\,\si_{4:5}} (c)
      (c) edge['] node {\genatop{1 + \be_{a,b} + 1}{ + 1 + \be_{a,b} + 1}}
      node['] {\si_{2}\si_{6}} (e)
      (b) edge node {1 + \be_{a+a,b+b} + 1}
      node['] {\si_{3:4}\,\si_{4:5}} (d)
      (d) edge node {\genatop{1 + \be_{a,b} + 1}{ + 1 + \be_{a,b} + 1}}
      node['] {\si_{2}\si_{6}} (f)
      (a) edge[equal] node {} (a')
      (b) edge[equal] node {} (b')
      (e) edge[equal] node {} (e')
      (f) edge[equal] node {} (f')
      ;
    \end{tikzpicture}
  \end{equation}
  The above is an $A$-formal diagram, in the sense of \cref{defn:formal-diagram}: it admits a lift to $\T(\ob A)$, with underlying braids shown on the inner labels.

  Composing with the inclusion of free objects and morphisms $\ka$ \cref{eq:ka-again}, with $\phi = h_\ob$, trivially yields a diagram in $\T(\ob A, \phi)$.
  However, as discussed in \cref{rmk:f-formal-vs-X-formal}, the resulting dissolution diagram is equal to the original lift and does not yield any simplification.

  The vertical morphisms in \cref{eq:cursed-cyclic} are the monoidal constraints for $h = f \circ f$, and one would obtain a simpler dissolution diagram if these were lifted to $\phi$-adjoined isomorphisms $[\phi]\qsf$.
  However, \cref{lem:cursed-cyclic} below shows that such morphisms are not generally composable with lifts of $\ga$.

  Here, we use the Braided and Symmetric Coherence \cref{thm:diagrcoh,thm:diagrcoh-s} directly on the $A$-formal lift of \cref{eq:cursed-cyclic}.
  The underlying braids of the left-bottom and top-right composites are shown below.
  These braids are distinct; strands 2 and 5 are linked on the left, but not on the right.
  \tikzset{lastbraidstyle/.style = {
      braid/height=-2mm,
      braid/width=5mm,
      braid/gap=.25, 
      braid/control factor=0.3, 
      braid/nudge factor=0.01, 
      braid/number of strands = 8,
      braid/every strand/.style={very thick},
      braid/strand 1/.style={brcola},
      braid/strand 3/.style={brcola},
      braid/strand 5/.style={brcola},
      braid/strand 7/.style={brcola},
      braid/strand 2/.style={brcolb},
      braid/strand 4/.style={brcolb},
      braid/strand 6/.style={brcolb},
      braid/strand 8/.style={brcolb},
    }
  }
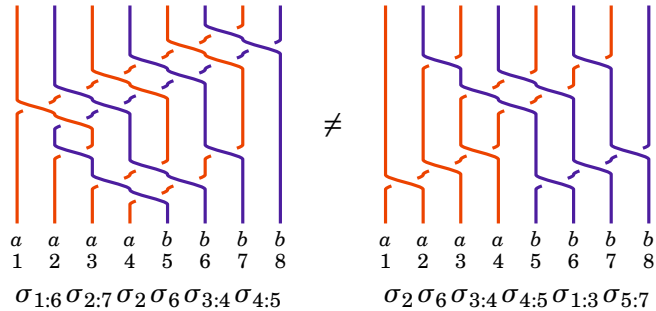
  \[
    \begin{tikzpicture}[baseline=(E.base)]
      \pic[
      stdbraidstyle,
      lastbraidstyle,
      name prefix=braid0,
      ] at (0,0)
      {braid={s_6 s_5-s_7 s_4-s_6 s_3-s_5 s_2-s_4 s_1-s_3 s_2 
          1 s_2-s_6 s_4 s_3-s_5 s_4}}
      ;
      \foreach \x/\y/\z in {1/a/2,2/b/6,3/a/3,4/b/7,5/a/4,6/b/8,7/a/1,8/b/5}
      \draw
      (braid0-\x-e) node[text height=.6cm] {$\scriptstyle \y$}
      ++(0,-3ex) node {$\scriptstyle \z$}
      ;
      \node (E) at (0,-1.7) {}; 
      \draw
      node[between=braid0-7-e and braid0-6-e at .5, shift={(0,-6ex)}]
      {$\si_{1:6}\,\si_{2:7}\,\si_2\,\si_6\,\si_{3:4}\,\si_{4:5}$};
    \end{tikzpicture}
    \quad\text{\large $\ne$}\quad
    \begin{tikzpicture}[baseline=(E.base)]
      \pic[
      stdbraidstyle,
      lastbraidstyle,
      name prefix=braid0,
      ] at (0,0)
      {braid={1 1 s_2-s_6 s_4 s_3-s_5 s_4
          1 1 s_3-s_7 s_2-s_6 s_1-s_5 1}}
      ;
      \foreach \x/\y/\z in {1/a/2,2/b/6,3/a/3,4/b/7,5/a/4,6/b/8,7/a/1,8/b/5}
      \draw
      (braid0-\x-e) node[text height=.6cm] {$\scriptstyle \y$}
      ++(0,-3ex) node {$\scriptstyle \z$}
      ;
      \node (E) at (0,-1.7) {}; 
      \draw
      node[between=braid0-4-e and braid0-1-e at .5, shift={(0,-6ex)}]
      {$\si_2\,\si_6\,\si_{3:4}\,\si_{4:5}\,\si_{1:3}\,\si_{5:7}$};
    \end{tikzpicture}
  \]
  Therefore, the cyclic braiding $\ga$ is generally \emph{not} a monoidal transformation in the braided case $\T = \B$.
  For example, if $A = \B\{a,b\}$ is the free braided monoidal category on two objects, then $\ga$ will not be a monoidal transformation.

  In the symmetric case $\T = \S$, checking the underlying permutations in \cref{eq:cursed-cyclic} simplifies via \cref{thm:diagrcoh-s}.
  The vertical composites have trivial $a$-permutation and hence the self-permutation of $a$ under either the left-bottom or top-right composite is the cyclic permutation $(1\ 4\ 3\ 2)$.
  The same statements apply to $b$.
  This is sufficient for \cref{thm:diagrcoh-s} to imply that \cref{eq:cursed-cyclic} commutes.
  Thus, the natural transformation $\ga$ in \cref{eq:ga-a} is monoidal natural in the symmetric case $\T = \S$, but generally not in the braided case, $\T = \B$.
\end{nonexample}

\begin{rmk}\label{rmk:cursed-n}
  In the symmetric case $\T = \S$, \cref{example:cursed-cyclic} can be generalized to show that any permutation $\ga \in \Si_n$ determines a monoidal natural automorphism of the \emph{$n$-fold sum} functor
  \[
    h^n(a) = \underbrace{a + \cdots + a}_{n\text{ summands}}
    \forspace a \in A.
  \]
  Here, $h^n$ is a symmetric monoidal functor defined inductively as $h^n = (h^2 + 1) \circ h^{n-1}$, with $h^2 = f$ being the symmetric monoidal doubling functor.
  For objects $a,b \in A$, both the $a$-permutation and the $b$-permutation of the monoidal constraint
  \[
    h^n_2 \cn h^n(a) + h^n(b) \to h^n(a+b)
  \]
  are identities.
  Letting $\ga\cn h^n \to h^n$ also denote the natural transformation induced by $\ga \in \Si_n$, both morphisms $\ga_{a} + \ga_{b}$ and $\ga_{a + b}$ have $a$-permutation equal to $\ga \in \Si_n$, and likewise for $b$-permutations.
  Thus, \cref{thm:diagrcoh-s} shows that the monoidal naturality diagram for $\ga$ commutes for each $a,b \in A$.
\end{rmk}

In the following discussion, we restrict to the symmetric case, $\T = \S$, because the quadrupling functor is an $\S$-map, but not a $\B$-map.
Thus, \cref{defn:La} applies to $h$ in the case $\T = \S$, but not in the case $\T = \B$.
The details of this discussion will require the following observation and subsequent terminology to exclude certain lifts of the monoidal unit $0 \in A$ and its identity morphism.
\begin{rmk}\label{rmk:cursed-trivialities}
  In the context of \cref{example:cursed-cyclic}, there are several objects and morphisms of $\S(\ob A, \phi)$ that are nontrivial lifts of the monoidal unit $0$ and its identity morphism.
  In particular, there are $\phi$-adjoined isomorphisms that lift the unit constraint $h_0 = 1_0$; the monoidal constraints at $0$,
  \[
    (h_2)_{0,0} = 1_0 \cn h(0) + h(0) \to h(0 + 0) = 0;
  \]
  or other such combinations of unit and monoidal constraints of $h$ at $0$.

  More generally, $\phi$-objects of the form $\bigl( [\phi](0,\ldots,0) \bigr)$, or $;$ products of such objects, will be lifts of $0$.
  Morphisms between such objects will be lifts of $1_0$, and therefore do not make substantial contributions to lifts of interest for, e.g., the composites around \cref{eq:cursed-cyclic}.
\end{rmk}

\begin{defn}\label{defn:xfund}
  Suppose $A$ is a symmetric strict monoidal category with unit $0$, and $\phi\cn \ob A \to \ob A$ is a map of sets.
  An object $x \in \S(\ob A, \phi)$ is called \emph{tidy} if it has no $;$ factors of the form $\bigl( [\phi](0,\ldots,0) \bigr)$.
  A composite of morphisms in $\S(\ob A, \phi)$
  \begin{equation}\label{eq:xifund}
    x_0 \fto{\xi_1} x_1 \fto{\xi_2} \cdots \fto{\xi_r} x_r
    \forspace r \ge 1
  \end{equation}
  is called \emph{tidy} if each $x_i$ is a tidy object and each morphism $\xi_i$ is a $;$ product of generating morphisms.
\end{defn}
\begin{lem}\label{lem:cursed-cyclic}
  Suppose $A = \S\{a,b\}$ is the free symmetric monoidal category on two objects $a$ and $b$.
  In the context of \cref{example:cursed-cyclic}, any tidy lift of \cref{eq:cursed-cyclic} to $\S(\ob A,\phi)$, with $\phi = h_\ob$, reduces to an $A$-formal lift.
\end{lem}
\begin{proof}
  We will use the description of
  \begin{equation}\label{eq:La-cursed}
    \La \cn \S(\ob A, \phi) \to A
  \end{equation}
  as shown in the following diagram, which is \cref{eq:La-pushout} applied to this case and explained further below.
  \begin{equation}\label{eq:La-pushout-cursed}
    \begin{tikzpicture}[x=16ex,y=10ex,vcenter,xscale=1.2]
      \draw[0cell] 
      (0,0) node (tc) {\S (\ob A)}
      (tc)++(1,0) node (tc') {\S (\ob A)}
      (tc)++(0,-1) node (iqtc) {\Q\S (\ob A)}
      (tc')++(0,-1) node (tc'phi) {\S(\ob A,\phi)}
      (tc'phi)++(1,0) node (ty) {\S A}
      (ty)++(0,-1) node (y) {A}
      (iqtc)++(.5,-.5) node (iqtx) {\Q \S A}
      (iqtx)++(.5,-.5) node (iqx) {\Q A}
      ;
      \draw[1cell] 
      (tc) edge node {\S\phi} (tc')
      (tc) edge['] node {\ze^\flat} (iqtc)
      (tc') edge['] node {\ka} (tc'phi)
      (iqtc) edge node {\wh{\phi}} (tc'phi)
      (tc') edge node {} (ty)
      (ty) edge node {+} (y)
      (iqtc) edge['] node {} (iqtx)
      (iqtx) edge['] node {\Q +} (iqx)
      (iqx) edge['] node {h^{\bot}} (y)
      (tc'phi) edge[dashed] node {\La} node['] {\exists !}(y)
      ;
    \end{tikzpicture}
  \end{equation}
  Here, the upper left square is a pushout of symmetric monoidal categories and symmetric strict monoidal functors.
  The dashed arrow $\La$ is the unique symmetric strict monoidal functor induced by the outer composites.

  Recalling \cref{defn:TGpphi}, the generating morphisms of $\S(\ob A, \phi)$ consist of free morphisms, $\phi$-morphisms, and formal morphisms, described as follows and explained further below.
  \begin{itemize}
  \item The \emph{free objects} and \emph{free morphisms} of $\S(\ob A, \phi)$ are those in the image of $\ka$.
  \item The \emph{$\phi$-objects} and \emph{$\phi$-morphisms} of $\S(\ob A, \phi)$ are those in the image of $\wh{\phi}$.
  \item The \emph{formal morphisms} of $\S(\ob A, \phi)$ are symmetry isomorphisms for the product $;$ induced by concatenation of tuples (\cref{notn:TA}).
  \end{itemize}
  Since $\S(\ob A, \phi)$ is obtained as a pushout, free objects and morphisms that are in the image of $\S \phi$ are identified with the corresponding $\phi$-objects and $\phi$-morphisms in the image of $\ze^\flat$.
  Furthermore, the symmetry isomorphisms in $\S(\ob A)$ and $\Q\S(\ob A)$ are identified with the corresponding formal morphisms of $\S(\ob A, \phi)$.
  In particular, formal morphisms between free objects are identified with the corresponding permutation isomorphisms of $\S(\ob A)$.

  Below, we will show that every lift of \cref{eq:cursed-cyclic} to a tidy composite in $\S(\ob A, \phi)$ factors through $\ka$.
  The argument uses the following two invariants that are associated to any map of sets $\phi \cn \ob A \to \ob A$.
  The hypothesis that $\phi$ is given by quadrupling will be used below, when we apply these invariants to the case of interest.
  \begin{itemize}
  \item\label{it:invar1} Each morphism in $A = \S\{a,b\}$ has an \emph{underlying $a$-permutation} and an \emph{underlying $b$-permutation}, described in \cref{defn:underlying}.
    Therefore, each morphism $\xi$ of $\S(\ob A, \phi)$ has underlying $a$- and $b$-permutations given by those of $\La\xi$.
  \item\label{it:invar2} Each object of $\S(\ob A, \phi)$ has an \emph{$a$-signature} and a \emph{$b$-signature} that are elements of $\S(\NN)$, explained below.
  \end{itemize}

  For objects of $A$, let $\nu^a$ denote the composite
  \[
    \ob A = \ob\bigl(\S(\{a,b\})\bigr) \to \S(\{a\}) \to \NN
  \]
  given first by sending $b$ to 0 and then taking isomorphism classes of objects.
  Let $\nu^b$ denote the similar composite that first sends $a$ to 0 and then takes isomorphism classes of objects.
  Each $\nu \in \{\nu^a,\nu^b\}$ induces a free functor
  \[
    \S(\ob A) \fto{\S\nu} \S\NN
  \]
  that is given by applying $\nu$ entry-wise to tuples of objects of $A$.
  The free functor $\S\nu$ is symmetric strict monoidal with respect to the concatenation of tuples, denoted $;$ as in \cref{notn:TA}.
  Define the \emph{$a$-signature} and \emph{$b$-signature} of an object $\ang{c} = \ang{c_i}_{i=1}^n$ in $\S(\ob A)$ as the tuples of natural numbers
  \begin{equation}\label{eq:signature}
    \begin{split}
      \sgn^a\ang{c} = (\S\nu^a)\ang{c}
      & = \bang{\nu^a(c_i)}_{i=1}^n \andspace \\
      \sgn^b\ang{c} = (\S\nu^b)\ang{c}
      & = \bang{\nu^b(c_i)}_{i=1}^n
    \end{split}
  \end{equation}
  for $c_i$ in $\ob A$.

  To define the $a$- and $b$-signatures of general objects $x \in \S(\ob A, \phi)$, recall from \cref{defn:TGpphi} that the upper square of \cref{eq:La-pushout-cursed} remains a pushout after taking the underlying monoid of objects.
  That is, applying the functor
  \[
    \ob\cn \Salgs \to \Mon
  \]
  as in \cref{eq:disc-ob-ind} preserves pushouts because it is left adjoint to $\indisc$.

  Now recall from \cref{defn:QA} that the objects of $\Q A$ are given by those of the free algebra $\S A$.
  Therefore, the monoid homomorphism $\S \phi$ in the following diagram of monoids induces the dashed arrow $\La'$, factoring
  \[
    \ob \La \cn \ob \bigl( \S(\ob A, \phi) \bigr) \to \ob \bigl( A \bigr)
  \]
  through $\ob\bigl(\S A\bigr)$.
  Here, the two unlabeled arrows are induced by inclusion of objects $\ob A \hookrightarrow A$.
  \begin{equation}\label{eq:La'}
    \begin{tikzpicture}[x=19ex,y=12ex,vcenter,xscale=1.2]
      \draw[0cell] 
      (0,0) node (tc) {\ob\bigl(\S (\ob A)\bigr)}
      (tc)++(1,0) node (tc') {\ob\bigl(\S (\ob A)\bigr)}
      (tc)++(0,-1) node (iqtc) {\ob\bigl(\Q\S (\ob A)\bigr)}
      (tc')++(0,-1) node (tc'phi) {\ob\bigl(\S(\ob A,\phi)\bigr)}
      (tc'phi)++(1,0) node (ty) {\ob\bigl(\S A\bigr)}
      (ty)++(0,-1) node (y) {\ob\bigl(A\bigr)}
      (ty)++(.7,0) node (nn) {\ob\bigl(\S\NN\bigr)}
      (iqtc)++(.5,-.5) node (iqtx) {\ob\bigl(\Q \S A\bigr)}
      (iqtx)++(.5,-.5) node (iqx) {\ob\bigl(\Q A\bigr)}
      (iqx)++(.4,.4) node (itx) {\ob\bigl(\S A\bigr)}
      ;
      \draw[1cell] 
      (tc) edge node {\S\phi} (tc')
      (tc) edge['] node {\ze^\flat} (iqtc)
      (tc') edge['] node {\ka} (tc'phi)
      (iqtc) edge node {\wh{\phi}} (tc'phi)
      (tc') edge node {} (ty)
      (ty) edge node {+} (y)
      (ty) edge[transform canvas={shift={(0,.5ex)}}] node {\S\nu^a} (nn)
      (ty) edge[',transform canvas={shift={(0,-.5ex)}}] node {\S\nu^b} (nn)
      (iqtc) edge['] node {} (iqtx)
      (iqtx) edge['] node {\Q +} (iqx)
      (iqx) edge[equal] node {} (itx)
      (itx) edge['] node {\S \phi} (ty)
      (iqx) edge['] node {h^{\bot}} (y)
      (tc'phi) edge[dashed] node[pos=.4] {\La'} node[',pos=.4] {\exists !} (ty)
      ;
    \end{tikzpicture}
  \end{equation}
  Define the \emph{$a$-signature} and \emph{$b$-signature} of a general object $x \in \S(\ob A, \phi)$ via the corresponding signature of $\La'(x)$, as follows:
  \begin{equation}\label{eq:signature2}
    \sgn^a(x) = (\S\nu^a)\La'(x)
    \andspace
    \sgn^b(x) = (\S\nu^b)\La'(x).
  \end{equation}
  This agrees with the previous definitions \cref{eq:signature} for free objects $x \in \S(\ob A)$ since commutativity of \cref{eq:La'} requires that $\La' \ka$ is the identity on objects.
  Note that these signatures are invariants of objects only; they do not extend to all morphisms of $\S(\ob A, \phi)$.
  This completes the definitions of $a$- and $b$- signature.

  Now we apply the underlying permutation and signature invariants to complete the proof.
  The following observations, for objects $x$ and $y$ in $\S(\ob A, \phi)$,
  make use of the hypothesis $\phi = h_\ob$ and details of the diagram \cref{eq:cursed-cyclic}.
  \begin{enumerate}
    \renewcommand{\theenumi}{\arabic{enumi}}
    \renewcommand{\labelenumi}{(\theenumi)}
  \item\label{it:obs3} If $x$ is a lift of an object in \cref{eq:cursed-cyclic}, or isomorphic to such a lift, then the sum of the entries in $\sgn^a(x)$, respectively the sum of the entries of $\sgn^b(x)$, is equal to four.
  \item\label{it:obs2} If $x$ is a $\phi$-object, then each entry of  $\sgn^a(x)$, respectively $\sgn^b(x)$, is divisible by four.
    This follows from \cref{defn:La-application} because $h$ is the quadrupling functor: $\La'$ sends each $\phi$-object to an object of $\S A$ for which each entry is $h(a^j_\bullet)$ for a certain object $a^j_\bullet \in A$.
  \item\label{it:obs8} If $x$ is a $\phi$-object such that each entry of $\sgn^a(x)$ and each entry of $\sgn^b(x)$ is zero, then $x$ is a $;$ product of objects of the form $\bigl( [\phi](0,\ldots,0) \bigr)$.
    This follows from the same explanation of $\La'$ as above, because every nonzero object of $\S A$ has nonzero $a$- or $b$-signature.
  \item\label{it:obs5} The underlying $a$-permutation of each composite around \cref{eq:cursed-cyclic} is $(1\ 4\ 3\ 2)$, which is an odd permutation.
    The same holds for the underlying $b$-permutations around \cref{eq:cursed-cyclic}.
  \item\label{it:obs4} If $\xi \cn x \to y$ is a free or formal morphism of $\S(\ob A, \phi)$, then $\sgn^a(x)$ and $\sgn^a(y)$ have the same set of entries, possibly in a permuted order.
    A similar observation holds for $\sgn^b(x)$ and $\sgn^b(y)$.
  \item\label{it:obs6} If $\xi \cn x \to y$ is a $\phi$-morphism in $\S(\ob A, \phi)$, then \cref{defn:La-application} shows that $\La\xi$ is given either by applying $h$ to certain permutations, or by the monoidal constraints of $h$.
    Since $h$ is given by quadrupling, and the underlying $a$-permutation of the monoidal constraint $h_2$ is trivial, the underlying $a$-permutation of $\La\xi$ is even in either case.
    Likewise, the underlying $b$-permutation of $\La\xi$ is also even.
  \item\label{it:obs7}
    If all the entries of $\sgn^a(x)$ are even, and $\xi\cn x \to y$ is a $;$ product of generating morphisms of $\S(\ob A, \phi)$, then the underlying $a$-permutation of $\xi$ is even.
    The $;$ factors of $\xi$ that are free or formal morphisms have underlying $a$-permutations that are even because they are given by permuting entries of tuples or factors of the $;$ product.
    The $;$ factors of $\xi$ that are $\phi$-morphisms have underlying $a$-permutations that are even by observation \cref{it:obs6}.
    A similar observation holds if all the entries of $\sgn^b(x)$ are even: then the underlying $b$-permutation of $\xi$ is even.
  \end{enumerate}

  Now suppose that 
  \begin{equation}\label{eq:lift-xi}
    x_0 \fto{\xi_1} x_1 \fto{\xi_2} \cdots \fto{\xi_r} x_r
  \end{equation}
  is a tidy composite in $\S(\ob A, \phi)$ lifting either of the composites around \cref{eq:cursed-cyclic}.
  Recalling \cref{defn:xfund}, the assumption that \cref{eq:lift-xi} is tidy means that each $\xi_i$ is a product of generating morphisms and none of the $x_i$ have $;$ factors of the form $\bigl( [\phi](0,\ldots,0) \bigr)$.
  The observations above lead to the following conclusions.
  \begin{enumerate}
  \item The $a$-signature $\sgn^a(x_0)$ must have at least one odd entry.
    If not---if all the entries of $\sgn^a(x_0)$ are even---then observations \cref{it:obs2,it:obs4,,it:obs7,it:obs6} imply that the underlying $a$-permutation of $\xi_1$ is even and that all the entries of $\sgn^a(x_1)$ are even.
    Repeating this reasoning, the underlying $a$-permutation of each $\xi_i$ is even, but this contradicts observation \cref{it:obs5}.
    Likewise, $\sgn^b(x_0)$ must have at least one odd entry.
  \item Observations \cref{it:obs3,it:obs2,,it:obs8}, combined with the previous conclusion, imply that any $\phi$-object factors of $x_0$ would have $a$- and $b$-signatures whose entries are all zeros.
    By \cref{it:obs8}, this would contradict the assumption that $x_0$ is a tidy object.
  \item Therefore, $x_0$ must be a free object such that each of
    $\sgn^a(x_0)$ and $\sgn^b(x_0)$ consists of entries that are all less than four and not all even.
  \item The previous conclusion implies that $\xi_1$ must be a free morphism, since a product of free objects or morphisms is free, and a formal morphism between free objects is identified with the corresponding free morphism.
  \item Hence, $x_1$ must be a free object such that each of
    $\sgn^a(x_1)$ and $\sgn^b(x_1)$ consists of entries that are all less than four and not all even.
  \item Repeating the conclusions above, each morphism $\xi_i$ and object $x_i$ in \cref{eq:lift-xi} is free.
  \end{enumerate}
  This completes the proof.
\end{proof}

\bibliographystyle{sty/amsalpha2}
\bibliography{sty/Refs}%

\end{document}